\documentclass[reqno]{amsart}
\usepackage[a4paper]{geometry}
\usepackage[utf8]{inputenc} 
\usepackage[english]{babel}
\usepackage{amsmath,amsfonts,amssymb,amscd,bm, amsthm}
\usepackage{amsaddr}
\usepackage{latexsym}
\usepackage{graphicx}
\usepackage[margin=1cm]{caption}
\usepackage{subcaption}
\usepackage{tikz}
\usepackage{pgfplots}
\pgfplotsset{compat=1.16}

\usepackage{fancyhdr}
\usepackage{color}
\usepackage{listings}
\usepackage{pdfpages}
\usepackage{hyperref}
\usepackage{hhline}
\usepackage{courier}
\usepackage{mathrsfs}
\usepackage{wrapfig}
\usepackage{chngcntr}
\usepackage{multicol}
\usepackage{xcolor,colortbl}
\usepackage{bigfoot}
\usepackage{afterpage}
\usepackage{xurl}
\usepackage{tabularx,colortbl}
\usepackage{algorithm,algpseudocode}
\usepackage[percent]{overpic}

\usepackage[backend=bibtex, style=numeric-comp]{biblatex}
\appto{\bibsetup}{\sloppy}

\newtheorem{theorem}{Theorem}[section]

\newtheorem{proposition}[theorem]{Proposition}
\newtheorem{corollary}[theorem]{Corollary}
\newtheorem{conjecture}[theorem]{Conjecture}

\theoremstyle{definition}
\newtheorem{definition}[theorem]{Definition}

\theoremstyle{remark}

\newcommand{\dd}{\,\mathrm{d} }

\numberwithin{equation}{section}

\begin{document}

\title{Numerical Optimization of Eigenvalues of the planar magnetic Dirichlet Laplacian with constant magnetic field}

\author{Matthias Baur}
\address{Institute of Analysis, Dynamics and Modeling,\\
Department of Mathematics,\\
University of Stuttgart,\\
Pfaffenwaldring 57, 70569 Stuttgart, Germany}
\email{matthias.baur@mathematik.uni-stuttgart.de}

\begin{abstract}
We present numerical minimizers for the first seven eigenvalues of the planar magnetic Dirichlet Laplacian with constant magnetic field in a wide range of field strengths. Adapting an approach by Antunes and Freitas, we use gradient descent for the minimization procedure together with the Method of Fundamental solutions for eigenvalue computation. Remarkably, we observe that when the magnetic flux exceeds the index of the target eigenvalue, the minimizer is always a disk. 
\end{abstract}

\maketitle

\section{Introduction}

Let $\Omega \subset \mathbb{R}^2$ be a bounded, open set.  We consider the eigenvalue problem of the magnetic Dirichlet Laplacian with constant magnetic field on $\Omega$, namely
\begin{align} \label{eq:magn_lap_eigvalproblem}
(-i\nabla + A)^2 u &= \lambda u, \qquad \text{in } \Omega,\\
u&=0, \qquad \text{on } \partial \Omega,
\end{align}
where 
\begin{align}
A(x_1,x_2)= \frac{B}{2}\begin{pmatrix}-x_2 \\ x_1\end{pmatrix}, \qquad B\in\mathbb{R}, \nonumber
\end{align} 
is the standard linear vector potential. From a physical point of view, the magnetic Dirichlet Laplacian models a charged quantum-mechanical particle confined to $\Omega$ and interacting with a constant magnetic field of strength $B\in \mathbb{R}$ perpendicular to the plane. Its eigenvalues are the energy levels the particle is allowed to take on. 

For general geometries $\Omega$, the magnetic Dirichlet Laplacian $H_B^\Omega$ is more rigorously realized by Friedrichs extension of the quadratic form 
\begin{align}
a[u] := \int_\Omega |(-i\nabla +A) u |^2 \dd x, \qquad u \in C_0^\infty (\Omega).
\end{align}
The operator $H_B^\Omega$ defined this way is self-adjoint and exhibits purely discrete spectrum with positive eigenvalues of finite multiplicity, accumulating at infinity only. Counting multiplicities, we shall sort the infinite sequence of eigenvalues in increasing order and denote them by
\begin{align}
0 < \lambda_1(\Omega, B) \leq \lambda_2(\Omega, B) \leq ... \leq \lambda_n(\Omega, B)  \leq ...
\end{align}
If the boundary of $\Omega$ is sufficiently smooth, then the eigenvalues of $H_B^\Omega$ are exactly those gained from \eqref{eq:magn_lap_eigvalproblem}, otherwise \eqref{eq:magn_lap_eigvalproblem} has to be understood in an appropriate weak sense. For further details on magnetic Laplacians and Schrödinger operators, we refer to Avron, Herbst and Simon \cite{Avron1978}, Iwatsuka \cite{Iwatsuka1990} and Fournais and Helffer \cite{Fournais2010}.

The operator $H_0^\Omega$, i.e.\ when $B=0$, coincides with the standard Dirichlet Laplacian. In that case, the square roots of its eigenvalues can be interpreted as the fundamental frequencies of a vibrating membrane of the form $\Omega$. The problem of determining these fundamental frequencies dates back as far as the late 19th century and has received wide-spread attention of mathematicians and physicists ever since. Despite the lack of exact expressions for the eigenvalues of $H_0^\Omega$ (except for a handful of explicit domains) certain assertions about the structure of the sequence of eigenvalues can be made. The celebrated Weyl law \cite{Weyl1911, Weyl1912} states that for any domain $\Omega$
\begin{align}
\lambda_n(\Omega,0) \sim \frac{4\pi n}{|\Omega|}, \qquad n\to \infty. 
\end{align}
Later, P\'olya \cite{Polya1954} conjectured that the asymptotic expression in Weyl's law is actually a uniform lower bound, i.e.\
\begin{align} \label{eq:polya_conjecture}
\lambda_n(\Omega,0) \geq \frac{4\pi n}{|\Omega|},  
\end{align}
for any $n\in\mathbb{N}$ and was able to prove his conjecture for so-called tiling domains \cite{Polya1961}. Advances for other specific domain classes have been made since, e.g.\ \cite{Filonov2022a,Laptev1997}, but a proof of the conjecture in full generality remains elusive.

P\'olya's conjectured spectral inequality is known to hold in the cases $n=1$ and $n=2$ for any domain $\Omega$. This is a consequence of the celebrated Faber-Krahn \cite{Faber1923, Krahn1925} and Krahn-Szegö \cite{Krahn1926, Hong1954, Polya1955} theorems which date back to the 1920's. The former theorem states that among domains of equal measure, the disk minimizes $\lambda_1(\Omega,0)$. Before the proof by Faber and Krahn, this was already conjectured decades earlier by Lord Rayleigh in his book \textit{The Theory of Sound} (1894). The Krahn-Szegö theorem states an analogous optimality result for $\lambda_2(\Omega,0)$, namely that among domains of equal measure, $\lambda_2(\Omega,0)$ is minimized by two disjoint disks of equal size. By explicit computation of the eigenvalues of the minimizers, one can easily check that the minimizers and hence any domain $\Omega$ satisfy \eqref{eq:polya_conjecture} for $n=1$ and $n=2$.

For higher eigenvalues one can ask the corresponding spectral shape optimization problem
\begin{align} \label{eq:lambda_n_min_problem_intro_noB}
\min_{\substack{\Omega \text{ open}, \\ |\Omega|=c}} \lambda_n(\Omega, 0).
\end{align}
While the above minimization problem is deceptively easy to write down on paper, it turns out that it is outstandingly hard to solve analytically. So hard, in fact, that \eqref{eq:lambda_n_min_problem_intro_noB} has not been solved for any $n \geq 3$ over the past century. It is conjectured that the minimizer of \eqref{eq:lambda_n_min_problem_intro_noB} is a disk for $n=3$ and a disjoint union of two disks of different size for $n=4$, see \cite{Henrot2006}. At the same time, it has been shown that the disk is never a minimizer of \eqref{eq:lambda_n_min_problem_intro_noB} for any $n\geq 5$, see Berger \cite{Berger2014}. The last decades have seen a rise of interest in problem \eqref{eq:lambda_n_min_problem_intro_noB} and similar spectral shape optimization problems and there is now a wealth of literature on the topic. 

A surprisingly difficult task is already that of establishing the existence of minimizers. For problem \eqref{eq:lambda_n_min_problem_intro_noB}, this has been achieved through a series of papers in the 1990's-2010's \cite{Henrot2000, Buttazzo1993, Bucur2012, Mazzoleni2013, Henrot2018} by relaxing the class of open sets to so-called quasi-open sets. Parallel to these developments, several works have adressed questions of regularity, connectedness and convexity of minimizers, see e.g.\ \cite{Wolf1994, Mazzoleni2017,Mazzoleni2017a, Henrot2018}. For more references on \eqref{eq:lambda_n_min_problem_intro_noB} and related problems, we recommend the survey articles by Henrot \cite{Henrot2003, Henrot2018a} and Ashbaugh and Benguria \cite{Ashbaugh2007} as well as the books by P\'olya and Szegö \cite{Polya1951}, Henrot \cite{Henrot2006}, Henrot et al.\ \cite{Henrot2017}, and Henrot and Pierre \cite{Henrot2018}.

Over the past two decades, spectral shape optimization problems have received increased attention also from the numerical community. Especially in the last decade, the literature on numerical shape optimization for problems of spectral geometry has grown considerably. Numerical studies have revealed many surprising results, sometimes found counterexamples for already existing conjectures and often led to many new conjectures. Numerical investigations of problem \eqref{eq:lambda_n_min_problem_intro_noB} and its Neumann boundary condition counterpart have been carried out by Oudet \cite{Oudet2004} and Antunes and Freitas \cite{Antunes2012}, presenting numerical optimizers up to $n=15$. Surprisingly, they found that the numerically obtained minimizer for $n=13$ shows no obvious symmetries. Analogous numerical studies for three-dimensional and four-dimensional geometries have been conducted in the following \cite{Antunes2017, Antunes2017b}. Robin instead of Dirichlet boundary conditions have been considered in \cite{Antunes2013b, Antunes2016} and diameter instead of measure constraints were discussed in \cite{Bogosel2018}. Other spectral functionals have also received attention, for example sums of consecutive eigenvalues $\lambda_n + \lambda_{n+1}$ \cite{Antunes2013}, quotients $\lambda_n/\lambda_1$ \cite{Antunes2013, Osting2010}, spectral gaps $(\lambda_n - \lambda_{n-1})/\lambda_1$ \cite{Osting2010} and convex combinations of eigenvalues \cite{Osting2013,Osting2013a}. Of course, several numerical studies also focused on other operators and their eigenvalues, e.g.\ Steklov eigenvalues \cite{Antunes2021}, eigenvalues of the Bilaplacian \cite{Antunes2014}, Biharmonic Steklov eigenvalues \cite{SimaoAntunes2013}, $p$-Laplacian eigenvalues \cite{Antunes2019}, Steklov-Lam\'e eigenvalues \cite{Antunes2025} and Dirac eigenvalues \cite{Antunes2024}.

The present paper adapts the methods of Antunes and Freitas \cite{Antunes2012} to numerically investigate the magnetic version of problem \eqref{eq:lambda_n_min_problem_intro_noB}, i.e.\
\begin{align} \label{eq:lambda_n_min_problem_intro}
\min_{\substack{\Omega \text{ open}, \\ |\Omega|=c}} \lambda_n(\Omega, B),
\end{align}
where $B\in\mathbb{R}$. In particular, we are interested in the change of shape of the optimizers with respect to the parameter $B$, the strength of the magnetic field.

Literature on shape optimization for the magnetic Laplacian is relatively sparse compared to the standard Laplacian. Laugesen has considered shape optimization problems for magnetic Laplacians with Neumann boundary conditions, sums of eigenvalues and certain weighted domain normalizations \cite{Laugesen2011, Laugesen2015}. One of the main motivations for our work is the solution of the magnetic shape optimization problem \eqref{eq:lambda_n_min_problem_intro} for $n=1$ in 1996, when Erd\H{o}s proved a generalization of the Faber-Krahn theorem for constant magnetic fields.

\begin{theorem}[Erd\H{o}s \cite{Erdoes1996}] \label{thm:erdos_lambda1}
Let $\Omega \subset \mathbb{R}^2$ be a bounded, open domain. Then, for any $B\geq 0$,
\begin{align}
\lambda_1(\Omega, B) \geq \lambda_1(D, B)
\end{align}
where $D$ is a disk with $|D| = |\Omega|$.
\end{theorem}

A quantitative version of Erd\H{o}s' theorem has recently been published \cite{Ghanta2024}. As far as we are aware, there is no result or even conjecture for minimizers of \eqref{eq:lambda_n_min_problem_intro} for $n\geq 2$, not to mention that the question about existence of minimizers is unsettled. Our numerical results allow us to formulate new conjectures, in particular for a magnetic Krahn-Szegö theorem and the behaviour of the minimizers in case of strong fields. 

The paper is structured as follows: In Section \ref{sec:prel_prop}, we first recall various properties of the eigenvalues $\lambda_n(\Omega, B)$. We then present a result on the shape of disconnected minimizers. Section \ref{sec:num_meth} summarizes the numerical methods we apply to solve \eqref{eq:lambda_n_min_problem_intro}. In particular, for the eigenvalue computation, we adapt the so-called Method of Fundamental Solutions (MFS) to the magnetic Laplacian with constant magnetic field. Our numerical results are presented in Section \ref{sec:results}. We present two benchmark tests for the accuracy of the MFS and discuss numerical minimizers of \eqref{eq:lambda_n_min_problem_intro} with the constraint $|\Omega|=1$ for $n=1$ to $7$ and $B=1.0$ to $80.0$. In the appendix, we discuss the shape optimization problem \eqref{eq:lambda_n_min_problem_intro} among the restricted class of domains that consist of disjoint unions of disks.

\section{Preliminary propositions} \label{sec:prel_prop}

First, let us collect some basic properties of the eigenvalues $\lambda_n(\Omega, B)$.

\begin{proposition} \label{prop:simple_est_domain_mono} Let $\Omega \subset \mathbb{R}^2$ be a bounded, open domain. Then:
\begin{enumerate}
\item The eigenvalues $\lambda_n(\Omega,B)$ are invariant under rigid transformations of $\Omega$.
\item For any $B\geq 0$, 
\begin{align}
\lambda_1(\Omega,B) \geq \max\{ \lambda_1(\Omega,0) , B \}.
\end{align}
\item For any $n \in\mathbb{N}$, the map $B \mapsto \lambda_n(\Omega, B)$ is piecewise real-analytic.
\item If $\Omega'$ is a domain with $\Omega' \subset \Omega$, then 
\begin{align}
\lambda_n(\Omega',B) \geq \lambda_n(\Omega,B)
\end{align}
for any $n \in\mathbb{N}$.
\end{enumerate}
\end{proposition}

\begin{proof}
1. Consider an arbitrary rigid transformation. After reversing the rigid transformation by shifting and rotating the coordinate system, one is left with the original domain and a vector potential in a different gauge. The eigenvalues $\lambda_n(\Omega, B)$ are however invariant under gauge transformations of the vector potential with which $H_B^\Omega$ is realized. 
2. The estimate $\lambda_1(\Omega,B) \geq  \lambda_1(\Omega,0) $ is a consequence of the diamagnetic inequality, see \cite{Cycon1987, Avron1978, Fournais2010}. The second estimate follows from a commutator estimate \cite{Ekholm2016, Fournais2010, Avron1978}.
3. The operators $\{ H_B^\Omega \}$ form a type (B) self-adjoint holomorphic family. The spectrum of the magnetic Dirichlet Laplacian over $B$ can be described in terms of a countable set of real-analytic eigenvalue curves, see Kato \cite[Chapter VII \S 3 and \S 4]{Kato1995}. Sorting the eigenvalues by size yields that $B\mapsto \lambda_n(\Omega,B)$ is piecewise real-analytic with the pieces joined together continuously.
4. Follows from the variational principle.
\end{proof}

An asymptotically sharp spectral inequality for higher eigenvalues of the magnetic Laplacian with constant magnetic field is due to Erd\H{o}s, Loss and Vougalter.

\begin{theorem}[Erd\H{o}s-Loss-Vougalter \cite{Erdoes2000}] \label{thm:erdoes_li_yau}
Let $\Omega \subset \mathbb{R}^2$ be a bounded, open domain and $B\geq 0$. Then for any $n\in\mathbb{N}$,
\begin{align} \label{eq:ineq_erdoes_li_yau}
\sum\limits_{k=1}^n \lambda_k(\Omega, B) \geq \frac{2\pi n^2}{|\Omega|}. 
\end{align}
\end{theorem}

The proof by Erd\H{o}s, Loss and Vougalter actually allows for a slight improvement for finite $n$ \cite{Barseghyan2016}.

In non-magnetic case ($B=0$), the spectral inequality \eqref{eq:ineq_erdoes_li_yau} is usually known as the Li-Yau inequality \cite{ Li1983, Frank2022} and is equivalent to the equally well-known Berezin inequality \cite{Frank2022, Berezin1973} via Legendre transformation.

Theorem \ref{thm:erdoes_li_yau} immediately implies a P\'olya-type spectral inequality.

\begin{corollary} \label{cor:magnetic_polya}
Let $\Omega \subset \mathbb{R}^2$ be a bounded, open domain and $B\geq 0$. Then for any $n\in\mathbb{N}$,
\begin{align}
\lambda_n(\Omega, B) \geq \frac{2\pi n}{|\Omega|}.  \label{eq:magnetic_polya}
\end{align}
\end{corollary}

Note that there appears an excess factor of $1/2$ compared to the non-magnetic conjecture \eqref{eq:polya_conjecture}. Corollary \ref{cor:magnetic_polya} shows that any eigenvalue $\lambda_n(\Omega,B)$ has a non-zero lower bound among domains of equal measure. The sharpness of \eqref{eq:magnetic_polya} as $B\to \infty$ has been subject of papers by Frank \cite{Frank2009} and by the author and Weidl \cite{Baur2025}. The shape optimization problem \eqref{eq:lambda_n_min_problem_intro} naturally arises when one asks how much the estimate \eqref{eq:magnetic_polya} can actually be improved for fixed $n$ and $B$. 

Another important fact to recognize is that the size of the domain and the effect of the magnetic field are intimately connected. It is easy to verify that the following scaling identity holds for the eigenvalues $\lambda_n(\Omega,B)$.
\begin{proposition} \label{prop:scaling_identity}
For any $t>0$,
\begin{align}
\lambda_n\left(t\Omega, \frac{B}{t^2}\right) = \frac{\lambda_n(\Omega,B)}{t^2}.
\end{align}
\end{proposition}
Here, $t\Omega = \{tx \, : \, x \in \Omega \}$ is a homothetic dilation of $\Omega$. 
The scaling identity yields an important simplification of problem \eqref{eq:lambda_n_min_problem_intro}. Any minimizer of \eqref{eq:lambda_n_min_problem_intro} for a constraint $|\Omega| = c$ with $c>0$ at field strength $B$ can be gained by appropriately scaling a minimizer of \eqref{eq:lambda_n_min_problem_intro} for the constraint $|\Omega| = 1$ at field strength $cB$. We therefore  consider only the case of planar domains of normalized area, i.e.\ $|\Omega| = 1$ in the following. It is common to use the normalized magnetic flux $\phi$ through $\Omega$, defined by
\begin{align}
\phi = \frac{1}{2\pi} \int_\Omega B \dd x = \frac{|\Omega| B }{2\pi},
\end{align} 
as a scaling invariant quantity that determines the effect of the magnetic field.

We will not touch on existence of minimizers of \eqref{eq:lambda_n_min_problem_intro} in this article, but rather assume they exist as in the case of the non-magnetic Laplacian. Let us introduce the following notation for the minimum and minimizers of \eqref{eq:lambda_n_min_problem_intro}.

\begin{definition}
For $B \geq 0$, let
\begin{align} \label{eq:lambda_n_min_problem}
\lambda_n^*(B) :=\min_{\substack{\Omega \text{ open}, \\ |\Omega|=1}} \lambda_n(\Omega, B).
\end{align}
Furthermore, we denote (possibly non-unique) minimizers of this problem by $\Omega_n^*(B)$. 
\end{definition}

The map $B \mapsto \lambda_n^*(B)$ is continuous. This follows from the fact that the eigenvalues maps $B\mapsto \lambda_n(\Omega, B)$ are piecewise real-analytic for any domain $\Omega$, in particular for any minimizer $\Omega_n^*(B_0)$ with fixed $B_0$. We also have the following elementary estimate.

\begin{proposition} Let $n \in\mathbb{N}$, $B\geq 0$ and suppose $\Omega_n^*(B)$ is star-shaped. Then, for any  $0 < t \leq 1$,
\begin{align}
 \lambda_n^*(B) \geq t \lambda_n^*(t B ). 
\end{align}
Moreover, if the left-sided derivative 
\begin{align}
(\lambda_n^*)_{-}'(B) = \lim_{h \rightarrow 0+ } \frac{\lambda_n^*(B) -\lambda_n^*(B -  h )}{h}
\end{align}
exists, it satisfies
\begin{align}
(\lambda_n^*)_{-}'(B) \geq - \frac{\lambda_n^*(B)}{B} .
\end{align}
\end{proposition}

\begin{proof}
We have for any $0 < t \leq 1$ that $\Omega_n^*(B)  \subset t^{-\frac{1}{2} } \Omega_n^*(B)$ after a suitable translation and thus by domain monotonicity, see Proposition \ref{prop:simple_est_domain_mono}, and scaling, see Proposition \ref{prop:scaling_identity},
\begin{align}
\lambda_n^*(B) = \lambda_n(\Omega_n^*(B), B) \geq \lambda_n(t^{-\frac{1}{2}}\Omega_n^*(B), B) = t \lambda_n\left(\Omega_n^*(B), tB\right) \geq t \lambda_n^*\left(t B \right). 
\end{align}
Let now $0 < h< B$ and set $t = 1- h/B$, then 
\begin{align}
\lambda_n^*(B) \geq \left(1-\frac{h}{B}\right) \lambda_n^*\left(B -  h \right)
\end{align}
and thus
\begin{align}
\frac{\lambda_n^*(B) -\lambda_n^*(B -  h )}{h} \geq -\frac{\lambda_n^*(B - h )}{B} .
\end{align}
Assuming the existence of a limit of the left hand side as $h$ is taken to zero and using continuity of $\lambda_n^*(B)$ for the right hand side yields the second assertion.
\end{proof}

An important property of the minimizers that we want to discuss is connectedness. Assuming existence of minimizers of \eqref{eq:lambda_n_min_problem_intro} for any $n\in\mathbb{N}$ and any $B\geq 0$, we can show the following generalization of a theorem by Wolf and Keller \cite{Wolf1994} that originally concerns the non-magnetic case.

\begin{theorem} \label{thm:wolf-keller}
Suppose that \eqref{eq:lambda_n_min_problem} is minimized by the union of two disjoint domains with positive measure. Then
\begin{align}
\lambda_n^*(B) = \min_{\substack{1\leq i \leq \lfloor n/2 \rfloor, \\ t\in(0,1)}} \max \left\lbrace \frac{\lambda_i^*(t B)}{t} , \frac{ \lambda_{n-i}^*((1-t) B)}{1-t} \right\rbrace. \label{eq:wolf-keller}
\end{align} 
If $(i^*,t^*)$ denotes a solution pair of above min-max-problem, then the corresponding minimizers are given by
\begin{align}
\Omega_n^*(B) = (t^*)^\frac{1}{2} \Omega_{i^*}^*(t^*B) \, \cup \,(1-t^*)^\frac{1}{2} \Omega_{n-i^*}^*((1-t^*)B).
\end{align}
where $\Omega_{i^*}^*(t^*B)$ and $\Omega_{n-i^*}^*((1-t^*)B)$ are any minimizers of \eqref{eq:lambda_n_min_problem} to the indices $i^*$ and $n-i^*$ with scaled field strength.
\end{theorem}

\begin{proof}
First, by the scaling property of the magnetic Laplacian with constant magnetic field, it holds
\begin{align}
\lambda_n(\Omega,B) = \frac{1}{|\Omega|} \lambda_n\left(\frac{\Omega}{|\Omega|^\frac{1}{2}}, |\Omega| B\right) .
\end{align}
Let $\Omega_n^*(B) = \Omega_1 \cup \Omega_2$ with $\Omega_1$, $\Omega_2$ disjoint, $|\Omega_1|>0$, $|\Omega_2|>0$ and $|\Omega_1|+|\Omega_2|=1$. Let $t:=|\Omega_1| \in (0,1)$, so that $|\Omega_2| = 1-t$. Without loss of generality we may assume that the eigenfunction to the $n$-th eigenvalue on $\Omega_n^*(B)$ is not vanishing on the component $\Omega_1$. Then there must exist a minimal index $i \leq n$ such that $$\lambda_n^*(B) = \lambda_n(\Omega_n^*(B),B) = \lambda_i(\Omega_1,B).$$
If we assume $i=n$, then $\lambda_1(\Omega_2,B) \geq \lambda_n^*(B)$ and domain monotonicity implies $\lambda_n(\Omega_1,B) = \lambda_n^*(B)$ can be lowered by slightly expanding $\Omega_1$ while shrinking $\Omega_2$. This clearly contradicts the assumption of $\Omega_n^*(B)$ being a minimizer. Therefore $i \leq n-1$.

Next, we notice that $t^{-\frac{1}{2}} \Omega_1  $ has measure equal to one, hence 
\begin{align}
\lambda_n^*(B)= \lambda_i(\Omega_1,B) = \frac{1}{t} \lambda_i\left(\frac{\Omega_1}{ t^\frac{1}{2}}, t B\right) \geq \frac{\lambda_i^*(t B)}{t}
\end{align}
and in order to not contradict that $\Omega_n^*(B)$ is a minimizer, we must have equality and $\Omega_1 = t^\frac{1}{2} \Omega_i^*(tB)$.

On the other hand, the other component of $\Omega_n^*(B)$ must have at least $n-i$ eigenvalues below or equal to $\lambda_n^*(B)$. As in case of $\Omega_1$, we notice 
\begin{align}
\lambda_{n-i}(\Omega_2,B) \geq \frac{\lambda_{n-i}^*((1-t) B)}{1-t} 
\end{align}
and we must have equality to avoid another contradiction with minimality of $\Omega_n^*(B)$ and $\Omega_2 = (1-t)^\frac{1}{2} \Omega_{n-i}^*((1-t)B)$. We have found that the minimal eigenvalue must satisfy
\begin{align} 
&\lambda_n^*(B) =   \max \left\lbrace \frac{\lambda_i^*(t B)}{t} , \frac{ \lambda_{n-i}^*((1-t) B)}{1-t} \right\rbrace.
\end{align}
for some $1\leq i \leq n-1$ and $t \in (0,1)$. But as $\lambda_n^*(B)$ is minimal, we conclude
\begin{align} \label{eq:disjoint_disk_n-1_pair}
&\lambda_n^*(B) =  \min_{\substack{1\leq i \leq n-1, \\ t\in(0,1)}} \max \left\lbrace \frac{\lambda_i^*(t B)}{t} , \frac{ \lambda_{n-i}^*((1-t) B)}{1-t} \right\rbrace.
\end{align}
Finally, we remark that one can replace the condition $1\leq i \leq n-1$ in \eqref{eq:disjoint_disk_n-1_pair} by $1\leq i \leq \lfloor n/2 \rfloor$ due to symmetry. 
\end{proof}

For $n=2$, this theorem leads to the following assertion.

\begin{corollary} \label{cor:lambda2_connectedness}
For $B > 4\pi$, any minimizer $\Omega_2^*(B)$ is connected. If $\Omega_2^*(B)$ is a disconnected minimizer for $B \leq 4\pi$, then it must be a disjoint union of two equally sized disks.
\end{corollary}

\begin{proof}
Let us assume $\Omega_2^*(B)$ is a union of two disjoint domains with positive measure. Then, according to the previous theorem, we find that $\Omega_2^*(B)$ must consist of scaled minimizers of $\lambda_1^*(t B)$ and $\lambda_1^*((1-t) B)$ and
\begin{align}
\lambda_2^*(B)=\min_{t\in(0,1)} \max \left\lbrace \frac{\lambda_1^*(t B)}{t} , \frac{ \lambda_{1}^*((1-t) B)}{1-t} \right\rbrace .
\end{align}
By Theorem \ref{thm:erdos_lambda1}, we have 
\begin{align}
\frac{\lambda_1^*(t B)}{t} = \frac{\lambda_1(D,t B)}{t} = B \cdot \frac{\lambda_1(D,t B)}{t B}
\end{align}
where $D$ is a disk with $|D| = 1$. Since $B \mapsto \lambda_1(D, B) /B $ is strictly decreasing with $B$, see e.g.\ \cite{Baur2025, Son2014}, it follows that 
\begin{align}
 \max \left\lbrace \frac{\lambda_1^*(t B)}{t} , \frac{ \lambda_{1}^*((1-t) B)}{1-t} \right\rbrace
\end{align}
is minimal at $t^*=1/2$. We get $ \lambda_2^*(B) = 2 \lambda_1^*(B/2)$ and $\Omega_2^*(B)$ is a disjoint union of two disks of equal size. For $B> 4\pi$ however, we have
\begin{align}
\lambda_2(D,B)<  2\lambda_1(D,B/2)=2 \lambda_1^*(B/2)=\lambda_2^*(B),
\end{align}
which contradicts that $\Omega_2^*(B)$ is a minimizer. Therefore, the minimizer $\Omega_2^*(B)$ must be connected for $B> 4\pi$.
\end{proof}

\section{Numerical Methods} \label{sec:num_meth}

In this section we describe the numerical methods we employ for general geometries. As already mentioned, we employ the Method of Fundamental Solutions to compute eigenvalues of the magnetic Dirichlet Laplacian and the minimization algorithm to solve the shape optimization problem is gradient-based.

\subsection{The Method of Fundamental Solutions} \label{subsec:MFS}

There is an extensive literature on the Method of Fundamental Solutions, see for example \cite{Karageorghis2001, Alves2005, Betcke2005, Antunes2011} and references therein.

The general idea of the MFS is the following. Let $\Omega\subset \mathbb{R}^d$ be a bounded, open domain and consider the eigenvalue problem
\begin{align}\label{eq:L_pde_equation}
L u = \lambda u \qquad  \text{ in } \Omega
\end{align}
for a self-adjoint, (classical) linear differential operator $L$, supplied with Dirichlet boundary condition
\begin{align}\label{eq:L_bnd_equation}
u = 0 \qquad  \text{ on } \partial \Omega.
\end{align}
The operator $L=L(x)$ may depend on $x\in \Omega$. We construct solutions of the eigenvalue equation \eqref{eq:L_pde_equation} by linear combinations of translated Green's functions on the whole space. Hence, let $G_\lambda$ denote the Green's function of the operator $L-\lambda$ on $\mathbb{R}^d$, i.e.\ $G_\lambda$ satisfies
\begin{align} \label{eq:greens_function_equation}
(L-\lambda)G_\lambda (x,y) = \delta(x-y) \qquad \text{ in } \mathbb{R}^d
\end{align}
in the distributional sense. If $L$ is independent of $x\in\Omega$, such as the standard Laplacian, Green's functions on the whole space are called fundamental solutions. In this case, the fundamental solution $G_\lambda(x,y)$ only depends on $\Vert x-y \Vert$.

Let $\{ y_j\}_{j=1}^m$ be a set of points outside of $\overline{\Omega}$. We call the points $y_j$ source points. Observe that any function of the form
\begin{align} \label{eq:eigenfunction_ansatz}
\hat{u}_\alpha(x)=\sum\limits_{j=1}^m \alpha_j G_\lambda(x,y_j),
\end{align}
where $\alpha = (\alpha_1, ..., \alpha_m )^T$ is a collection of arbitrary complex-valued coefficients, is then a solution of the eigenvalue equation \eqref{eq:L_pde_equation}. While any such function does satisfy \eqref{eq:L_pde_equation}, it does not necessarily satisfy the Dirichlet boundary condition \eqref{eq:L_bnd_equation}. We try to approximate $\hat{u} = 0$ on $\partial \Omega$ for any $\lambda$ by imposing the boundary condition on a suitable set of collocation points, i.e.\ by demanding $\hat{u}(x_i) = 0$ for a set of points $\lbrace x_i \rbrace$ chosen from the boundary $\partial\Omega$. This leads to the linear system of equations
\begin{align}
A(\lambda) \cdot \alpha = 0 \label{eq:Alambda_null_eq}
\end{align}
where $A(\lambda)$ is the matrix with entries $(A(\lambda))_{ij} = G_\lambda(x_i,y_j)$. One can now search for $\lambda$ where \eqref{eq:Alambda_null_eq} has a non-trivial solution $\alpha$. If the number of collocation points exceeds the number of source points, the matrix $A(\lambda)$ is rectangular and equation \eqref{eq:Alambda_null_eq} usually does not have a non-trivial solution, but one can instead look for normalized solutions $\alpha$ that minimize the left hand side. Looking for $\lambda$ where the left hand side can be made close the zero, singles out certain values $\hat{\lambda}_n$ and to each of these a correspoding coefficient vector $\hat{\alpha}_n$. For each of the pairs $(\hat{\lambda}_n, \hat{\alpha}_n)$ the eigenvalue equation \eqref{eq:L_pde_equation} is satisfied exactly and the boundary condition is almost fulfilled. The values $\hat{\lambda}_n$ are considered approximations for the eigenvalues of the operator $L$ and the functions $\hat{u}_n := \hat{u}_{\hat{\alpha}_n}$ their associated approximate eigenfunctions.

In the literature, the model operator usually discussed in view of the MFS is the standard Laplacian $L = - \Delta$. It is translationally invariant and positive. In $\mathbb{R}^2$, the fundamental solution $G_\lambda$ of $-\Delta - \lambda$ is given by
\begin{align}
G_\lambda(x,y) = \frac{i}{4}H_0^{(1)}(\sqrt{\lambda} \Vert x-y \Vert) , \qquad x,y \in\mathbb{R}^2,
\end{align}
where $H_0^{(1)}$ is a Hankel function of the first kind. 

For the magnetic Laplacian $L = (-i\nabla + A)^2$ in $\mathbb{R}^2$ with constant magnetic field, we now replace this by the corresponding Green's function $G_\lambda$ that fulfills
\begin{align} \label{eq:eigenfrequency_equation_magnetic}
((-i\nabla + A(x))^2 - \lambda) G_\lambda(x,y)&= \delta(x-y) \qquad \text{ in } \mathbb{R}^2,
\end{align}
where $A(x) = B/2 (-x_2,x_1)^T $.
The Green's function is given by
\begin{align} \label{eq:greens_function_magnetic}
G_\lambda(x,y) = \frac{\Gamma( \frac{1}{2}-\frac{\lambda}{2B} )}{4 \pi } \frac{ \exp\left({i \frac{B}{2}(x_2 y_1 - x_1 y_2)}\right)}{\left(\frac{B}{2} \Vert x-y  \Vert^2 \right)^{1/2} } \,  W_{\frac{\lambda}{2B},0}\left( \frac{B}{2} \Vert x-y  \Vert^2 \right),  \qquad x,y \in\mathbb{R}^2,
\end{align}
see Dodonov et al.~\cite{dodonov1975}. Here, $\Gamma(z)$ is the gamma function and $W_{\kappa, \mu}(z)$ denotes the Whittaker function. For references on the Whittaker function, we refer to \cite{Abramowitz1972,Buchholz1953, NIST_DLMF}. Note that $G_\lambda(x,y)$ is not defined if $\frac{1}{2}-\frac{\lambda}{2B} $ is a negative integer, i.e.\ if $\lambda = (2k+1) B$ for some $k \in \mathbb{N}$. This is a manifestation of the so-called Landau levels of the magnetic Laplacian with constant magnetic field on the full-space $\mathbb{R}^2$. It is well-known that the magnetic Laplacian on $\mathbb{R}^2$ exhibits infinitely degenerate eigenvalues at the Landau levels, see e.g.\ \cite{Avron1978, Iwatsuka1990, Landau1991}, so that a Green's function for these $\lambda$ cannot be defined. \\

\textbf{Subspace Angle Technique} \\

There are multiple techniques to determine the values $\hat{\lambda}_n$ for which \eqref{eq:Alambda_null_eq} can be approximately solved. If the number of source and collocation points is equal and hence $A(\lambda)$ is a square matrix, a common approach is to look for local minima of the map $\lambda \mapsto \log(|\det A(\lambda)|)$. However, we observed that this method tends to produce spurious pseudo-eigenvalues and -eigenfunctions for the magnetic Laplacian. We will comment on this again later. A more robust method is the so-called Subspace Angle Technique (SAT) that has been proposed by Betcke and Trefethen \cite{Betcke2005} and has been applied for example in \cite{Osting2010, Antunes2017b}. This method slightly modifies the standard MFS scheme that was outlined above.

The SAT works as follows. First, choose $m_B \geq m$ collocation points on the boundary and then additionally $m_I$ random points inside the domain. Let $m' = m_B + m_I$ be the number of boundary and interior points. Then define the $m' \times m$-matrix
\begin{align}
A(\lambda) = \begin{pmatrix}
A_B(\lambda)  \\ 
A_I(\lambda) 
\end{pmatrix} 
\end{align}
with $A_B(\lambda)$ indicating the subblock of Green's function evaluations with collocation points on the boundary and $A_I(\lambda)$ indicating the subblock of Green's function evaluations with interior points. By reduced QR factorization one can decompose $A(\lambda)$ into $A(\lambda) = Q(\lambda) R(\lambda)$ where $Q(\lambda)$ is a complex-valued $m' \times m$ matrix and $R(\lambda)$ is a complex-valued upper triangular $m \times m$ matrix. Every vector $b \in \text{ran } A(\lambda) \subset \mathbb{C}^{m'}$ can then be written as 
\begin{align}
b = Q(\lambda) c = \begin{pmatrix}
Q_B(\lambda)  \\ 
Q_I(\lambda) 
\end{pmatrix} c
\end{align}
for some  $c \in\mathbb{R}^m$ and with $ \Vert c \Vert = \Vert b \Vert$. Here, $Q_B(\lambda)$ is the upper $m_B \times m$ subblock and $Q_I(\lambda)$ is the lower $m_I \times m$ subblock of $Q(\lambda)$. The idea is to minimize the part of $b$ that corresponds to the boundary under the condition that $c$ is non-trivial. Therefore, consider the minimization problem
\begin{align}
\underset{\substack{c \in\mathbb{R}^m,\\ \Vert c \Vert = 1}}{\min} \, \Vert Q_B(\lambda) c\Vert. \label{eq:Q_B_min}
\end{align}
This minimization problem is solved by the right singular vector $\hat{c} \in\mathbb{C}^{m}$ to the smallest singular value of $ Q_B(\lambda)$ and problem \eqref{eq:Q_B_min} yields precisely the smallest singular value of $ Q_B(\lambda)$. If we let
\begin{align}
\sigma(\lambda) := \underset{k = 1,..., m}{\min} \sigma_k (Q_B(\lambda))= \Vert Q_B(\lambda) \hat{c}\Vert = \underset{\substack{c \in\mathbb{R}^m,\\ \Vert c \Vert = 1}}{\min} \, \Vert Q_B(\lambda) c\Vert,
\end{align}
then $\sigma(\lambda)$ is a measure of how good the boundary condition can be fulfilled in all collocation points. The eigenvalue approximations $\hat{\lambda}_n$ are chosen as the values of $\lambda$ where $\sigma(\lambda)$ takes on a local minimum close to zero. For a given eigenvalue approximation $\hat{\lambda}_n$ the coefficient vector for the corresponding approximate eigenfunction $\hat{u}_n = \hat{u}_{\hat{\alpha}_n}$ is then gained from 
\begin{align}
\hat{\alpha}_n = R(\lambda)^{-1} \hat{c}_n,
\end{align}
where $\hat{c}_n$ denotes the minimizer of \eqref{eq:Q_B_min} at $\lambda = \hat{\lambda}_n$.

Note that if we set $\hat{b}_n = Q(\lambda) \hat{c}_n$, then we have
\begin{align}
\Vert A_I(\hat{\lambda}_n)  \hat{\alpha}_n\Vert ^2 = \Vert Q_I(\hat{\lambda}_n) \hat{c}_n\Vert ^2 = \Vert \hat{b}_n\Vert^2 - \Vert Q_B(\hat{\lambda}_n) \hat{c}_n\Vert^2  = 1 - \sigma(\hat{\lambda}_n)^2. 
\end{align}
This means that if the boundary condition is well satisfied in the collocation points and $\sigma(\lambda)$ is small, then $\Vert A_I(\lambda) \hat{\alpha}_n\Vert$ is close to one. We conclude that the approximative eigenfunction function will be non-zero in at least one interior point. This means that the SAT forces solutions $\hat{u}_n$ to be non-zero somewhere in the interior of the domain. We found this property to be particularly important for the magnetic Laplacian. Note that for bounded, open $\Omega$ there exist eigenfunctions not just in $\Omega$ but also in $\mathbb{R}^2 \setminus \overline{\Omega}$, the unbounded exterior space. During tests with the MFS scheme we observed that the simple approach of minimization of $\lambda \mapsto \log (|\det A(\lambda)|)$ appears to fail to distinguish between interior and exterior eigenfunctions. It yields a union of the desired solutions of the magnetic Dirichlet Laplacian on the interior $\Omega$ and undesired pseudo-solutions on the exterior $\mathbb{R}^2 \setminus \overline{\Omega}$. Generally, the second type of solutions are functions $\hat{u}_{\hat{\alpha}_n}$ that are very close to zero in the interior $\Omega$ and non-vanishing on the exterior $\mathbb{R}^2 \setminus \overline{\Omega}$. These pseudo-solutions satisfy the eigenvalue equation in $\mathbb{R}^2 \setminus \overline{\Omega}$ except in the finite number of source points near the boundary $\partial \Omega$ and we consider them to be artifacts introduced by the $\log (|\det A(\lambda)|)$-method. The SAT eliminates these artifical exterior solutions and is hence preferred by us in the case of the magnetic Laplacian.\\

%

\textbf{Further comments on parameter choices and implementation} \\

The placement of the collocation points $\lbrace x_i \rbrace$ and the source points $\lbrace y_j \rbrace$ is crucial for the success of the MFS. Following previous studies \cite{Alves2005,Antunes2012,Antunes2017b}, we assume that $\Omega$ is bounded by a closed, simple curve and choose to place the boundary collocation points $\lbrace x_i \rbrace$ equidistant along the boundary curve. The source points $\lbrace y_j \rbrace$ are then set to $y_i = x_i + \delta n_i $, where $n_i$ is the outer normal vector at $x_i$ and $\delta>0$ is some small fixed parameter. The interior points are sampled uniformly from the interior of the domain. Typically we used around $300$ collocation points on the boundary, the same amount of source points and $50$ to $100$ interior points. For domains with normalized area, we found that a good choice for the parameter $\delta$ is around $0.015$ to $0.03$.

Our code is written in Python. The Whittaker function $W_{\kappa, \mu}$ is rarely implemented in numerical libraries of special functions. However, the  Whittaker function is related to Tricomi's confluent hypergeometric function $U$ via
\begin{align}
W_{\kappa, \mu}(z) = \exp\left(-\frac{z}{2}\right) z^{\mu + \frac{1}{2}} U(\mu -\kappa + \frac{1}{2},1+ 2\mu,z),
\end{align}
see for example \cite{bateman1953} for reference, hence, yielding for the Green's function
\begin{align} \label{eq:greens_function_magnetic_kummer}
G_\lambda(x,y) = \frac{\Gamma( \frac{1}{2}-\frac{\lambda}{2B})}{4 \pi } \exp\left({i \frac{B}{2}(x_2 y_1 - x_1 y_2) - \frac{B}{4} \Vert x-y  \Vert^2 }\right) \, U\left(\frac{1}{2}-\frac{\lambda}{2B},1,\frac{B}{2} \Vert x-y  \Vert^2 \right).
\end{align}
We noticed that the implementation of Tricomi's $U$ function in SciPy (version 1.7.3), called \texttt{hyperu}, yields erroneous outputs for certain ranges of inputs. An excellent alternative is the PyGSL package (\url{https://github.com/pygsl/pygsl}) which is a Python wrapper for GSL - GNU Scientific Library (\url{https://www.gnu.org/software/gsl/}).

Finally, $G_\lambda(x,y)$ may be multiplied by any scalar $c(\lambda) \neq 0$ independent of $x$ and $y$  as a scalar multiplication does not change the solvability of the system $A(\lambda) \cdot \alpha = 0$, it can however have a positive effect on numerical stability. We have decided to simply disregard the factor $\Gamma( \frac{1}{2}-\frac{\lambda}{2B})/(4 \pi )$ entirely and instead work with the scaled Green's function
\begin{align}
\tilde{G}_\lambda(x,y) = \exp\left({i \frac{B}{2}(x_2 y_1 - x_1 y_2) - \frac{B}{4} \Vert x-y  \Vert^2 }\right) U\left(\frac{1}{2}-\frac{\lambda}{2B},1,\frac{B}{2} \Vert x-y  \Vert^2 \right).
\end{align}

\newpage

\subsection{Optimization routine}

We describe the basic principles of the minimization algorithm.\\

 \textbf{Domain parametrization} \\

For the numerical optimization scheme, we restrict ourselves to the following class of domains. We assume $\Omega$ is simply connected and has a boundary that is described by a closed Jordan curve $\gamma(\varphi)$ of the form
\begin{align}\label{eq:star_shaped_set2d}
\gamma(\varphi):=  \begin{pmatrix} \gamma_x( \varphi ) \\  \gamma_y( \varphi ) \end{pmatrix} ,  \qquad \varphi \in [0,2\pi).
\end{align}
We further assume $\gamma_x$ and $\gamma_y$ to be given by truncated Fourier series, i.e. 
\begin{align}
\gamma_x(\varphi) &= a_{x,0} + \sum\limits_{j=1}^{N-1} a_{x,j} \cos(j\varphi) + \sum\limits_{j=1}^{N-1}  b_{x,j} \sin(j \varphi), \\
\gamma_y(\varphi) &= a_{y,0} + \sum\limits_{j=1}^{N-1} a_{y,j} \cos(j\varphi) + \sum\limits_{j=1}^{N-1}  b_{y,j} \sin(j \varphi), 
\end{align}
where $N$ is some fixed integer and $\{ a_{x,j}\}$, $\{ b_{x,j}\}$, $\{ a_{y,j}\}$ and $\{ b_{y,j}\}$ are suitable real-valued coefficients. With $c_x$, $c_y$ and $f(\varphi)$ defined by
\begin{align}
c_x &= (a_{x,0},..., a_{x,N-1}, b_{x,0},..., b_{x,N-1})^T, \\
c_y &= (a_{y,0},..., a_{y,N-1}, b_{y,0},..., b_{y,N-1})^T, \\
f(\varphi) &= (1, \cos(\varphi), ..., \cos((N-1) \varphi), 0 , \sin(\varphi), ..., \sin((N-1) \varphi))^T,
\end{align}
we can abbreviate this to
\begin{align}
\gamma_x(\varphi) &= c_x^T \cdot f(\varphi), \\
\gamma_y(\varphi) &= c_y^T \cdot f(\varphi),
\end{align}
with the standard euclidian dot product in $\mathbb{R}^{2N}$. Note that $\gamma_x$ and $\gamma_y$ are linear in $c_x$ resp.~$c_y$.

For any domain $\Omega$ that is described by a boundary curve $\gamma$ going around the domain anti-clockwise, the area is given by the integral
\begin{align} \label{eq:measure_2d_integral}
|\Omega| = \frac{1}{2}\int_{\partial \Omega } x \dd y - y \dd x = \frac{1}{2}\int_{0}^{2 \pi} \gamma_x(t) \gamma_y'(t) - \gamma_y(t) \gamma_x'(t) \dd t.
\end{align}
The outer normal $n(\varphi)$ to $\Omega$ at any given boundary point $\gamma(\varphi)$ can be computed with 
\begin{align}
n(\varphi) = \frac{T(\gamma'(\varphi))}{\Vert T(\gamma'(\varphi))\Vert} , \qquad \text{where} \qquad 
T\left(\begin{pmatrix}  x  \\   y \end{pmatrix} \right)=\begin{pmatrix} y\\ -x \end{pmatrix}.
\end{align}
One can also decide whether a point lies inside or outside $\Omega$ by computing an integral. By the Jordan curve theorem, a closed Jordan curve divides $\mathbb{R}^2$ into an interior region and an exterior region. A point $(x_0 , y_0)$ that does not lie on the boundary of $\Omega$ is either inside the domain or outside the domain. If the winding number 
\begin{align}
\mathrm{wind}(x_0,y_0; \gamma) &= \frac{1}{2\pi} \int_{\partial \Omega } \frac{(x-x_0) \dd y - (y - y_0) \dd x}{(x-x_0)^2 + (y - y_0)^2}  \\
&= \frac{1}{2\pi} \int_{0}^{2 \pi} \frac{(\gamma_x(t) - x_0) \gamma_y'(t) - (\gamma_y(t) - y_0) \gamma_x'(t)}{(\gamma_x(t)-x_0)^2 + (\gamma_y(t) - y_0)^2}  \dd t
\end{align}
of $\gamma$ around $(x_0 , y_0)$ is non-zero, then it lies inside the domain, otherwise it lies outside the domain.\\

\textbf{Gradient descent scheme} \\

By the scaling property of the eigenvalues, the shape optimization problem \eqref{eq:lambda_n_min_problem} is equivalent to the shape optimization problem
\begin{align} \label{eq:lambda_n_min_problem_scaled_no_constraint}
\min_{\Omega \text{ open}} |\Omega| \lambda_n\left(\Omega, \frac{B}{|\Omega|} \right) 
\end{align}
in the sense that any minimizer of \eqref{eq:lambda_n_min_problem} is a minimizer of \eqref{eq:lambda_n_min_problem_scaled_no_constraint} and any minimizer of \eqref{eq:lambda_n_min_problem_scaled_no_constraint} scaled to unit area is a minimizer of \eqref{eq:lambda_n_min_problem}. Both problems give the same result up to scaling of the minimizer. We focus on problem \eqref{eq:lambda_n_min_problem_scaled_no_constraint} in the following, since it has the advantage of being an unconstrained minimization problem.

The domain parametrization introduced previously now reduces the shape optimization problem to an optimization problem over coefficient vectors $c\in\mathbb{R}^{2N}$. Formally, we introduced a map $c \mapsto \Omega_c$ that yields a bounded, open domain $\Omega_c$ for every admissible coefficient vector $c$. Here, we call a coefficient vector $c$ is admissible if the associated boundary curve is a Jordan curve. On this restricted class of domains, the problem we want to solve becomes
\begin{align}
\min_{\substack{c\in\mathbb{R}^{2N}\\ \text{ admissible}} } J(c), \qquad \text{where } J(c):=|\Omega_c|\lambda_n\left(\Omega_c, \frac{B}{|\Omega_c|} \right).
\end{align}
We solve the now finite dimensional problem with a modified gradient descent algorithm. The algorithm we employ proceeds as follows: First we define an initial domain by choosing an initial vector of coefficients $c_0$ and compute the objective eigenvalue. Then we compute the gradient of $J$ with respect to $c$ at our current domain and optimize our objective $J$ as usual along the half-line spanned by the gradient. In the standard gradient descent scheme, the minimizer found along the half-line is the starting point of the next iteration. But since our spectral objective function is independent of $|\Omega|$, we choose to normalize the coefficient vector $c$ in every iteration such that $|\Omega_c|=1$. Because of \eqref{eq:measure_2d_integral}, this is simply achieved by replacing $c$ by $c/|\Omega_c|^\frac{1}{2}$.  This prevents shrinking or blow-up of domains over the course of several iterations. The gradient descent loop continues until an exit condition is fulfilled, e.g.\ $|J(c_i)-J(c_{i-1})| < \varepsilon$ where $\varepsilon$ is small or a fixed number of iterations $i_{\mathrm{max}}$ is exceeded. The full algorithm is summarized in pseudo-code in Algorithm 1.

{\centering
\begin{minipage}{.7\linewidth}

\begin{algorithm}[H]
\caption{Gradient descent algorithm}
\begin{algorithmic}
\State set $i=0$
\While{$i\leq i_{\mathrm{max}}$ or $|J(c_i)-J(c_{i-1})| \geq \varepsilon$}
	\State compute $J(c_i)$
	\State compute $d = \nabla_c J(c_i)$
	\State set $c(\beta) = \dfrac{c_i - \beta d}{|\Omega_{c_i - \beta  d}|^\frac{1}{2}}$ 
	\State compute $\beta^*= \mathrm{argmin}_{\beta\in [0, \beta_{\mathrm{max}}]} \, J(c(\beta))$
	\State set $c_{i+1} = c(\beta^*)$
	\State increase $i \rightarrow i+1$
\EndWhile
\State \Return $J(c_i), c_i$
\end{algorithmic}
\end{algorithm}
\end{minipage}
\par\vspace{0.5cm}
}

To apply this algorithm, it is necessary to be able to determine the gradient $\nabla_c J(c)$ at given $c$. By chain rule, the gradient is composed of the three derivatives
\begin{align} \label{eq:gradient_chain_rule_components}
\nabla_c (|\Omega_c|), \qquad \nabla_c \lambda_n\left( \Omega_c, B\right), \qquad \frac{\partial}{\partial B}  \lambda_n\left( \Omega_c, B\right).
\end{align}

If $\lambda_n(\Omega_c, B)$ is simple, each of these can be evaluated purely in terms of the eigenvalue $\lambda_n(\Omega_c, B)$ and its associated eigenfunction $u_n$. Expressions for the first two derivatives of \eqref{eq:gradient_chain_rule_components} are derived within the framework of shape derivatives. \\

\textbf{Shape derivatives} \\

Let $\Phi: (-T, T) \rightarrow W^{1, \infty}(\mathbb{R}^d,\mathbb{R}^d)$
such that $\Phi(t)$ is differentiable with $\Phi(0) = I$ and $\Phi'(0) = V$ where $V$ is a vector field. Here, $W^{1, \infty}(\mathbb{R}^d,\mathbb{R}^d)$ denotes the set of all bounded Lipschitz maps from $\mathbb{R}^d$ to $\mathbb{R}^d$. The vector field $V$ is also called a deformation field. For simplicity, we assume that $\Phi$ has the form $\Phi(t) = I + tV$. Applied to a domain $\Omega$, the applications $\Phi(t)$ generate the family of domains
\begin{align}
\Omega_t = \Phi(t)(\Omega) = \lbrace \Phi(t)(x)\; :\; x\in \Omega \rbrace, \qquad t\in (-T,T).
\end{align}
The parameter $t$ controls the strength of the deformation of $\Omega$. 

Under a given application $\Phi$, the rate of change of the domain's measure is fully characterized by a boundary integral.
\begin{theorem}[\cite{Henrot2006}] \label{thm:change_volume}
Let $\Omega$ be a bounded, open set. If $\Omega$ is Lipschitz, then $t \mapsto |\Omega_t|$ is differentiable at $t=0$ with
\begin{align}
(|\Omega_t|)'(0) = \int_{\partial \Omega}  V \cdot n \, d\sigma .
\end{align}
\end{theorem}

Let us assume now that $\Omega \subset \mathbb{R}^2$ is a bounded, open set with $\partial \Omega$ of class $C^2$ and $A: \mathbb{R}^2 \rightarrow \mathbb{R}^2$ some real-valued, divergence-free vector potential of class $C^1$. This includes the case of linear gauge $A(x)=B/2(-x_2,x_1)$. We denote by $\lambda_n(t) := \lambda_n(\Omega_t)$ the $n$-th eigenvalue of the magnetic Laplacian $(-i\nabla + A)^2$ on $\Omega_t$ with Dirichlet boundary condition. The change of eigenvalues with respect domain deformations can also be computed by boundary integrals. This leads to so-called Hadamard formulas which are well-known for the standard Laplacian. Hadamard's formula generalizes to the case of magnetic fields. In fact the formula looks exactly the same as in the case of the Dirichlet Laplacian.

\begin{theorem}[Hadamard formula for magnetic Dirichlet Laplacian]\label{thm:change_eigval}
If $\lambda_n(\Omega)$ is a simple eigenvalue with corresponding normalized eigenfunction $u_n \in L^2(\Omega)$, then $\lambda_n(t)$ is differentiable at $t=0$ with
$$\lambda_n'(0) = -\int_{\partial \Omega} \left| \frac{\partial u_n}{\partial n}\right|^2 V \cdot n \, d\sigma .$$
Here, $\partial / \partial n$ denotes the normal derivative with respect to the outer normal of the domain. 
\end{theorem}

We have not yet specified how the derivatives with respect to the coefficients of the domain parametrization relate to the shape derivatives in the language of deformation fields. The partial derivative of the $n$-th rescaled eigenvalue with respect to the $i$-th entry of the coefficient vector $c$ is the same as the directional derivative of the $n$-th rescaled eigenvalue in direction $e_i$ in coefficient space. We must now find a deformation field $V_i$ that corresponds to incremental change in direction $e_i$ in coefficient space. \\
For our class of parametrized domains, we observe that changing the first coefficient vector from $c_x$ to $c_x + te_i$ results in a new set with boundary defined by $\gamma_{x,t}(\varphi) = (c_x + te_i)^T \cdot f(\varphi)$ and $\gamma_{y}(\varphi) = c_y^T \cdot f(\varphi)$. This means that the boundary of $\Omega$ transforms from 
\begin{align}
\partial \Omega := \{ (c_x^T \cdot f(\varphi) ,c_y^T \cdot f(\varphi))^T \in \mathbb{R}^2 \, : \,   \varphi \in [0,2\pi)  \}
\end{align}
to 
\begin{align}
\partial \Omega_t := \{ ((c_x+ te_i)^T \cdot f(\varphi) ,c_y^T \cdot f(\varphi))^T \in \mathbb{R}^2 \, : \,   \varphi \in [0,2\pi)  \}.
\end{align}
The shape derivative yields the desired derivative in direction $e_i$ in coefficient space, if we now choose a deformation field $V_{x,i}$ such that $(I+tV_{x,i})(\partial \Omega) = \partial \Omega_t$. A suitable deformation field is
\begin{align}
V_{x,i}(\gamma(\varphi)) =  e_i^T\cdot f(\varphi) \begin{pmatrix} 1 \\ 0 \end{pmatrix} .
\end{align} 
Similarly, for the derivative with respect to the $i$-th coefficient of the second coefficient vector $c_y$, we can choose
\begin{align}
V_{y,i}(\gamma(\varphi)) =  e_i^T\cdot f(\varphi) \begin{pmatrix} 0 \\ 1 \end{pmatrix} .
\end{align} 
With $V_i$ given, the $i$-th partial derivative of our spectral objective function $J(c)$ is then for domains with $|\Omega_c| = 1$ given by
\begin{align} 
\begin{split}
\frac{\partial}{\partial c_i} J(c)&=\frac{\partial}{\partial c_i}\left(|\Omega_c| \lambda_n\left(\Omega_c, \frac{B}{|\Omega_c|} \right) \right)\\
&=\int_{\partial \Omega_c} \left[  \left(  \lambda_n(\Omega_c, B )- B \cdot \frac{\partial}{\partial B} \lambda_n(\Omega_c, B )\right)-  \left|\frac{\partial u_n}{\partial n} \right|^2 \right] V_i \cdot n \, d\sigma .
\end{split}
\end{align}

Finally, we give an expression for the last needed derivative, the derivative of $\lambda_n(\Omega_c,B)$ with respect to the field strength $B$. Under enough regularity of the boundary of the domain, the desired derivative can be obtained by differentiating the magnetic Laplacian. We obtain a Hellmann-Feynman-type formula.

\begin{theorem}[Hellmann-Feynman] \label{thm:Hellmann-Feynman}
Let $\Omega \subset \mathbb{R}^2$ be a bounded, open set with $\partial \Omega$ of class $C^2$, $\hat{A}: \mathbb{R}^2 \rightarrow \mathbb{R}^2$ a sufficiently regular vector field. Denote the eigenvalues of $H(B)=(-i\nabla +B\hat{A})^2$ on $\Omega \subset \mathbb{R}^2$ with Dirichlet boundary condition by $\lambda_n(\Omega,B)$. If $\lambda_n(\Omega,B)$ is a simple eigenvalue with associated $L^2$-normalized eigenfunction $u_n$, then there exists an open neighborhood around $B$ where $\lambda_n(\Omega,B)$ is (infinitely) differentiable with respect to $B$ and where
\begin{align}
\frac{\partial}{\partial B} \lambda_n (\Omega, B) = \left\langle u_n , \frac{\partial H(B)}{\partial B} u_n \right\rangle_{L^2(\Omega)}.
\end{align}

\end{theorem}

We omit a fully rigorous proof here. It can be obtained for example by the implicit function theorem. Note that if $\hat{A}$ is divergence-free, then the Hellmann-Feynman formula gives
\begin{align} \label{eq:lambda_n_partialB}
\frac{\partial}{\partial B} \lambda_n (\Omega, B)& = \frac{1}{B} \left( \lambda_n(\Omega, B) +  \int_{ \Omega} B^2 \hat{A}^2 |u_n|^2  - |\nabla u_n|^2 \dd x \right), & \text{if } B \neq 0,\\
\frac{\partial}{\partial B} \lambda_n (\Omega, B)& = 0, & \text{if } B = 0,
\end{align} 
which follows from integration by parts. In the case of the constant magnetic field, we take the vector potential $\hat{A}(x)=1/2(-x_2,x_1)^T$ which has vanishing divergence.

With the preceding formulas, the gradient $\nabla_c J(c)$ can be fully evaluated using the approximative eigenpairs $(\hat{\lambda}_n, \hat{u}_n)$ from the MFS. We remark that the normal derivative 
\begin{align}
\frac{\partial \hat{u}_n}{\partial n} = (\nabla \hat{u}_n)^T \cdot n.
\end{align}
of the MFS approximative eigenfunction can be evaluated exactly on the boundary since
\begin{align}
\nabla_x \hat{u}_{\alpha_n^*}(x) =  \sum\limits_{j=1}^m (\alpha_n^*)_j \nabla_x G_\lambda(x,y_j)
\end{align}
where the gradient $\nabla_x G_\lambda(x,y_j)$ for the Green's function of the Magnetic Laplacian in 2D is 
\begin{align}
\nabla_x G_\lambda(x,y) &= \frac{\Gamma( \frac{1}{2}-\frac{\lambda}{2B})}{4 \pi } \exp\left({i \frac{B}{2}(x_2 y_1 - x_1 y_2) - \frac{B}{4} \Vert x-y  \Vert^2 }\right) \\
& \hspace{1cm} \cdot \Bigg[ \left( i \frac{B}{2} y^\perp - \frac{B}{2} (x-y) \right) U\left(\frac{1}{2}-\frac{\lambda}{2B},1,\frac{B}{2} \Vert x-y  \Vert^2 \right) \\
& \hspace{2.5cm} - \left( \frac{1}{2}-\frac{\lambda}{2B} \right) B(x-y) \; U\left(\frac{3}{2}-\frac{\lambda}{2B},2,\frac{B}{2} \Vert x-y  \Vert^2 \right) \Bigg].
\end{align}
Here, we used the notation $y^\perp = (-y_2 ,y_1)^T$ as well as \eqref{eq:greens_function_magnetic_kummer} for the Green's function and
\begin{align}
\frac{d}{dz} U(a,b,z) = -a \cdot U(a+1,b+1,z)
\end{align}
for the derivative of the $U$ function \cite{bateman1953}. \\

\textbf{Eigenvalue clustering} \\

As noted in \cite{Antunes2017b}, it is desirable to adjust the descent direction when the target eigenvalue comes close to lower eigenvalues. The reason for such eigenvalue clustering is that any potential decrease of the target eigenvalue is limited by the distance to the next lower eigenvalue and although following the negative gradient of the target eigenvalue will decrease it for a sufficiently small step width, it does not necessarily decrease the next lower eigenvalue at the same time. In case of eigenvalue clustering during the optimization procedure, choosing the gradient $d_n=\nabla_c J(c_i)$ as the descent direction usually leads to stagnation before it can reach a more preferable region in coefficient space. 

To mitigate this problem, it is beneficial to choose a different search direction, one that decreases the whole cluster of eigenvalues simultaneously. With given parameter $\varepsilon_{\lambda}>0$, we say $\lambda_k$ belongs to the cluster of $\lambda_n$ if $|\lambda_k-\lambda_n|< \varepsilon_{\lambda}$ or $|\lambda_k-\lambda_{k+1}|< \varepsilon_{\lambda}$ and $\lambda_{k+1}$ is part of the cluster. Under this definition, $\lambda_n$ is always part of its own cluster. If we detect a cluster of $j\geq 2$ eigenvalues, let
\begin{align}
d_k = \nabla_c \left(|\Omega_c|^\frac{2}{d} \lambda_k\left(\Omega_c, \frac{B}{|\Omega_c|^\frac{2}{d}} \right) \right) \Bigg|_{c=c_i}
\end{align}
denote the gradient of the $k$-th rescaled eigenvalue. Then new search direction is determined by solving the optimization problem
\begin{align} \label{eq:adjust_d_search_direction}
\max_{\substack{d \in \mathbb{R}^{2 N}, \\ \Vert d\Vert = 1}} \min \lbrace (d_n)^T \cdot d, (d_{n-1})^T \cdot d , ... , (d_{n-j+1})^T \cdot d   \rbrace.
\end{align}
The minimizer of this problem is a search direction that has good alignment with all individual gradients of the rescaled eigenvalues in the cluster in the sense that the projections are simultaneously as large as possible. It is easily shown that any minimizer of \eqref{eq:adjust_d_search_direction} is found in $\operatorname{span}\{d_{n-j+1}, ..., d_n \}$, so \eqref{eq:adjust_d_search_direction} reduces to a $j$-dimensional problem that can be solved quickly.

\section{Results} \label{sec:results}

\subsection{MFS Benchmark}

To illustrate the potential accuracy of the MFS, we apply it in a first benchmark test to the disk. We computed the first six eigenvalues of an area-normalized disk at the magnetic field strengths $B=6.0, 30.0$ and $100.0$. 

Note that eigenvalues of the disk can be obtained from the solutions to a set of implicit equations, see \eqref{eq:magnetic_disk_eigval_transcendental}. These equations can be solved very accurately by a simple root search algorithm. Hence, this root search method will serve as our control experiment. 

For the MFS, we used $m_B = 300$ collocation points on the boundary of the disk and an equal number of source points. We selected $m_I = 100$ random interior points and a distance of $\delta = 0.08$ between collocation and source points.

\begin{table}
\begin{center}
\begin{tabular}{c|ccc}
$B=6.0$ & root search & MFS & FreeFEM \\ \hline
 $\lambda_1$ & 18.78985322628746 & 18.7898532262874\textbf{2} & 18.789\textbf{90462255829}                                            \\
$\lambda_2$ & 41.07719366592705 & 41.0771936659270\textbf{6} & 41.077\textbf{32594796938}  \\
$\lambda_3$ & 53.07719366592705 & 53.07719366592705 & 53.077\textbf{32594796850}  \\
$\lambda_4$ & 72.02862974019817 & 72.02862974019\textbf{772} & 72.028\textbf{86895347926}  \\
$\lambda_5$ & 96.02862974019801 & 96.02862974019\textbf{767} & 96.028\textbf{86895347715}  \\
$\lambda_6$ & 96.62148527797662 & 96.6214852779766\textbf{5} & 96.621\textbf{76298112768}  \\ \hline\hline

$B=30.0$ & root search & MFS & FreeFEM  \\ \hline

$\lambda_1$ & 31.97451772480147 & 31.974517724801\textbf{51} & 31.9745\textbf{3981161004}  \\
$\lambda_2$ & 38.47621009105776  &  38.4762100910577\textbf{9} & 38.4762\textbf{8828825477} \\
$\lambda_3$ & 51.01243934805061 & 51.012439348050\textbf{54} & 51.012\textbf{60625635323} \\
$\lambda_4$ & 70.26913184739855  & 70.2691318473985\textbf{8} & 70.269\textbf{41669425181} \\
$\lambda_5$ & 96.54045269537635 &  96.5404526953\textbf{4237} & 96.540\textbf{88225478540} \\
$\lambda_6$ & 98.47621009105780 & 98.476210091057\textbf{78} & 98.4762\textbf{8828825187} \\\hline\hline

$B=100.0$ & root search & MFS & FreeFEM \\ \hline

$\lambda_1$ & 100.00036335077738 & 100.000363350777\textbf{05} & 100.0003633\textbf{9324728} \\
$\lambda_2$ & 100.00535529860144  &  100.00535529860\textbf{097} & 100.005355\textbf{62088823}\\
$\lambda_3$ & 100.03914939250473  & 100.039149392504\textbf{38} & 100.0391\textbf{5110568723}\\
$\lambda_4$ & 100.18903860798198  & 100.1890386079819\textbf{3} & 100.1890\textbf{4550379986}\\
$\lambda_5$ & 100.67837722417839 & 100.67837722417\textbf{798}  & 100.6783\textbf{9873530745} \\
$\lambda_6$ & 101.93536559801957 & 101.935365598019\textbf{23} & 101.935\textbf{41960030071} \\

\end{tabular}
\end{center}
 \caption{Comparison of numerical methods for the solution of the eigenvalue problem on a disk $D$ with $|D|=1$. Computed are the eigenvalues $\lambda_1(D,B)$ to $\lambda_6(D,B)$ at three different field strengths. Digits that deviate from those found by root search are printed in bold.}\label{tab:eigval_methods_comparison} 
\end{table}

\begin{table}

\begin{center}

\begin{tabular}{c|cccc}
 & $\lambda_1$ & $\lambda_2$ & $\lambda_3$ & $\lambda_4$  \\ \hline
 \begin{tabular}{c} FreeFEM (788 vertices)\end{tabular} & 7.84152214 & 12.10122972 & 18.66639944 & 24.12429257 \\
\begin{tabular}{c} FreeFEM (3092 vertices)\end{tabular} & 7.84117784 & 12.09950210 & 18.65981318 & 24.12232805 \\
\begin{tabular}{c} FreeFEM  (12167 vertices)\end{tabular} & 7.84111243 & 12.09914378 & 18.65838429 & 24.12193757  \\
\begin{tabular}{c} FreeFEM (48080 vertices)\end{tabular}  & 7.84109824 & 12.09906161 & 18.65804734 & 24.12184950  \\
\begin{tabular}{c} FreeFEM  (191521 vertices)\end{tabular}  & 7.84109492 & 12.09904176 & 18.65796477 & 24.12182836  \\
\begin{tabular}{c} MFS  ($m_B = 300$, $\delta=0.1$) \end{tabular} & 7.84109385 & 12.09903527 & 18.65793752 & 24.12182146  \\
\end{tabular}

\end{center}
 \caption{Comparison of eigenvalue approximations for $\lambda_1(\Omega_{test},B)$ to $\lambda_4(\Omega_{test},B)$ at $B=6.0$. The FE eigenvalue approximations appear to converge to the MFS eigenvalue approximations when the mesh size is increased.}\label{tab:eigval_methods_comparison2} 
\end{table}

We also computed the eigenvalues by Finite Elements (FE). The FE computations were done in FreeFEM v4.9 (\url{https://freefem.org/}, Aug 26 2024). We used a piecewise $P_2$ continuous finite element space and a discretisation of the disk with $196363$ vertices. The system dimension was $783949$. 

All computations (root search, MFS as well as FE) were performed on an 8-core Intel i7-6700 CPU with 16 GB of RAM. We remark that the system size of the FE experiment almost saturates the memory limit of the machine. With the given settings and on the given machine, MFS and FE computations both took about one minute of wall time whereas root search took only a fraction of a second.

The resulting eigenvalue approximations of the three methods are found in Table \ref{tab:eigval_methods_comparison}. It can be seen that the eigenvalues obtained by MFS agree with those found by root search up to $12$ to $14$ digits, sometimes even down to $16$ digits, which is essentially machine precision. In comparison, significantly less accuracy is achieved by FE. Here, typically only five to seven digits agree with those found by root search. So, at comparable computation times, MFS is capable of giving a lot more accurate digits for the disk.

In a second benchmark test, we have computed the lowest four eigenvalues of the magnetic Dirichlet Laplacian on an arbitrarily chosen non-circular test domain $\Omega_{test}$. We chose this domain as the interior of the closed Jordan curve
\begin{align*}
\gamma(\varphi) = \begin{pmatrix}
\cos(\varphi) +0.05 \cdot \cos(2 \varphi) +0.1 \cdot \sin(4 \varphi)\\
\sin(\varphi ) - 0.05 \cdot  \cos(6 \varphi) - 0.05 \cdot \sin(4\varphi)
\end{pmatrix}, \qquad \varphi\in[0,2\pi),
\end{align*}
which yields the following shape:
\begin{center}
\begin{overpic}[trim={4cm 3cm 4cm 2.5cm},clip,scale=0.35]{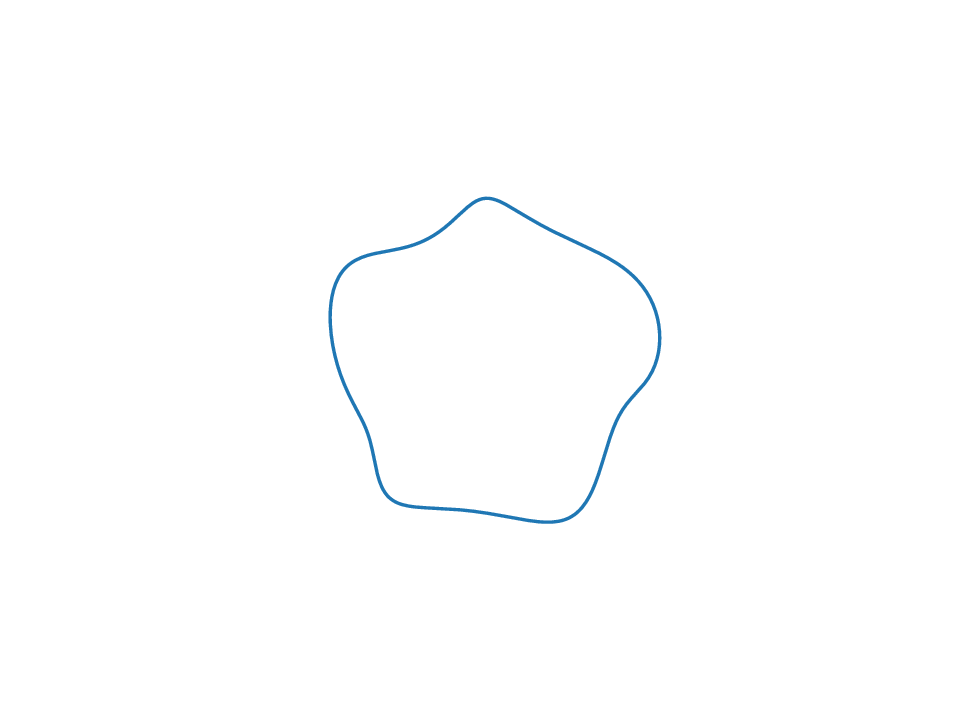}
 \put (100,25) {$\Omega_{test}$}
\end{overpic}
\end{center}

Table \ref{tab:eigval_methods_comparison2} shows the eigenvalues for the domain $\Omega_{test}$ obtained by FE for different mesh sizes compared to the eigenvalues obtained by MFS. The eigenvalue approximations are consistent and it appears that the eigenvalues obtained by FE slowly converge to the eigenvalues obtained by MFS. For the largest mesh size, the relative errors between FE and MFS are within $10^{-6}$.

We admit that we cannot guarantee that MFS is always more accurate than FE for general domains. The performance of both numerical methods naturally depends on parameter choices (number and location of collocation and source points for MFS; mesh and FE space for FE). We expect however that similar as in the case of the non-magnetic, standard Laplacian, eigenvalue accuracy of MFS is superior for domains sufficiently close to disks and domains whose boundary exhibits no points of extreme curvature. Problematic domains (for both MFS and FE) are domains with sharp corners or cusps. Such domains usually require much more collocation points and a small $\delta$ parameter resp. a very fine mesh to get decent accuracy.  

Shape optimization requires solving the eigenvalue problem for domains of different shape over and over again, so we want a robust, but also time-efficient eigenvalue solver. Convergence of the gradient descent method we described occurs already with much less accuracy than the $14$ digits of accuracy the MFS can potentially provide. We employ MFS with around $300$ collocation points and $\delta=0.015$ to $0.03$. We found this to be a robust parameter choice for the MFS that allows computation of eigenvalues for a wide class of domains with decently strong curvatures of the boundary curve and believe that it yield at least four to five correct significant digits for the domains we considered.

\subsection{Numerical minimizers of low eigenvalues}

In this section, we present numerical minimizers of problem \eqref{eq:lambda_n_min_problem} generated with the gradient descent shape optimization method. Computed were the minimizers for eigenvalues $\lambda_1(\Omega, B)$ to $\lambda_7(\Omega, B)$ at field strengths ranging from $B=1.0$ to $B=80.0$. For $B=0.0$, the minimizers were taken from \cite{Antunes2012}.

Initial domains were selected from a database of around $700$ domains with randomly chosen coefficients for the domain parametrization. For all domains in the database, we computed the lowest few eigenvalues for the set of field strengths $B=2.0, 4.0, 6.0, ..., 80.0$. Next, we sorted the domains for each combination of field strength $B=2.0, 4.0, 6.0, ..., 80.0$ and eigenvalue index $n=1,2,...,7$ according to the target eigenvalue $\lambda_n(\Omega, B)$. This step lead to $40 \cdot 7 = 280$ sorted lists of domains.

In the next step, we selected between five and ten domains from the top $10 \%$ of each list, i.e.\ the $10 \%$ of the list with lowest targest eigenvalue $\lambda_n(\Omega, B)$, as initial domains for the gradient descent method. The converged, final domains from these initial domains were then compared among each other and the domain with minimal $\lambda_n(\Omega, B)$ was declared numerically optimal. Selecting multiple initial domains is crucial to guarantee exploration of the domain parameter space to a sufficient degree and reduce the chances that a local but non-global minimum is found by the gradient descent method. Indeed, we found that multiple local minimizers appear quite frequently for higher eigenvalues, so starting from different initial domains is important when aiming to find a global minimizer. 

Up to this point, we have run the minimization algorithm for field strength $B=2.0, 4.0, 6.0, ..., 80.0$ and eigenvalue index $n=1,2,...,7$, thus expecting $280$ minimizers in total. However, as we will discuss later, the gradient descent method does not converge for all combinations of $B$ and $n$ due to formation of cusps as the iterations progress, so we did not obtain a minimizer for all combinations of $n$ and $B$.

For eigenvalues $\lambda_1$ and $\lambda_2$, we present the results obtained from the procedure described above. For eigenvalues $\lambda_3$ to $\lambda_7$, we continued with a second pass to refine the resolution in field strength $B$. We aimed to compute minimizers at a ten-fold finer resolution than in the first pass of generation of minimizers, i.e.\ for $B=1.0, 1.2, 1.4, ...$\ . Initial domains for these second pass computations were now chosen from the final minimizers from the first pass at field strengths $B=2.0, 4.0, 6.0, ..., 80.0$. We always selected the minimizers of the two closest field strengths from the first pass. If, for example, we aimed to minimize $\lambda_3$ at $B=5.2$, we chose the numerical minimizers of $\lambda_3$ at $B=4.0$ and $B=6.0$ from the first pass as initial domains for the gradient descent method. Both resulting domains were compared and the best of the two chosen. Selection of the initial domains from the minimizers of the first pass is motivated by the intuition that the minimizer should not change too much for small change in field strength $B$. We observed that this intuition appears to work for local minimizers. The gradient descent method converges quite quickly to a local minimizer if one starts with a local minimizer of a close field strength. However, since for the global minimizers we take the minimum over all local minimizers, minimizers can instanteneously switch shape at two subsequent field strengths when two branches of local minimizers intersect. In fact, we observe this phenomenon a few times. We believe that the resolution of the first pass is fine enough that restarting the minimization scheme from the two closest minimizers available from the first pass does not miss any potentially even lower lying branches of local minimizers. \\

\textbf{Numerical minimizers for $\lambda_1$} \\

According to Theorem \ref{thm:erdos_lambda1}, the unique minimizer of $\lambda_1$ among domains of equal measure is always given by a single disk. This is exactly what we observe numerically. For any field strength $B$ we have checked, the gradient descent scheme converges toward a domain that by eye looks indistinguishable from a disk. Likewise, the numerically found values for the minimal first eigenvalue are very close to the corresponding values of the first eigenvalue of a disk, see Figure \ref{fig:compcurve_lambda1}.\\

\begin{figure}[p]
\includegraphics[scale=0.5]{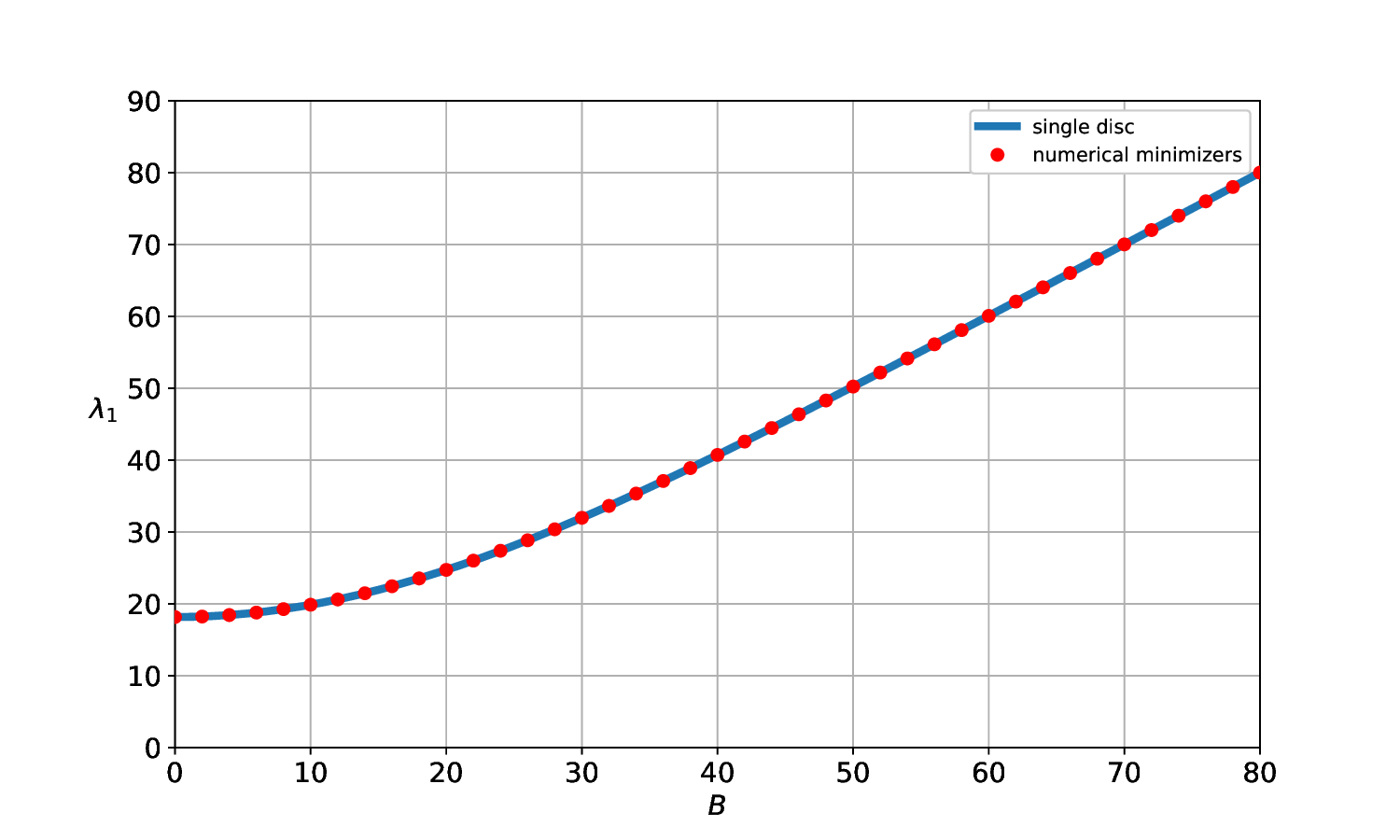}
\caption{Comparison of $\lambda_1$ of the numerical minimizers (red) with $\lambda_1$ of a single disk (blue) over different field strengths $B$. Numerical values have very good agreement with $\lambda_1$ of a single disk.} \label{fig:compcurve_lambda1}
\end{figure}

\begin{figure}
\includegraphics[trim={2cm 2cm 2cm 2cm},clip,scale=0.5]{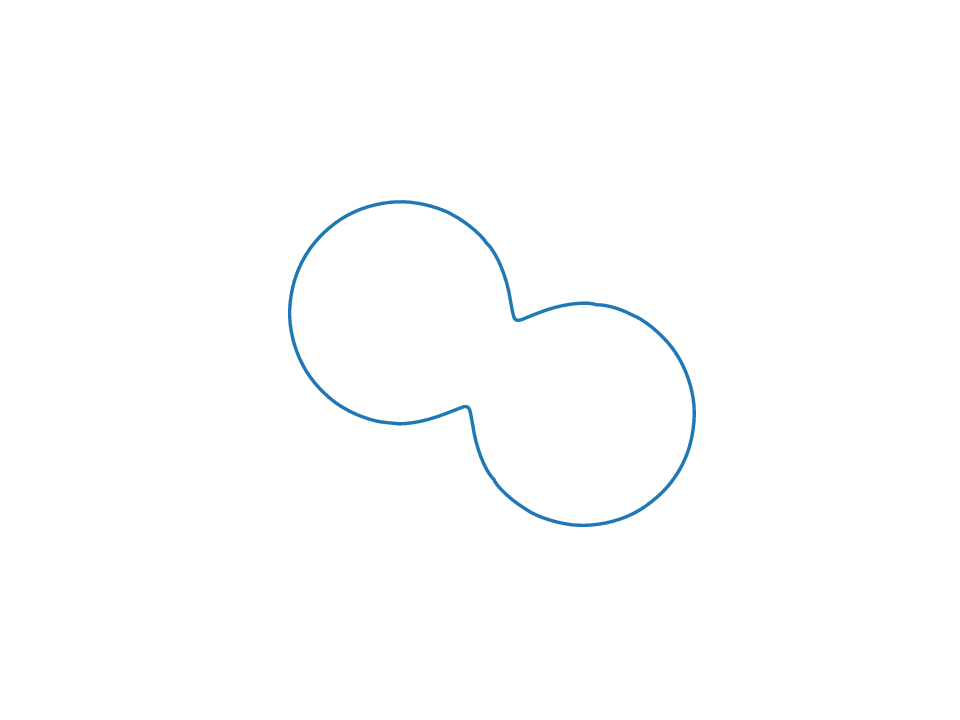}
\caption{Final admissible domain from minimization of $\lambda_2(\Omega, B=10.0)$.} \label{fig:lambda2_final_domain}
\end{figure}

\begin{figure} [p]
\includegraphics[scale=0.5]{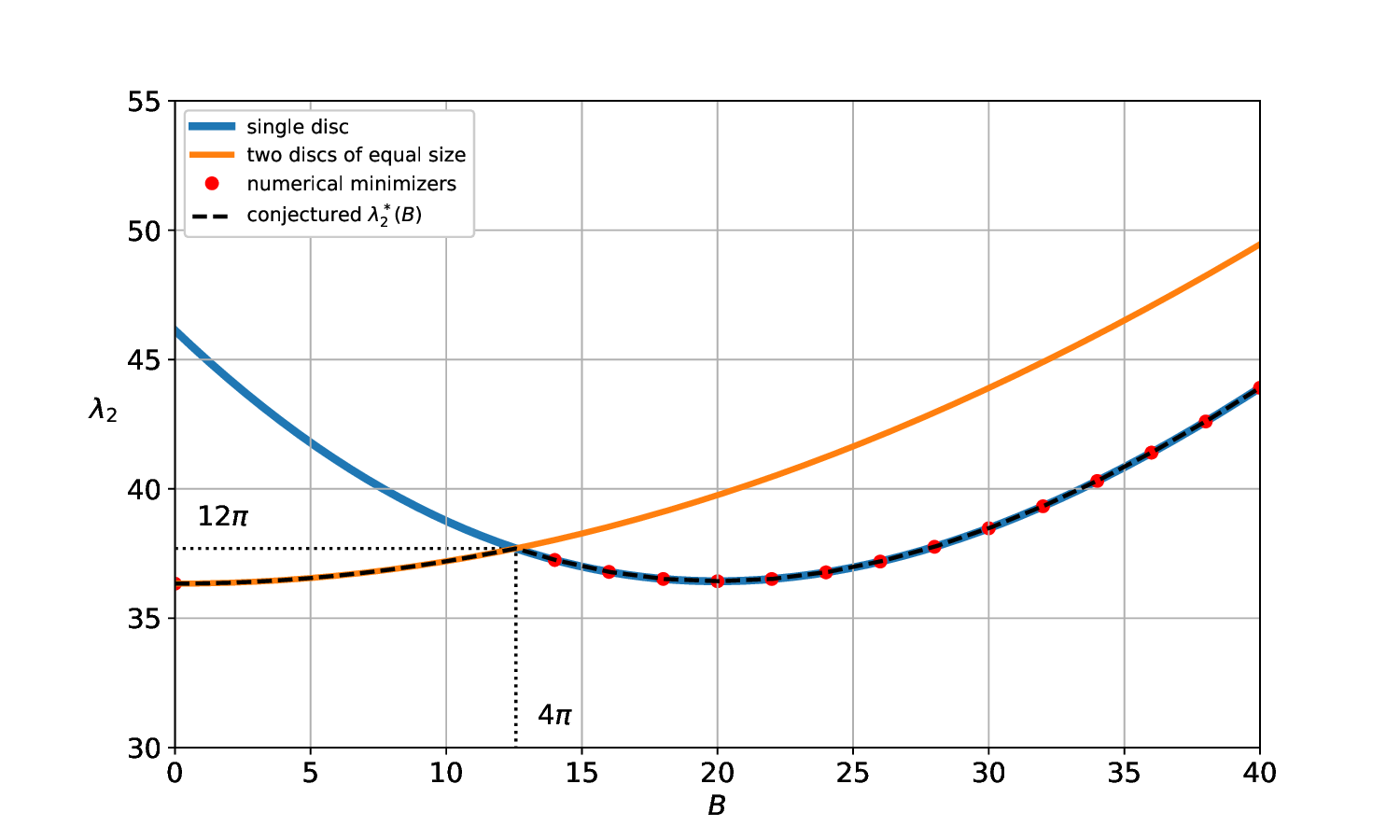}
\caption{Comparison of $\lambda_2$ of the numerical minimizers (red) with $\lambda_2$ of a single disk (blue) and two disjoint disks of equal size (orange) over different field strengths $B$. } \label{fig:compcurve_lambda2}
\end{figure}

\textbf{Numerical minimizers for $\lambda_2$} \\

For the second eigenvalue of the magnetic Dirichlet Laplacian, we observe that the gradient descent scheme does not converge for small $B$ (experimentally, for $B \leq 12.0$). Instead, the minimization scheme produces domains with two opposing cusps that become sharper as the minimization scheme progresses. At some point, the gradient descent scheme breaks down because one iteration yields coefficients for a boundary curve that is self-intersecting. The optimizaton scheme steps outside the region of admissible domains and must be stopped. Figure \ref{fig:lambda2_final_domain} shows the last admissible domain for $B=10.0$ before failure of the minimization scheme. 

We interpret this behaviour as an indication that the minimization procedure tries to separate the domain into two disjoint components but is hindered by the nature of our domain parametrization which only allows simply connected domains. In this regard, we recall Corollary \ref{cor:lambda2_connectedness} which states that if the optimal domain for $\lambda_2$ is the union of two disjoint components and if $B\leq 4\pi \approx 12.566$, then the optimal disconnected domain is given by two disks of equal size. Comparing the numerically attained value of $\lambda_2(\Omega, B)$ of the last admissible domains from the optimization scheme with the value of $\lambda_2$ for two equally sized disjoint disks shows that the latter is always smaller. So for $B \leq 12.0$, we have not found any domain $\Omega$ with $\lambda_2(\Omega, B)$ smaller than for two equally sized disjoint disks. We find this to be very convincing evidence that the minimizer of $\lambda_2(\Omega, B)$ should be two disjoint disks for $B\leq 4\pi$.

For $B \geq 14.0$, we find very different behaviour. The numerical minimization always yields a single disk. This might be surprising at first glance, but Corollary \ref{cor:lambda2_connectedness} states that any minimizer for $B> 4\pi \approx 12.566$ must be connected and we should not expect that the numerical minimization scheme runs into similar problems as it did for small $B$.

In Figure \ref{fig:compcurve_lambda2}, we compare the numerically obtained minimal eigenvalue $\lambda_2(\Omega, B)$ with $\lambda_2(\Omega, B)$ of a single disk and that of two disjoint disks for different $B$. We only obtained numerical minimizers for $B>4\pi$ where they appeared to be a single disk. For $B< 4\pi $, we did not find numerical minimizers and think that two disjoint disks are optimal. We also see that at $B=4\pi$, the second eigenvalue of a single disk and two disjoint disks coincide with $\lambda_2(\Omega, B) = 12 \pi$. We conjecture that $B=4\pi$ is a critical point where the minimizer of $\lambda_2(\Omega, B)$ undergoes a sudden phase transition from two disjoint disks to a single disk. 
Therefore, we expect that the minimal second eigenvalue among all open domains $\lambda_2^*(B)$ resembles the black dashed line in Figure \ref{fig:compcurve_lambda2}. We summarize our observations in the following conjecture.

\begin{conjecture}[Minimizers of $\lambda_2$] \label{conj:lambda2_minimizers}
If $B\geq 4\pi$, then, among all bounded, open domains $\Omega \subset \mathbb{R}^2$ with $|\Omega|=1$, $\lambda_2(\Omega,B)$ is minimized by a single disk. If $B \leq 4\pi$, then $\lambda_2(\Omega,B)$ is minimized by two disjoint disks of equal measure.
\end{conjecture}

This conjecture should be seen in the light of the Krahn-Szegö theorem that asserts that without magnetic field (i.e.\ $B=0.0$), the second eigenvalue of the Dirichlet Laplacian is minimized by two disks of equal size. Our numerical computations suggest that the Krahn-Szegö result extends to homogeneous magnetic fields of low strength but breaks after magnetic flux exceeds the threshold of $\phi=2$.\\

\textbf{Numerical minimizers for $\lambda_3$} \\

The third eigenvalue is the lowest eigenvalue for which numerical minimization produces domains that are not disks. This happens in the range $B=4.6$ to $18.8$, see Figure \ref{fig:lambda3_final_domain} for three examples. All non-circular numerical optimizers exhibit a striking rotational symmetry of order $3$. As suggested by the plots, the optimal domains develop sharper and sharper cusps as the field strength $B$ increases. For $B\geq 12.2$, we run into the same problem as we encountered for $\lambda_2(\Omega,B)$: the cusps become too sharp and the optimization scheme leaves the region of admissible domains. Hence, we do not claim convergence for $12.2 \leq B \leq 18.8$. 

There are two remaining ranges of field strengths: $B \leq 4.4$ and $B \geq 19.0$. In both cases, we always observe convergence toward a single disk. It is interesting to note that the transition between the ranges $B\leq 4.4$ and $B\geq 4.6$ seems to be smooth in the sense that the domains with rotational symmetry of order $3$ seems to resemble a disk more and more when the field strength is lowered toward $B=4.4$.

\begin{figure}[t]
\begin{overpic}[trim={4cm 3cm 4cm 2.5cm},clip,scale=0.35]{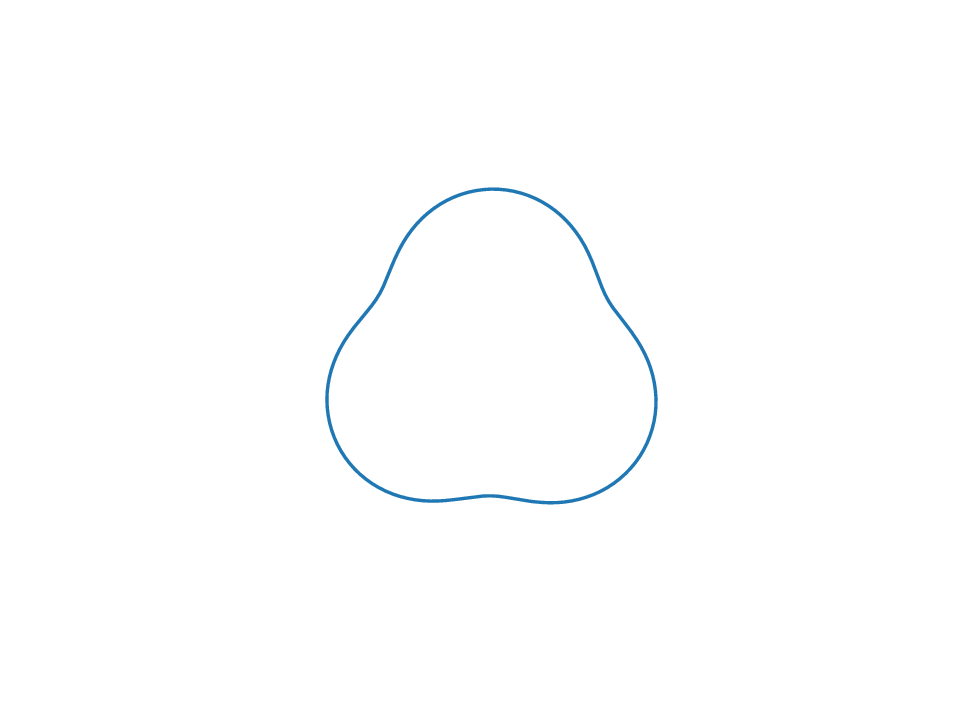}
 \put (32,-15) {$B=5.0$}
\end{overpic}
\begin{overpic}[trim={4cm 3cm 4cm 2.5cm},clip,scale=0.35]{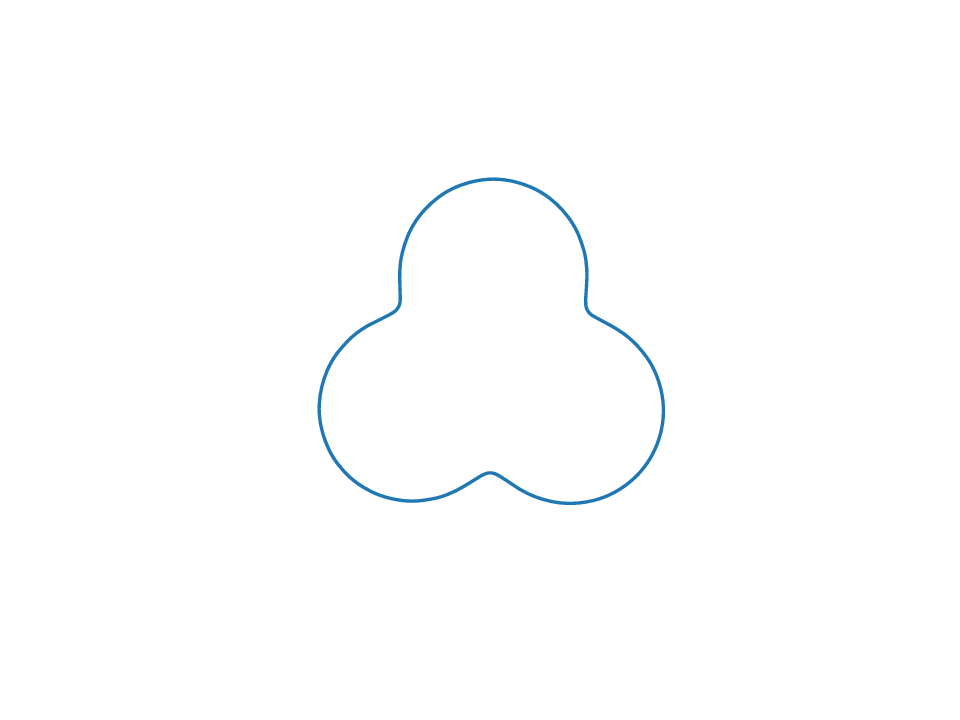}
 \put (32,-15) {$B=7.0$}
\end{overpic}
\begin{overpic}[trim={4cm 3cm 4cm 2.5cm},clip,scale=0.35]{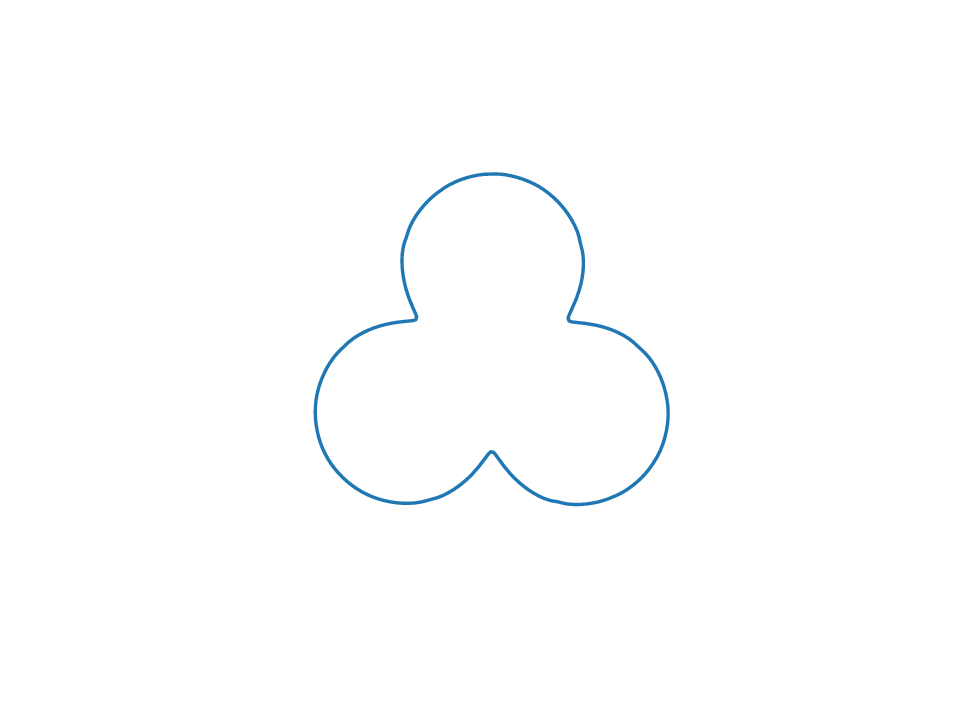}
 \put (30,-15) {$B=12.0$}
\end{overpic}
\vspace{0.5cm}
\caption{Final domains from minimization of $\lambda_3(\Omega, B)$ for various $B$.} \label{fig:lambda3_final_domain}
\end{figure}

\begin{figure}[t]
\includegraphics[scale=0.5]{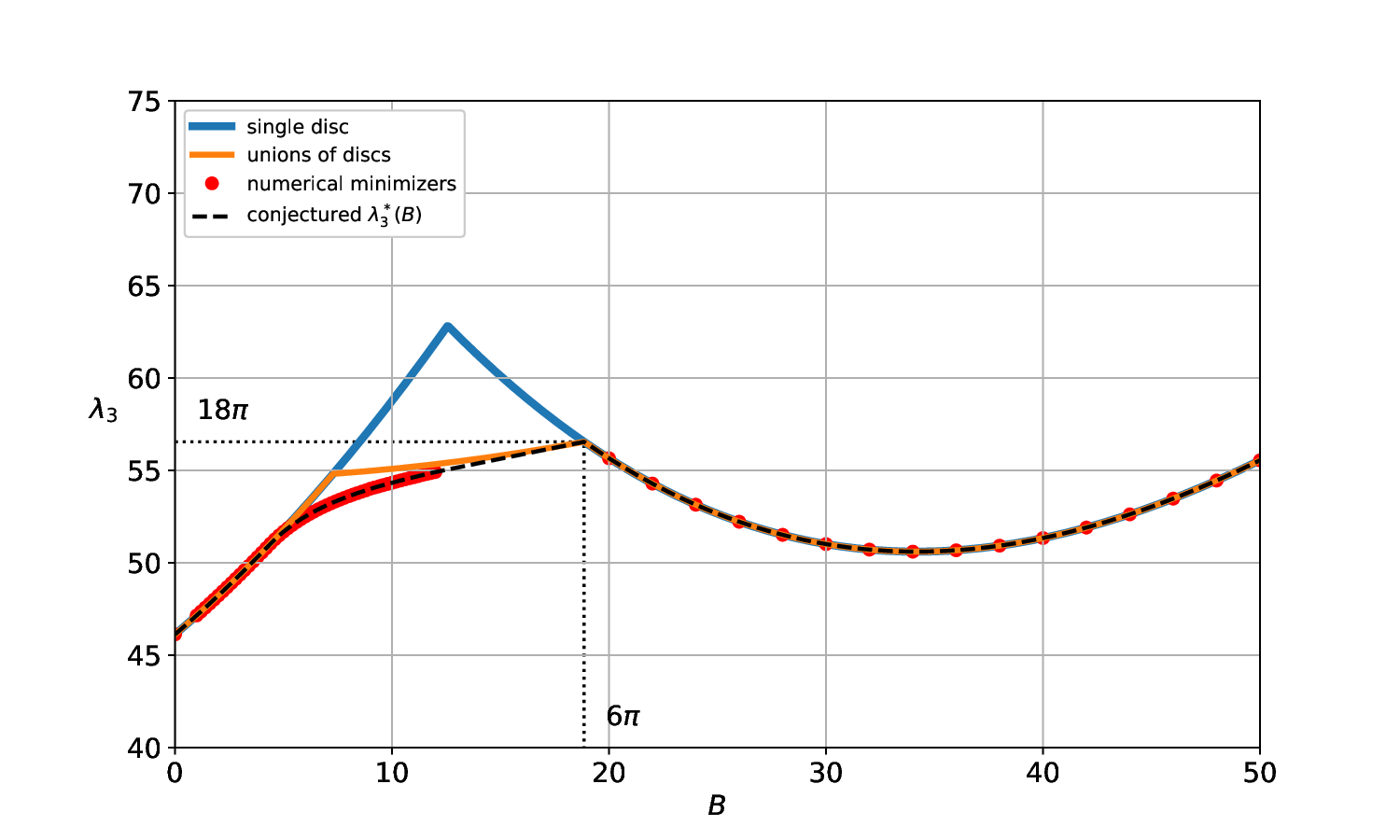}
\caption{Comparison of $\lambda_3$ of the numerical minimizers (red) with $\lambda_3$ of a single disk (blue) and the minimal $\lambda_3$ obtained over arbitrary disjoint unions of disks (orange) over different field strengths $B$. } \label{fig:compcurve_lambda3}
\end{figure}

Figure \ref{fig:compcurve_lambda3} shows the values of $\lambda_3(\Omega, B)$ of the numerical minimizers in comparison to the third eigenvalue of a single disk (blue line) and the minimal third eigenvalue obtained among arbitrary disjoint unions of disks (orange line). We see two intervals where the blue and the orange line coincide, meaning that a single disk is optimal among arbitrary disjoint unions of disks. In the interval where the orange curves lies strictly below the blue curve, the minimal third eigenvalue among arbitrary disjoint unions of disks is not obtained by a single disk. It is in fact obtained by three disks of equal size and said interval ends at $B=6\pi\approx 18.850$ where the eigenvalue $\lambda_3(\Omega, B)$ of a single disk and of three equally sized disks coincide with $\lambda_3(\Omega, B) = 18 \pi$.

Let us discuss our numerical minimizers, which are connected domains, in view of minimization among general, possibly disconnected, domains. Could there exist a disconnected minimizer with lower $\lambda_3(\Omega, B)$ than what we found among connected domains? If we assume Conjecture \ref{conj:lambda2_minimizers} to be true, Theorem \ref{thm:wolf-keller} guarantees us that any disconnected minimizer consists of either two disks of equal size or three disks with at least two of them having equal size, since the components must consist of minimizers of $\lambda_1(\Omega, B)$ and $\lambda_2(\Omega, B)$. It is therefore sufficient to compare our numerical minimizers to the minimizers obtained among arbitrary disjoint unions of disks. 

We observe in Figure \ref{fig:compcurve_lambda3} that for $B=4.6$ to $12.0$, the numerically obtained optimal values for $\lambda_3(\Omega, B)$ are strictly smaller than any $\lambda_3(\Omega, B)$ obtained among arbitrary disjoint unions of disks. For $B\leq 4.4$, the numerically obtained optimal values for $\lambda_3(\Omega, B)$ are in very good agreement with those of a single disk. We recall that it is conjectured that the third eigenvalue of the non-magnetic Dirichlet Laplacian (i.e.\ $B=0.0$) is minimized by a single disk \cite{Henrot2006}. It appears that this conjecture should continue to be true for small $B$.

We have no results for the range $B=12.2$ to $18.8$, but we believe that the optimizers are still connected and their eigenvalue $\lambda_3(\Omega,B)$ is smaller than for any disjoint union of disks. We think that the optimizers in this range of $B$ have very sharp cusps so they cannot be reached by our domain parametrization which has only a low number of Fourier coefficients and thus limited resolution for the boundary curve. From the shapes seen in Figure \ref{fig:lambda3_final_domain} it seems reasonable to expect that the optimizers approach three disconnected disks in a suitable sense when the field strength increases to $B= 6\pi \approx 18.850$. For any larger field strength, the minimizer appears to be a disk. Similar to the second eigenvalue, we think that the minimizers undergo sudden phase transition at the threshold $B=6\pi$ where they turn into a single disk. Overall, we conjecture that the minimal third eigenvalue among all open domains $\lambda_3^*(B)$ looks like the black dashed line in Figure \ref{fig:compcurve_lambda3}.\\

\textbf{Numerical minimizers for $\lambda_4$} \\ 

Similar to the case of the second eigenvalue, optimizers for the fourth eigenvalue show convergence issues for small $B$ (experimentally, $B\leq 3.8$). This is not too surprising in view of the conjecture that the minimizer of $\lambda_4$ of the non-magnetic Dirichlet Laplacian is given by two disjoint disks (but unlike for the second eigenvalue, of two different sizes), see \cite{Henrot2006}. We think that the minimizers for $B$ close to zero are either disconnected as well or have very sharp cusps which leads to failure of the optimization routine.

For $4.0 \leq B \leq 19.4$, we observe convergence and obtain simply connected domains. Six minimizers of this range are seen in Figure \ref{fig:lambda4_final_domain}. The minimizers seem to deform continuously from an elongated shape with rotational symmetry of degree $2$ to a disk and then to a flower shape with four petals having rotational symmetry of degree $4$. Then, for a small range above $19.4$ and just below $B=8\pi \approx 25.133$, convergence issues arise again as the cusps of the flower shape become too sharp. Above $B=8\pi$, we again observe convergence and all  minimizers appear to be a single disk.

Figure \ref{fig:compcurve_lambda4} compares the numerically obtained minimal eigenvalues $\lambda_4(\Omega, B)$ with those of a single disk (blue), arbitrary disjoint unions of disks (orange) and the optimal disjoint union of previous minimizers (green). The last curve was computed with \eqref{eq:wolf-keller} and under the assumption that the conjectured curves $\lambda_k^*(B)$ for $k=1,2,3$ for previous eigenvalues are optimal. Note that except at $B=0.0$ and $B=8\pi$, the green curve lies significantly above the minimal eigenvalues $\lambda_4(\Omega, B)$ of our simply connected numerical minimizers, which is a strong indication that the global minimizers are indeed simply connected in the ranges where our method converged.\\

\begin{figure}[t]
\begin{overpic}[trim={4cm 3cm 4cm 2.5cm},clip,scale=0.35]{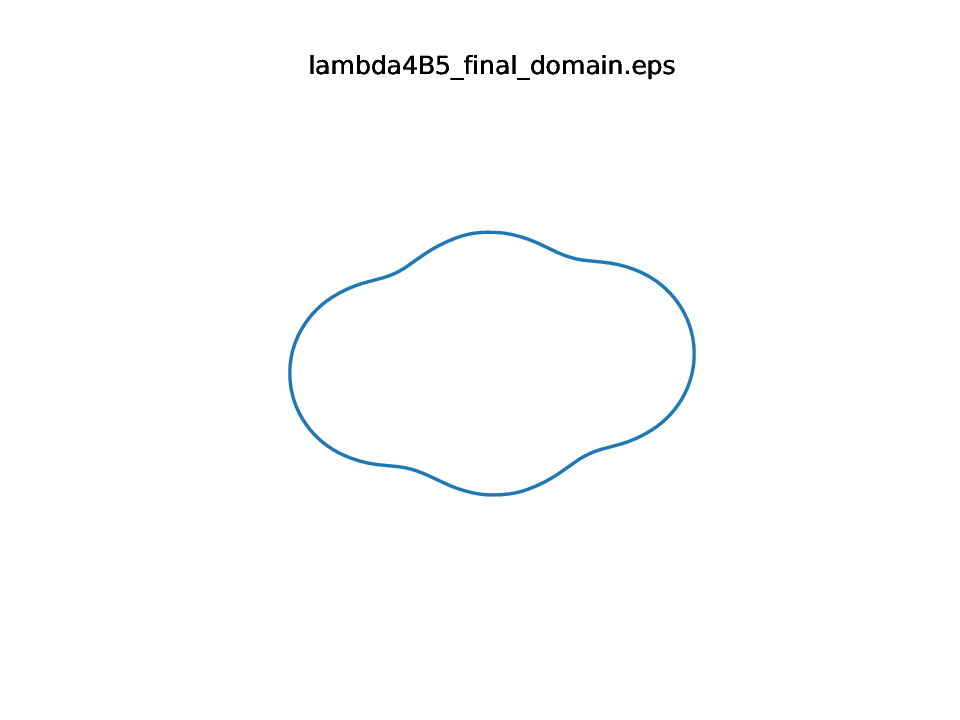}
 \put (32,-15) {$B=5.0$}
\end{overpic}
\begin{overpic}[trim={4cm 3cm 4cm 2.5cm},clip,scale=0.35]{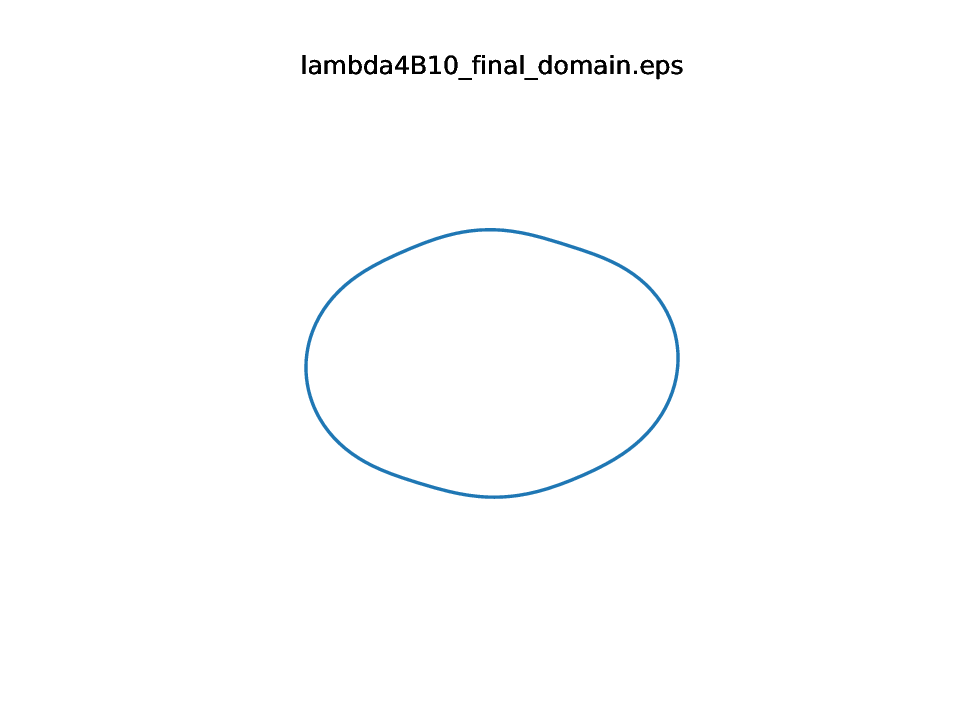} 
 \put (32,-15) {$B=10.0$}
\end{overpic}
\begin{overpic}[trim={4cm 3cm 4cm 2.5cm},clip,scale=0.35]{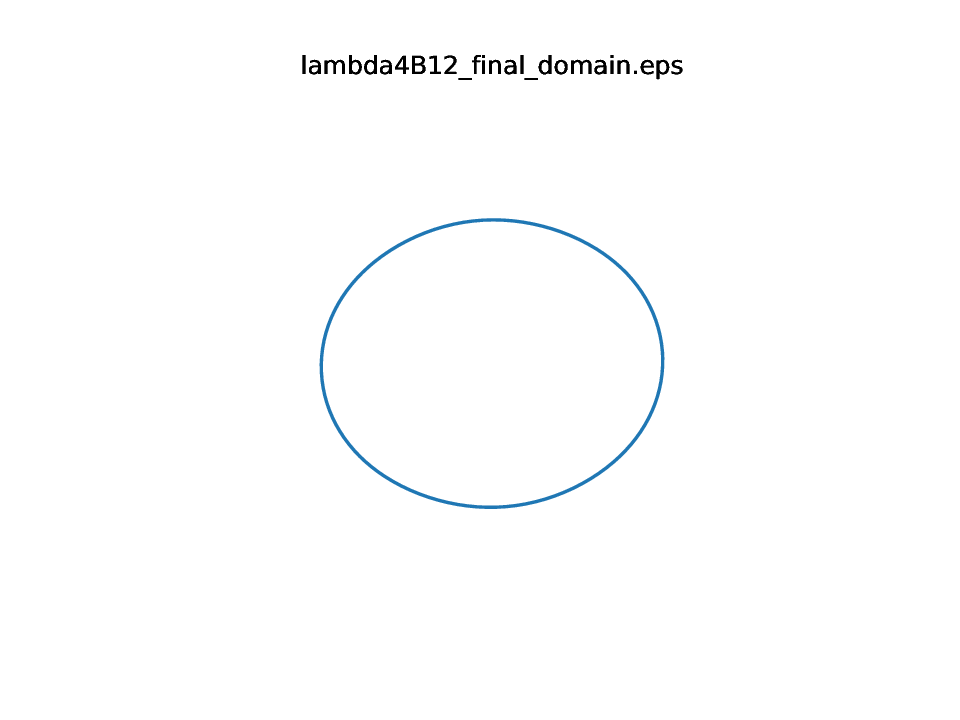}
 \put (32,-15) {$B=12.0$}
\end{overpic}
\vspace{1cm}

\begin{overpic}[trim={4cm 3cm 4cm 2.5cm},clip,scale=0.35]{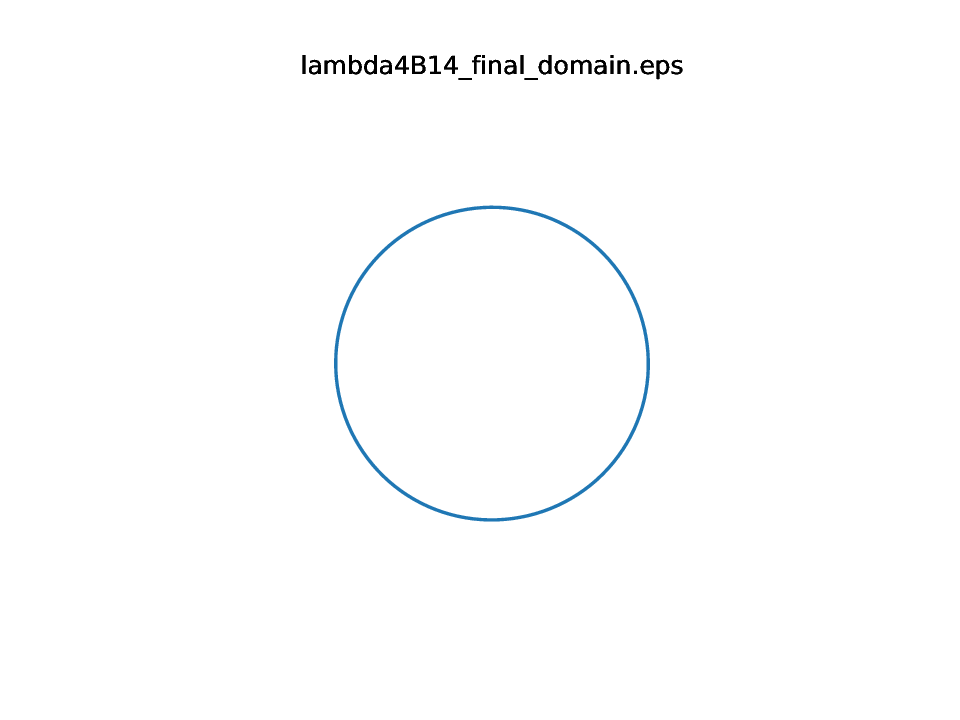}
 \put (32,-15) {$B=14.0$}
\end{overpic}
\begin{overpic}[trim={4cm 3cm 4cm 2.5cm},clip,scale=0.35]{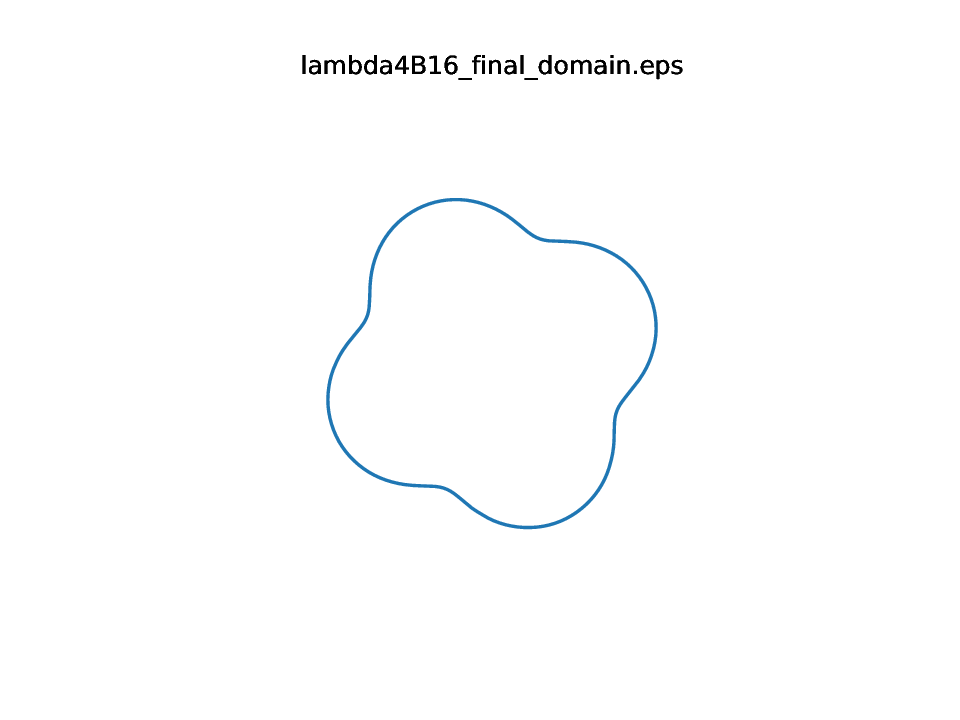} 
 \put (32,-15) {$B=16.0$}
\end{overpic}
\begin{overpic}[trim={4cm 3cm 4cm 2.5cm},clip,scale=0.35]{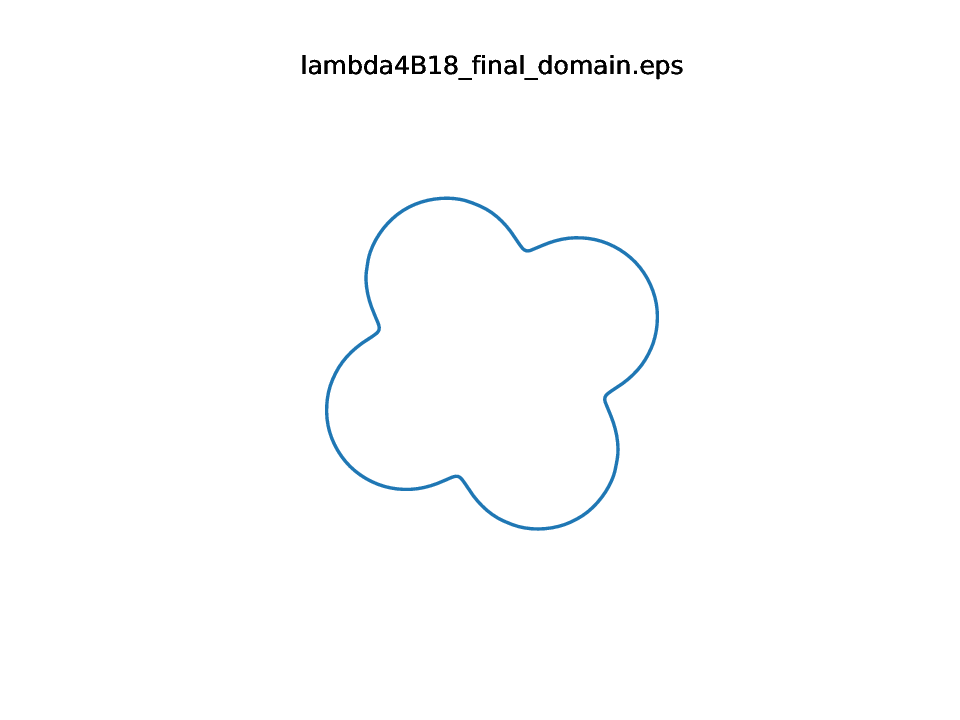}
 \put (32,-15) {$B=18.0$}
\end{overpic}
\vspace{0.5cm}
\caption{Final domains from minimization of $\lambda_4(\Omega, B)$ for various $B$.} \label{fig:lambda4_final_domain}
\end{figure}

\begin{figure}[t]
\includegraphics[scale=0.5]{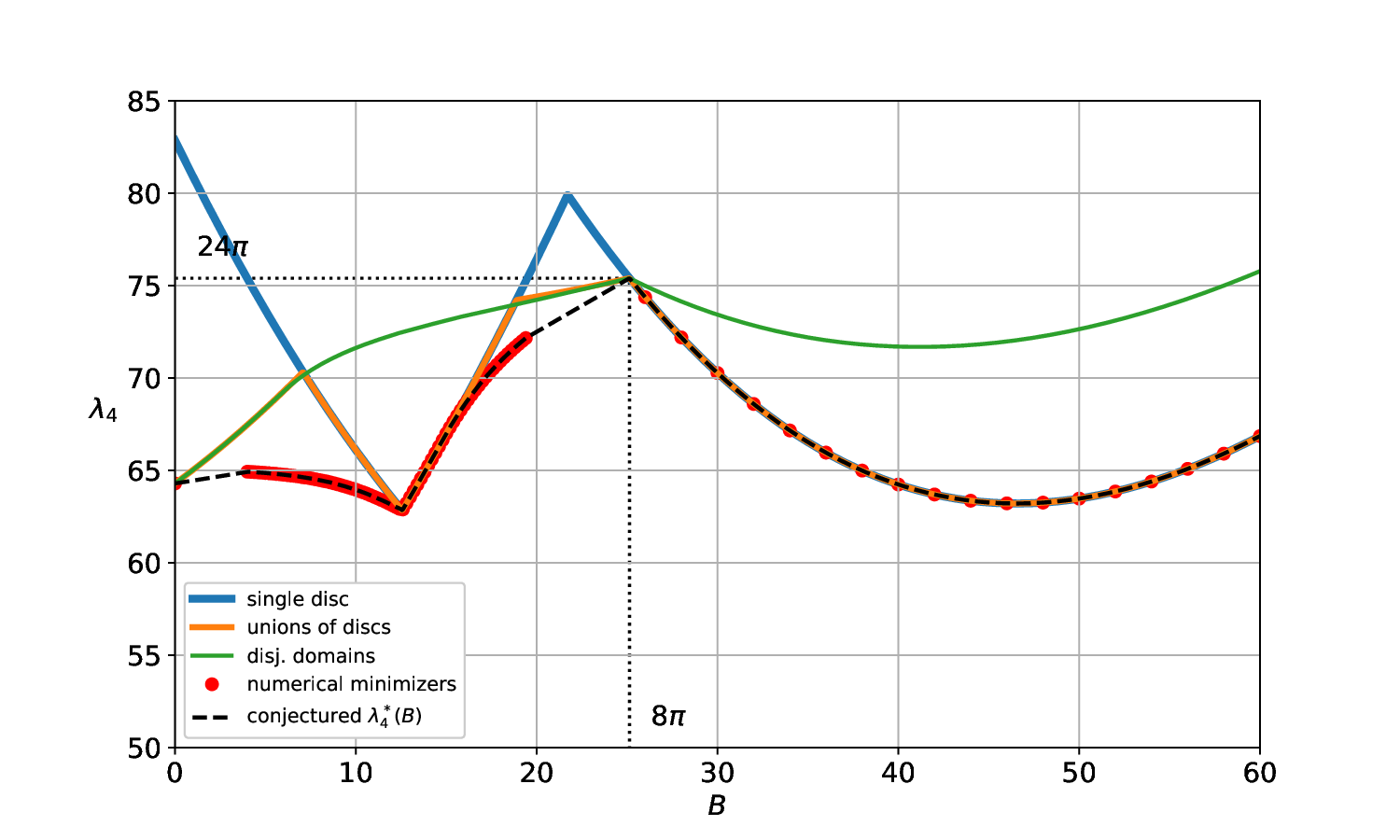}
\caption{Comparison of $\lambda_4$ of the numerical minimizers (red) with $\lambda_4$ of a single disk (blue), the minimal $\lambda_4$ obtained over arbitrary disjoint unions of disks (orange) and the minimal $\lambda_4$ obtained over disjoint unions of previous minimizers (green) over different field strengths $B$. } \label{fig:compcurve_lambda4}
\end{figure}

\textbf{Numerical minimizers for $\lambda_5$, $\lambda_6$ and $\lambda_7$}\\

The eigenvalues $\lambda_5$, $\lambda_6$ and $\lambda_7$ lead to convergence for all field strengths except just below $B=10 \pi \approx 31.416$, $B=12 \pi \approx 37.699$ and $B=14 \pi \approx 43.982$, respectively. Figures \ref{fig:lambda5_final_domain}, \ref{fig:lambda6_final_domain} and \ref{fig:lambda7_final_domain} show shapes of numerical minimizers of $\lambda_5$, $\lambda_6$ and $\lambda_7$ at various field strengths below $B=10 \pi $, $B=12 \pi $ and $B=14 \pi $. Above these critical field strengths, all numerical minimizers of $\lambda_5$, $\lambda_6$ and $\lambda_7$ appear to be a single disk. Recall that the disk also appeared to be the numerical minimizer for $\lambda_2$ if $B\geq 4\pi$, for $\lambda_3$ if $B\geq 6\pi$ and for $\lambda_4$ if $B\geq 8\pi$. Supported by this numerical evidence, it seems reasonable to make the following conjecture.

\begin{conjecture}[Minimizers of $\lambda_n$]
Let $n \geq 2$. If $B\geq 2\pi n$, then, among all bounded, open domains $\Omega \subset \mathbb{R}^2$ with $|\Omega|=1$, $\lambda_n(\Omega,B)$ is minimized by a single disk, i.e.
\begin{align*}
\lambda_n(\Omega,B) \geq \lambda_n(D, B),
\end{align*}
where $D$ is a disk with $|D|=1$.
\end{conjecture}

We note that the condition $B\geq 2\pi n$ is equivalent to $\phi \geq n$ where $\phi$ is the magnetic flux through the domain. Thus, the conjecture from above can be equivalently stated in the easy-to-remember form

\begin{center}
''if $\phi \geq n$, then $\lambda_n(\Omega,B)$ is minimized by a single disk.''
\end{center}

We find it remarkable that the shapes of the minimizers of $\lambda_3$ to $\lambda_7$ undergo similar patterns even below the critical field strength of $2\pi$ times the eigenvalue index. We identify the following pattern: For $\lambda_n$, $n\geq 3$, the minimizers first seem to become a single disk for a certain field strength strictly below $2\pi n$. As the field strength increases, they start to look like a flower with $n$ petals, exhibiting a rotational symmetry of degree $n$. The domains then develop sharp cusps when $B$ goes to $2\pi n$ and the minimization routine fails to converge. After the threshold of $B=2\pi n$ is surpassed, the minimizers of $\lambda_n$ suddenly return to a single disk and stay at this domain for any larger field strength.

In the regime where $ \phi < n$, the transition between minimizers seems to be smooth except at certain field strengths. There appears to be a phase transition for $\lambda_5$ at around $B=4.2$, where the minimizer jumps from a domain with rotational symmetry of degree $2$ to a domain\footnote{Various people have pointed out the resemblance of the shape with the head of a certain popular comic mouse character. James B. Kennedy has argued that the shape looks more like a koala head.} with only one symmetry axis, see Figure \ref{fig:lambda5_final_domain}. For $\lambda_6$, we observe one phase transition at $B=4.8$, where the minimizer jumps from a domain with rotational symmetry of degree $3$ to one of degree $2$, see Figure \ref{fig:lambda6_final_domain}. Similarly, for $\lambda_7$, we observe two phase transitions. The first one appears around $B=3.4$, the second one at $B=13.0$, see Figure \ref{fig:lambda7_final_domain}, both changing the symmetry type of the minimizer. Generally, these phase transitions can be recognized in plots of Figure \ref{fig:compcurve_lambda567} as well where there appears to be a corner in the curve representing the numerically attained minimal eigenvalue $\lambda_n$.

Contrary to that, it appears that the symmetry type of the minimizers can also change via a smooth transition by passing through the disk. This appears for $\lambda_5$ around $B=22.0$, for $\lambda_6$ around $B=30.0$ and for $\lambda_7$ around $B=10.0$ and $B=37.0$.

We interpret this behaviour of the minimizers as a manifestation of multiple branches of local minimizers. For higher eigenvalues, we observed that it is not unusual that the numerical optimization scheme finds multiple simply connected, local minimizers for fixed field strength $B$, when starting from different initial domains. The domains we present are gained by taking the minimum over all local minima. The local minima seem to form branches over which the minimizing domains transform smoothly with $B$. As the field strengths varies, it occurs that some branches fall below others and become the current global minimizer. A phase transition happens when the eigenvalue $\lambda_n$ of one local minimizer branch undercuts the same eigenvalue of another local minimizer branch. On the other hand, when a minimizer branch passes through the disk, it can lead to a change of the symmetry type of the global minimizer without sudden phase transition. 

Concerning connectedness of minimizers, we observe a rather puzzling phenomenon for $\lambda_5$. In Figure \ref{fig:compcurve_lambda567}, it appears that for $B=13.6$ to $16.4$, the numerically obtained simply connected minimizers have a larger eigenvalue $\lambda_5$ than what can be achieved by a disjoint union of previous minimizers (green line). The numerically obtained simply connected minimizers are in this case beaten by a disjoint union of scaled minimizers of $\lambda_4$ and $\lambda_1$.  This would mean that in the above mentioned range the global minimizer suddenly changes from being connected to being disconnected and being connected again. The same phenomenon cannot be observed for $\lambda_6$ or $\lambda_7$ where the disconnected domains always yield larger eigenvalues than the numerically obtained simply connected minimizers. 

This phenomenon begs the question: Why did we not observe that the minimization scheme tries to approach a disconnected domain consisting of scaled minimizers of $\lambda_4$ and $\lambda_1$, producing a domain with sharp cusps? As it turns out, our initial domains were too close to a local minimizer and the gradient descent scheme was trapped in the local well. Handpicking an initial domain with having already some inward dents, see Figure \ref{fig:lambda5_disc_rerun_domains} (left), and performing the minimization procedure for $B=16.2$, we observe that the minimization procedure indeed produces a domain with sharp cusps that looks as if the scaled minimizers of $\lambda_4$ and $\lambda_1$ were stuck together, see Figure \ref{fig:lambda5_disc_rerun_domains} (middle). The final domain exhibits a fifth eigenvalue ($\lambda_5 \approx 81.41$) which is smaller than that of the connected local minimizer we found ($\lambda_5 \approx 81.69$) but slightly larger than that of the optimal disconnected domain consisting of scaled minimizers of $\lambda_4$ and $\lambda_1$ ($\lambda_5 \approx 81.27$). This makes it very plausible that the global minimizers of $\lambda_5$ indeed become suddenly disconnected for a short range of $B$.

\begin{figure}[t]
\begin{overpic}[trim={4cm 3cm 4cm 2.5cm},clip,scale=0.35]{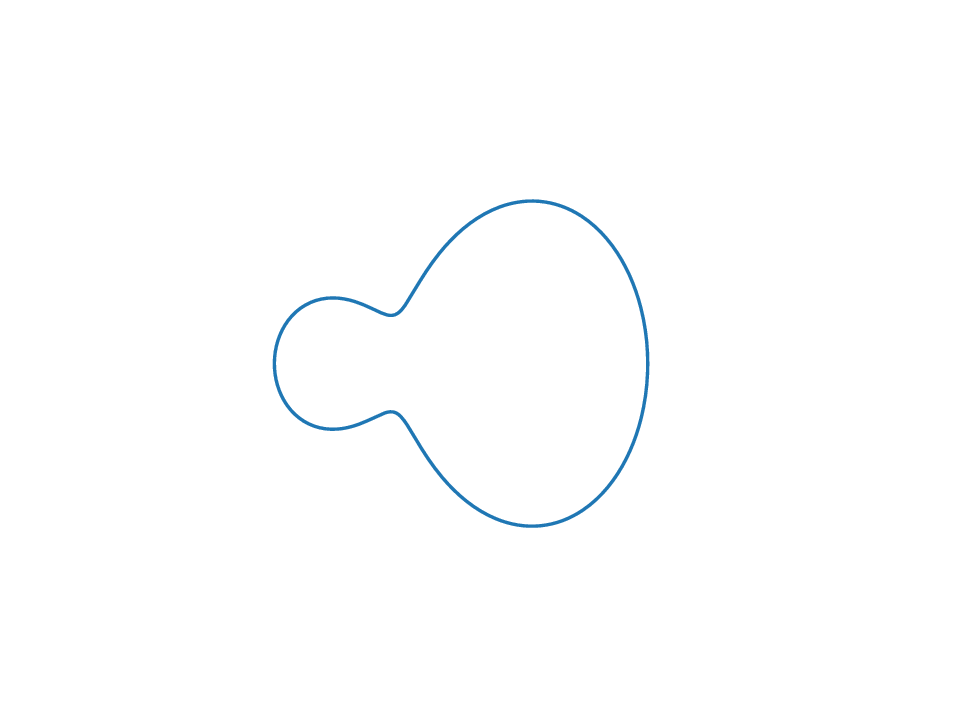}
 \put (32,-15) {initial }
 \put (0,-32) {$\lambda_5(\Omega, B) \approx 88.89$	}
\end{overpic}
\hspace{0.5cm}
\begin{overpic}[trim={4cm 3cm 4cm 2.5cm},clip,scale=0.35]{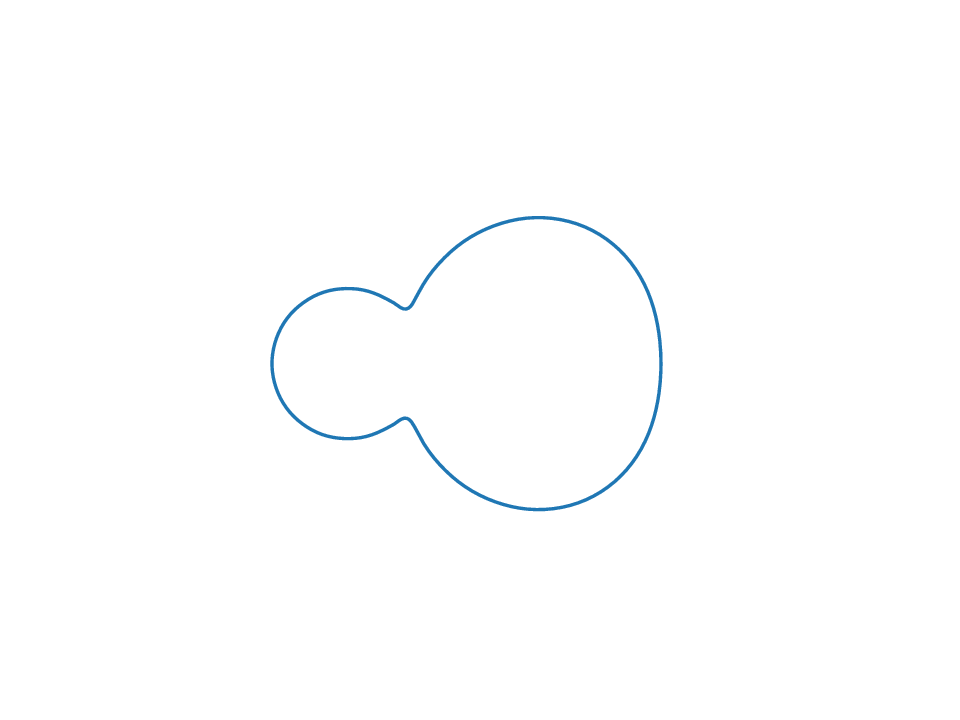}
 \put (38,-15) {final}
 \put (2,-32) {$\lambda_5(\Omega, B) \approx 81.41$	}
\end{overpic}
\hspace{0.5cm}
\begin{overpic}[trim={4cm 3cm 4cm 2.5cm},clip,scale=0.35]{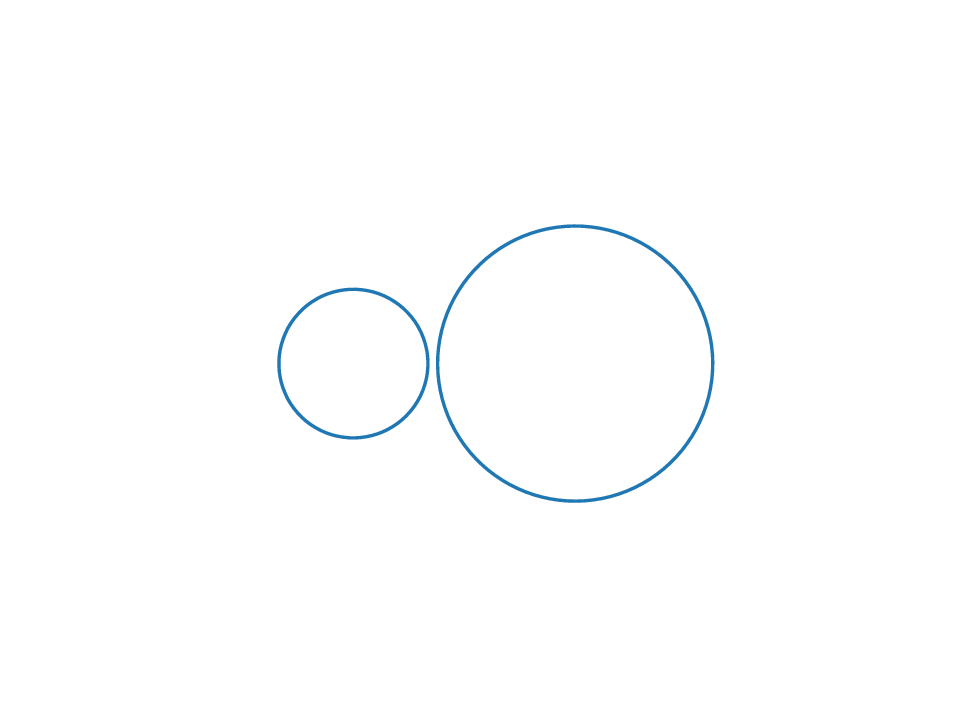}
 \put (35,-15) {disjoint}
 \put (7,-32) {$\lambda_5(\Omega, B) \approx 81.27$	}
\end{overpic}
\vspace{1.0cm}

\caption{Comparison of hand-picked initial domain with cusps with its final domain after minimization and optimal disjoint domain for $\lambda_5(\Omega, B)$ where $B=16.2$.} \label{fig:lambda5_disc_rerun_domains}
\end{figure}

\begin{figure}[p]
\begin{overpic}[trim={4cm 3cm 4cm 2.5cm},clip,scale=0.35]{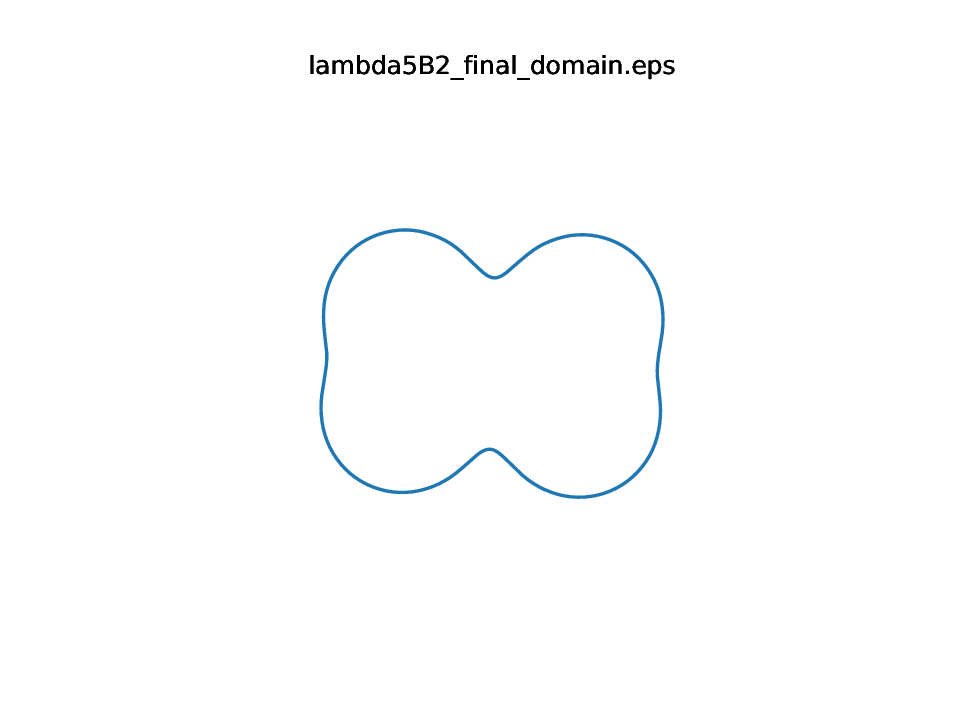}
 \put (32,-15) {$B=2.0$}
\end{overpic}
\begin{overpic}[trim={4cm 3cm 4cm 2.5cm},clip,scale=0.35]{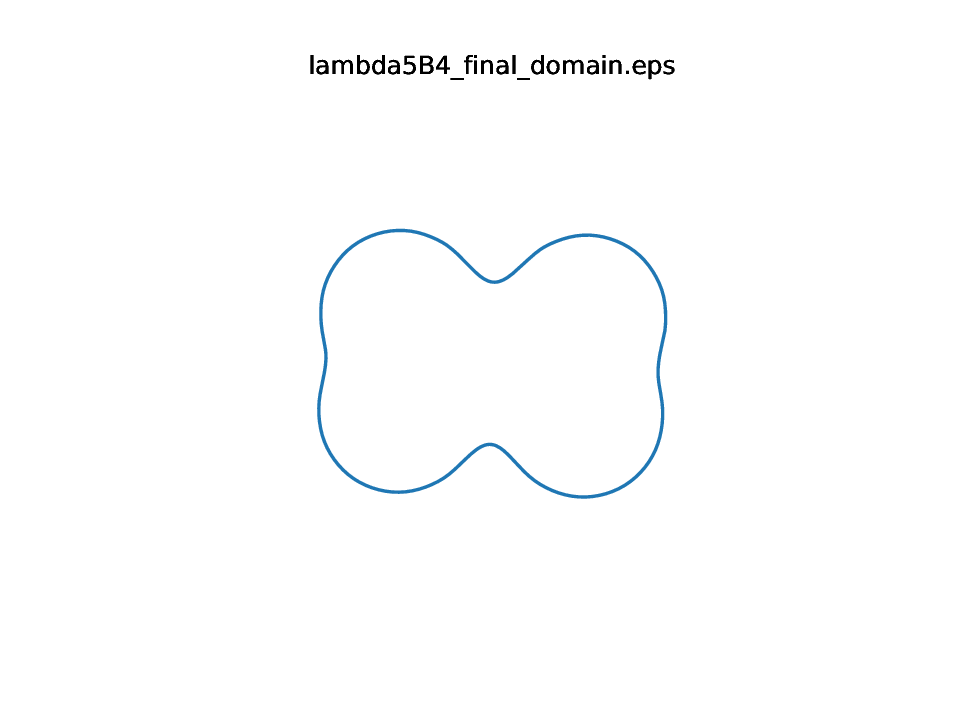}
 \put (32,-15) {$B=4.0$}
\end{overpic}
\begin{overpic}[trim={4cm 3cm 4cm 2.5cm},clip,scale=0.35]{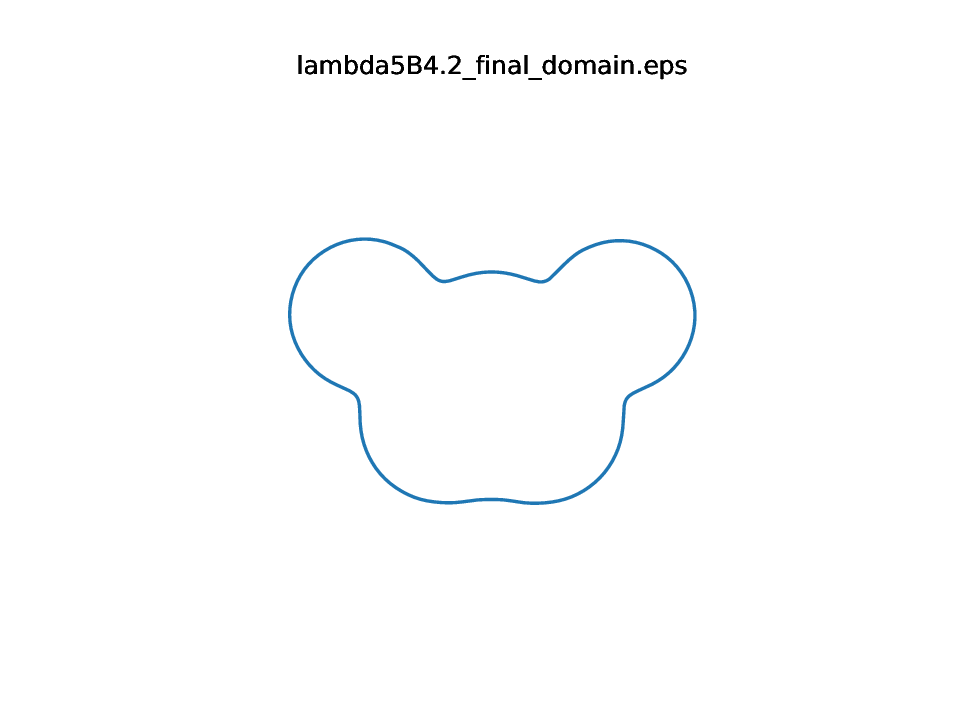}
 \put (32,-15) {$B=4.2$}
\end{overpic}
\begin{overpic}[trim={4cm 3cm 4cm 2.5cm},clip,scale=0.35]{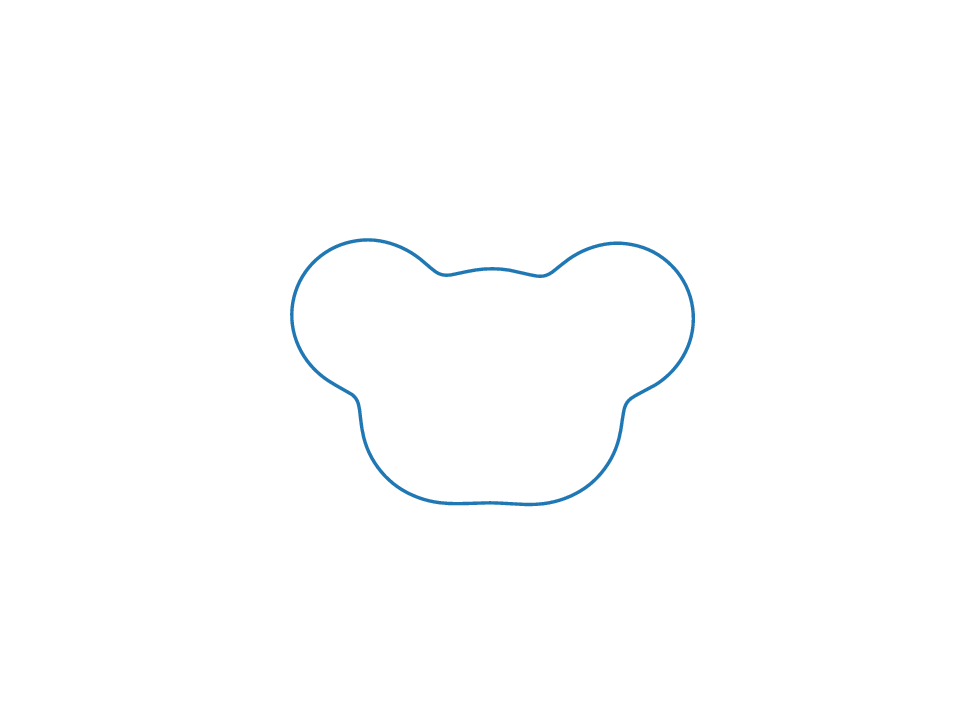}
 \put (32,-15) {$B=10.0$}
\end{overpic}
\begin{overpic}[trim={4cm 3cm 4cm 2.5cm},clip,scale=0.35]{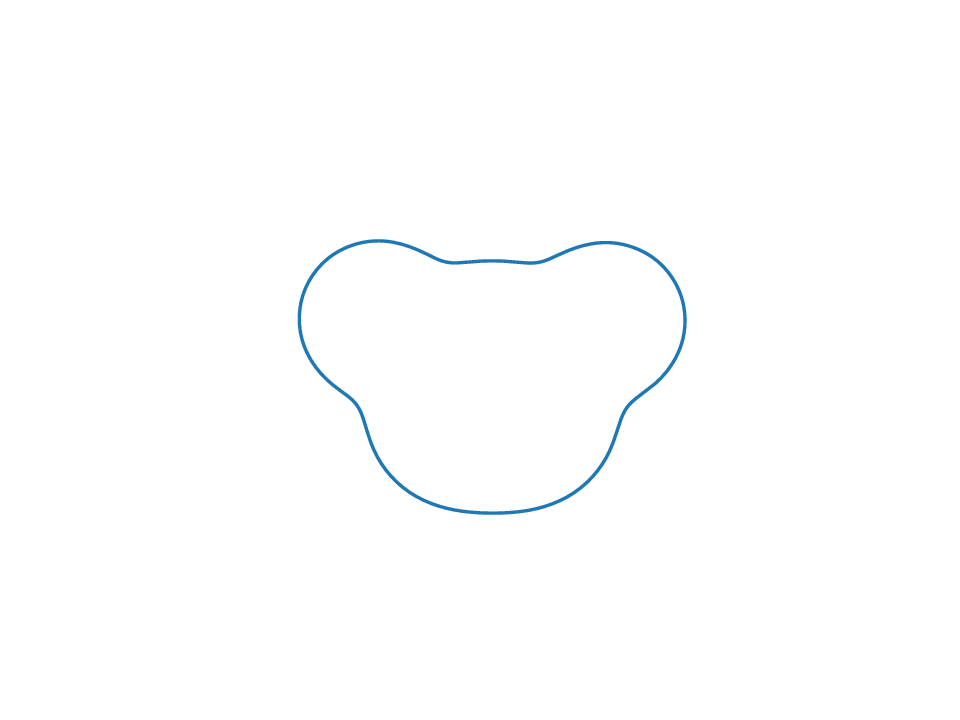}
 \put (32,-15) {$B=15.0$}
\end{overpic}
\vspace{1cm}

\begin{overpic}[trim={4cm 3cm 4cm 2.5cm},clip,scale=0.35]{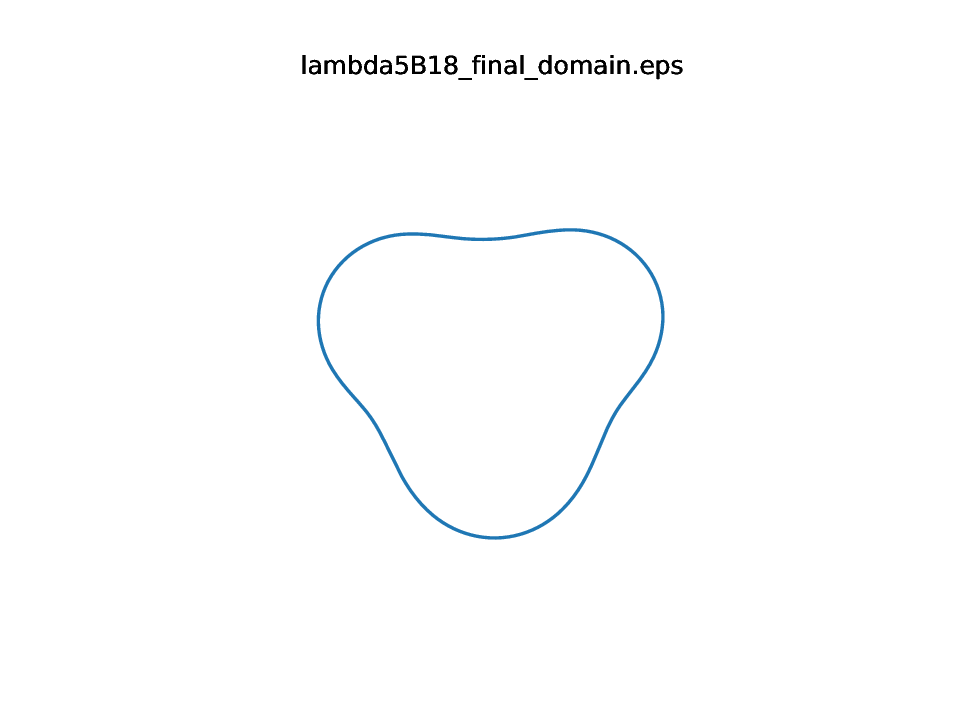}
 \put (30,-15) {$B=18.0$}
\end{overpic}
\begin{overpic}[trim={4cm 3cm 4cm 2.5cm},clip,scale=0.35]{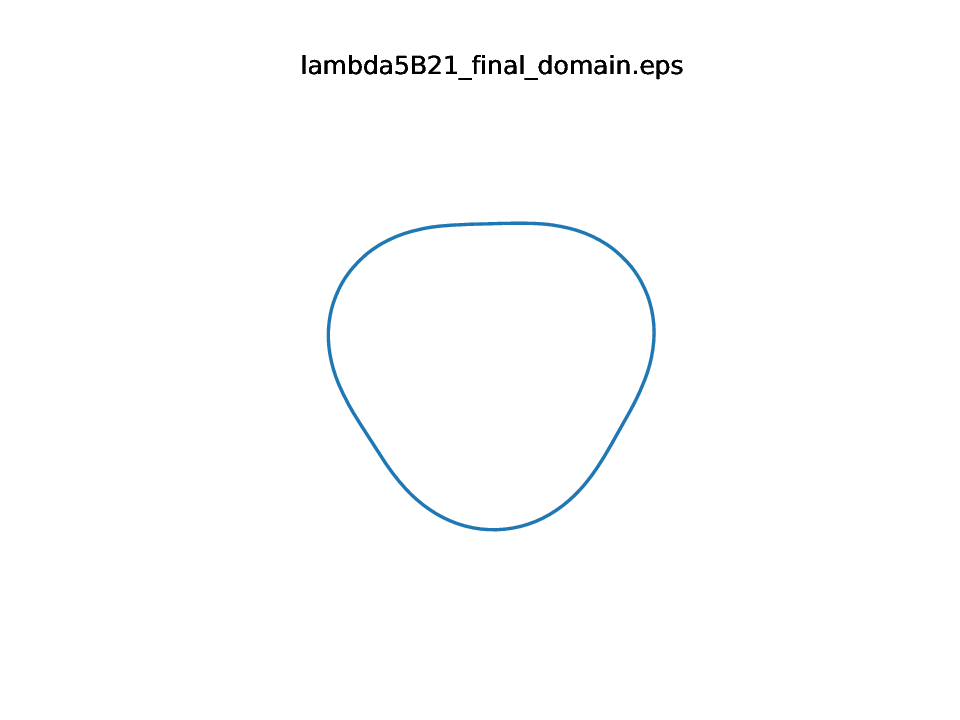}
 \put (30,-15) {$B=21.0$}
\end{overpic}
\begin{overpic}[trim={4cm 3cm 4cm 2.5cm},clip,scale=0.35]{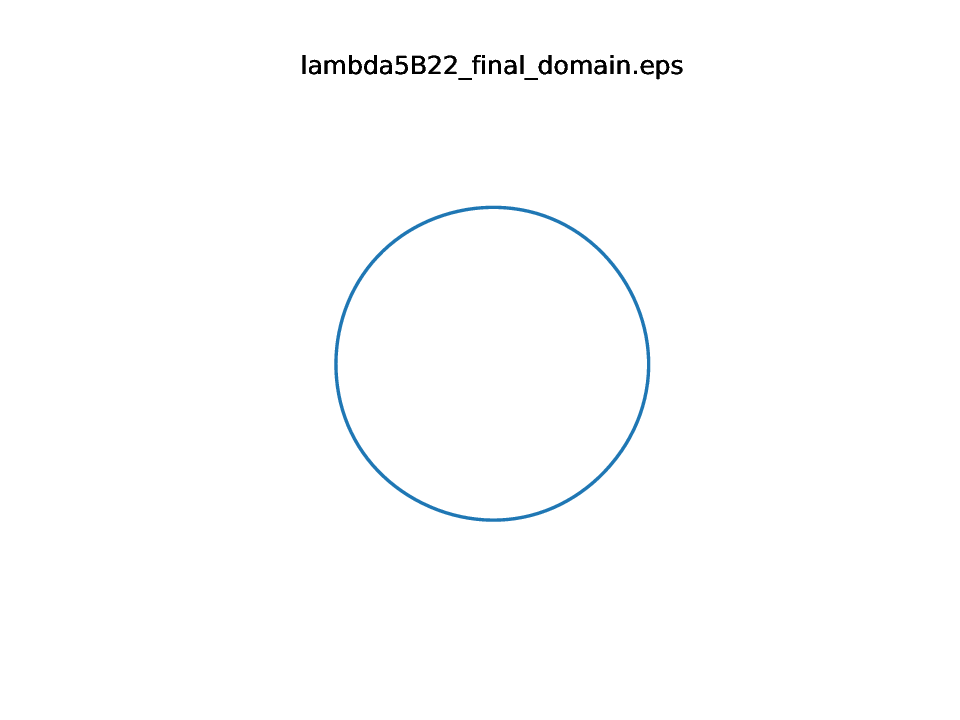}
 \put (30,-15) {$B=22.0$}
\end{overpic}
\begin{overpic}[trim={4cm 3cm 4cm 2.5cm},clip,scale=0.35]{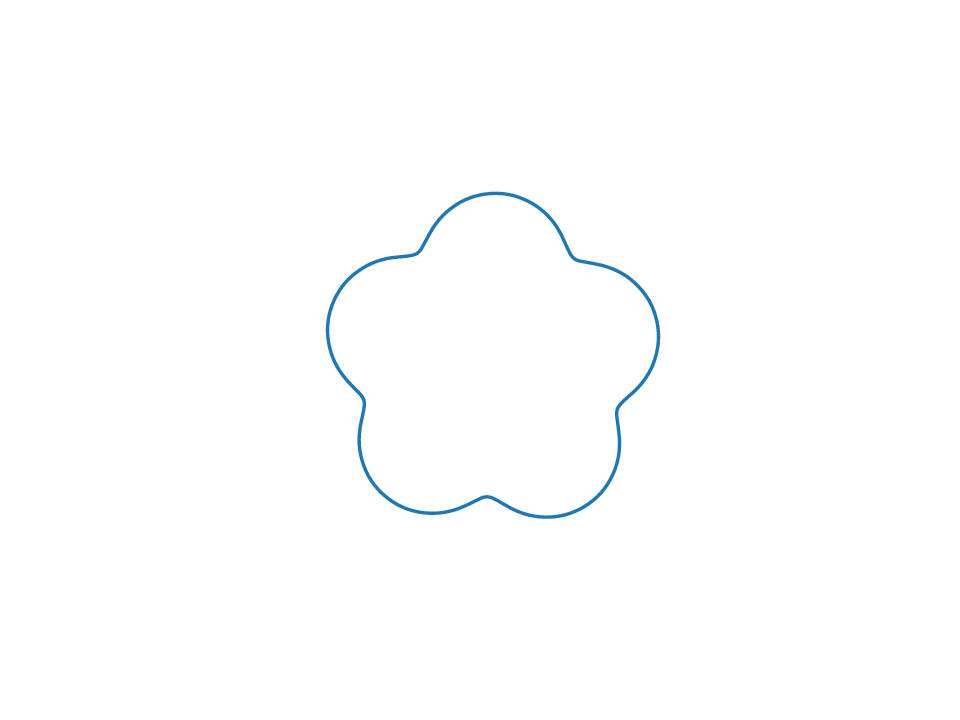}
 \put (30,-15) {$B=25.0$}
\end{overpic}
\begin{overpic}[trim={4cm 3cm 4cm 2.5cm},clip,scale=0.35]{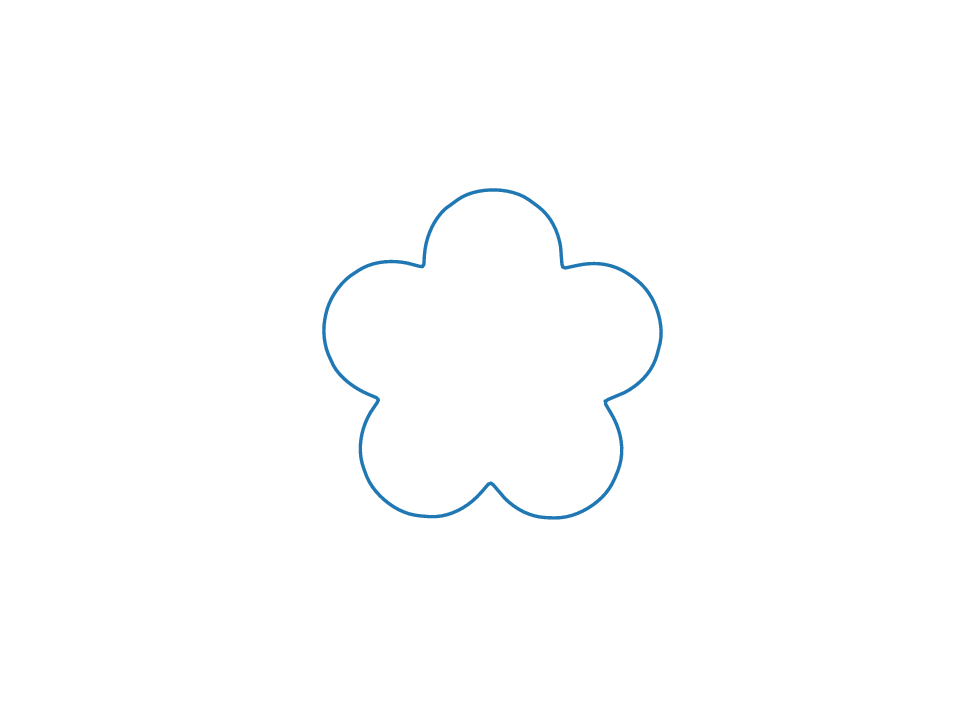}
 \put (30,-15) {$B=28.0$}
\end{overpic}
\vspace{1cm}
\caption{Final domains from minimization of $\lambda_5(\Omega, B)$ for various $B$.} \label{fig:lambda5_final_domain}
\end{figure}

\begin{figure}[h]
\begin{overpic}[trim={4cm 3cm 4cm 2.5cm},clip,scale=0.35]{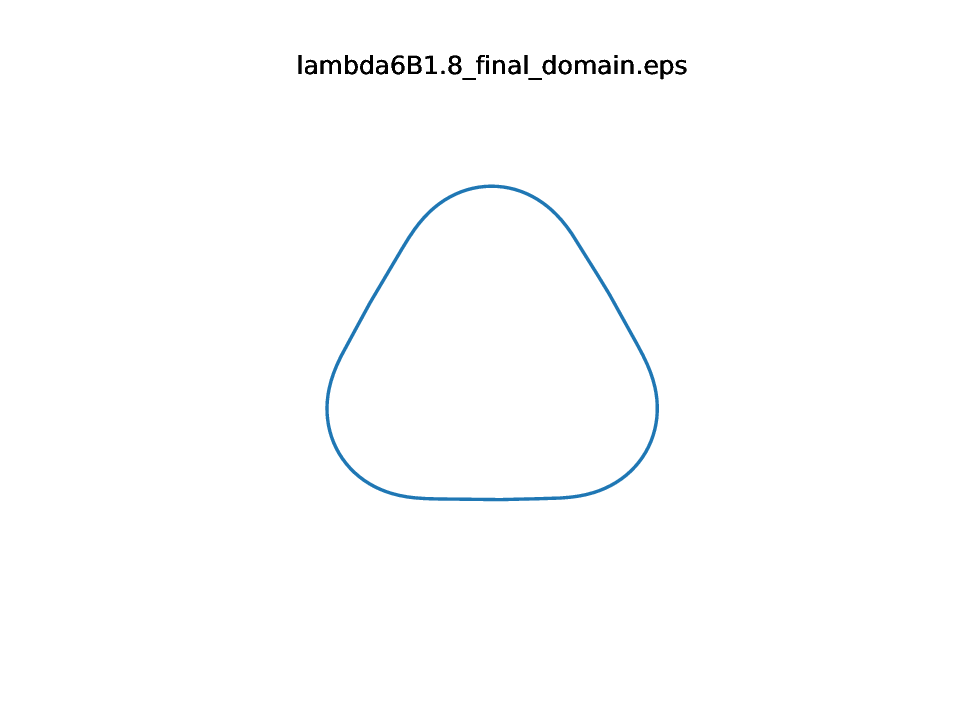}
 \put (32,-15) {$B=1.8$}
\end{overpic}
\begin{overpic}[trim={4cm 3cm 4cm 2.5cm},clip,scale=0.35]{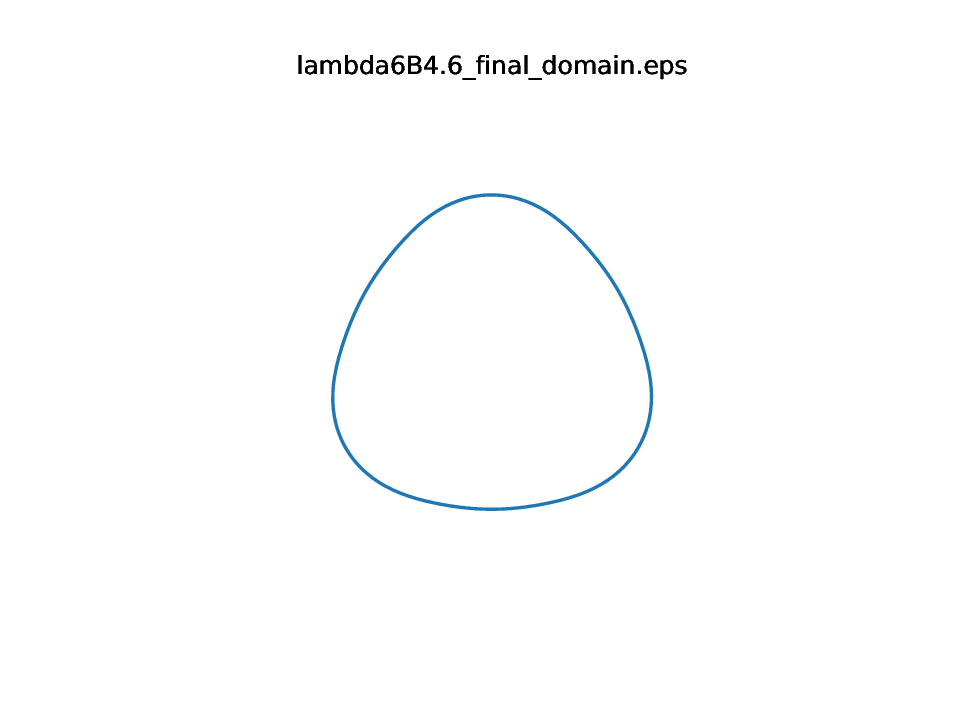}
 \put (32,-15) {$B=4.6$}
\end{overpic}
\begin{overpic}[trim={4cm 3cm 4cm 2.5cm},clip,scale=0.35]{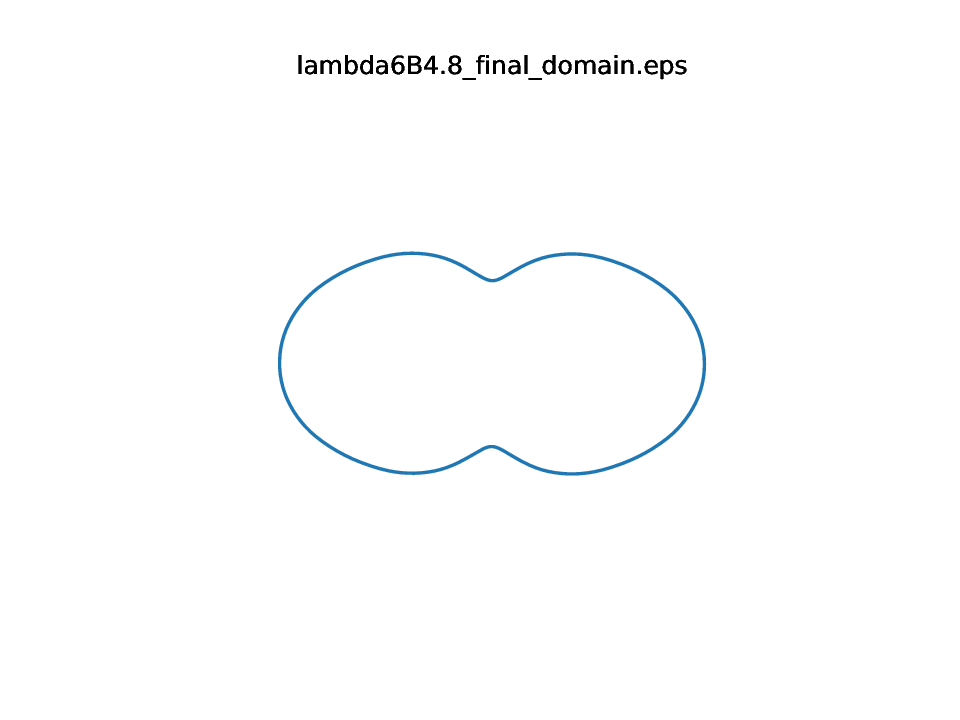}
 \put (32,-15) {$B=4.8$}
\end{overpic}
\begin{overpic}[trim={4cm 3cm 4cm 2.5cm},clip,scale=0.35]{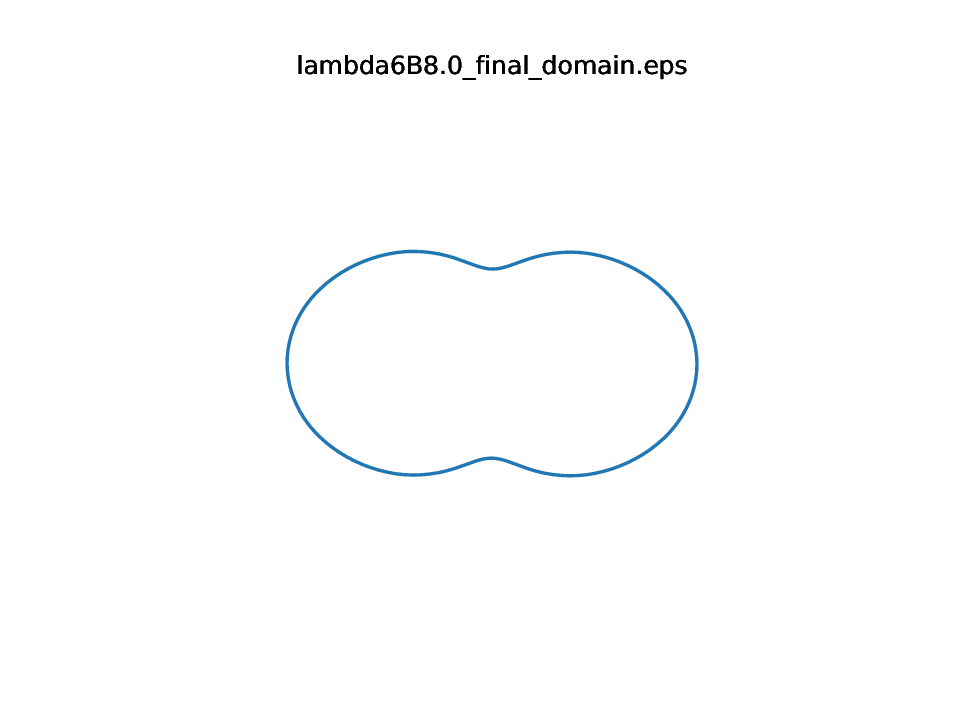}
 \put (32,-15) {$B=8.0$}
\end{overpic}
\begin{overpic}[trim={4cm 3cm 4cm 2.5cm},clip,scale=0.35]{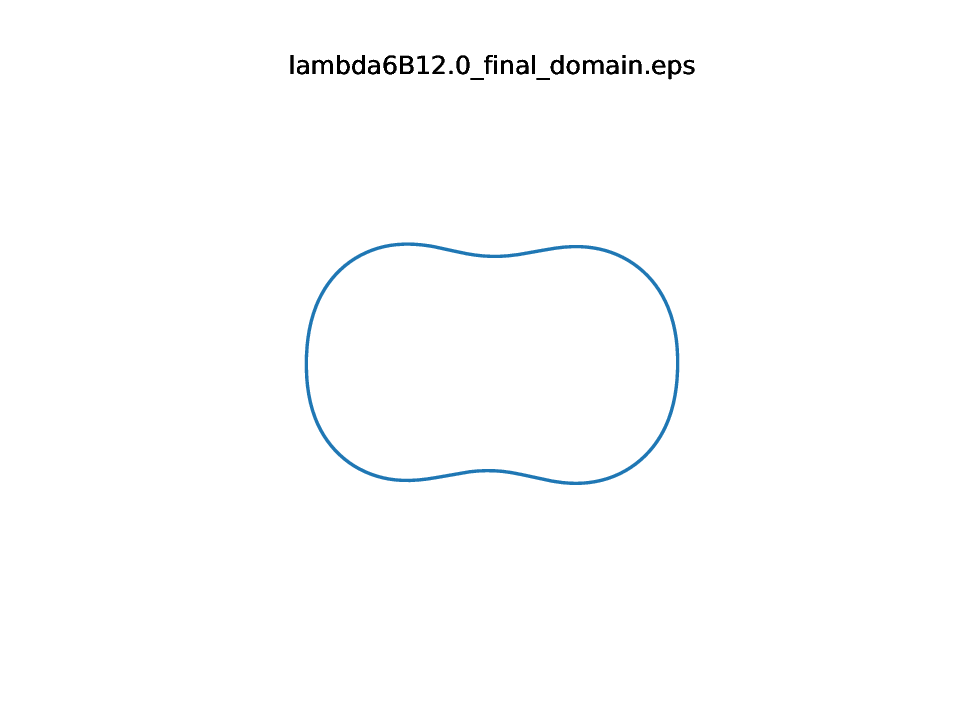}
 \put (32,-15) {$B=12.0$}
\end{overpic}
\vspace{1cm}

\begin{overpic}[trim={4cm 3cm 4cm 2.5cm},clip,scale=0.35]{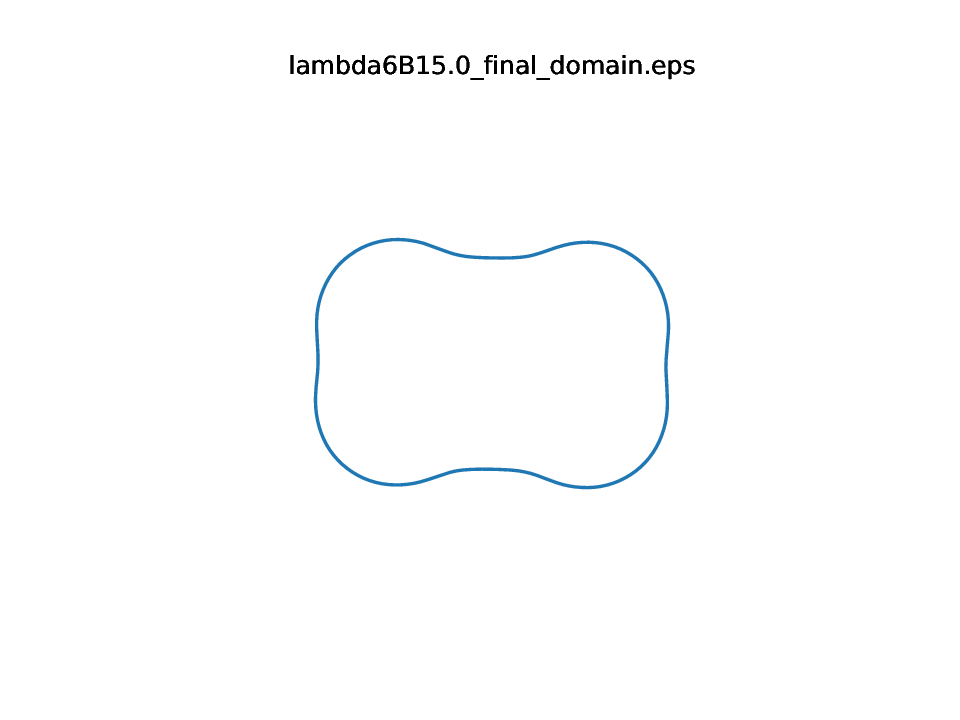}
 \put (32,-15) {$B=15.0$}
\end{overpic}
\begin{overpic}[trim={4cm 3cm 4cm 2.5cm},clip,scale=0.35]{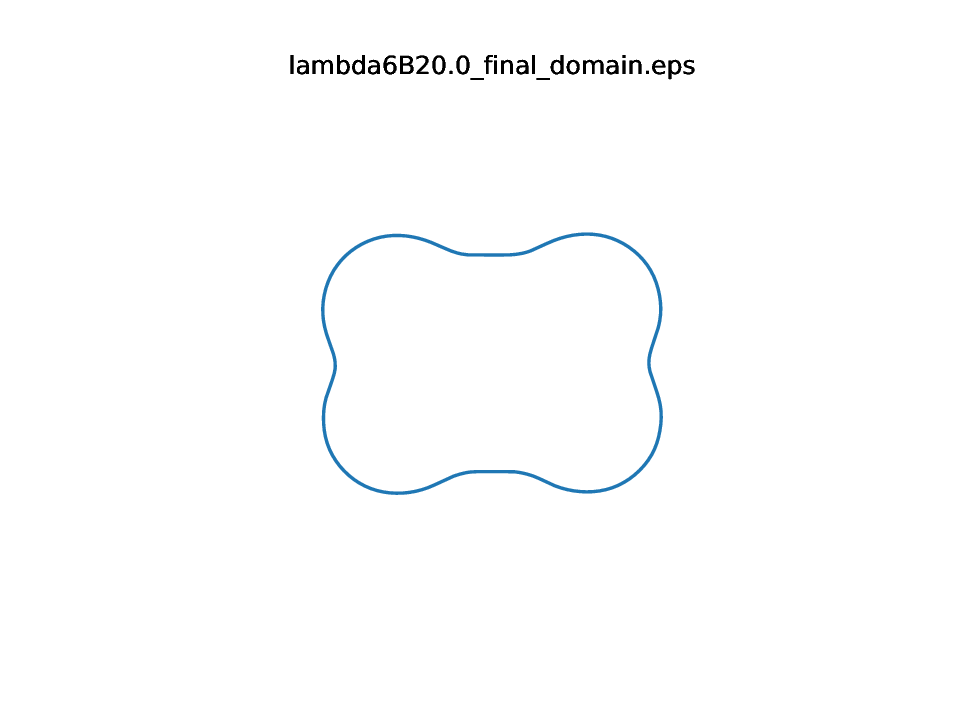}
 \put (30,-15) {$B=20.0$}
\end{overpic}
\begin{overpic}[trim={4cm 3cm 4cm 2.5cm},clip,scale=0.35]{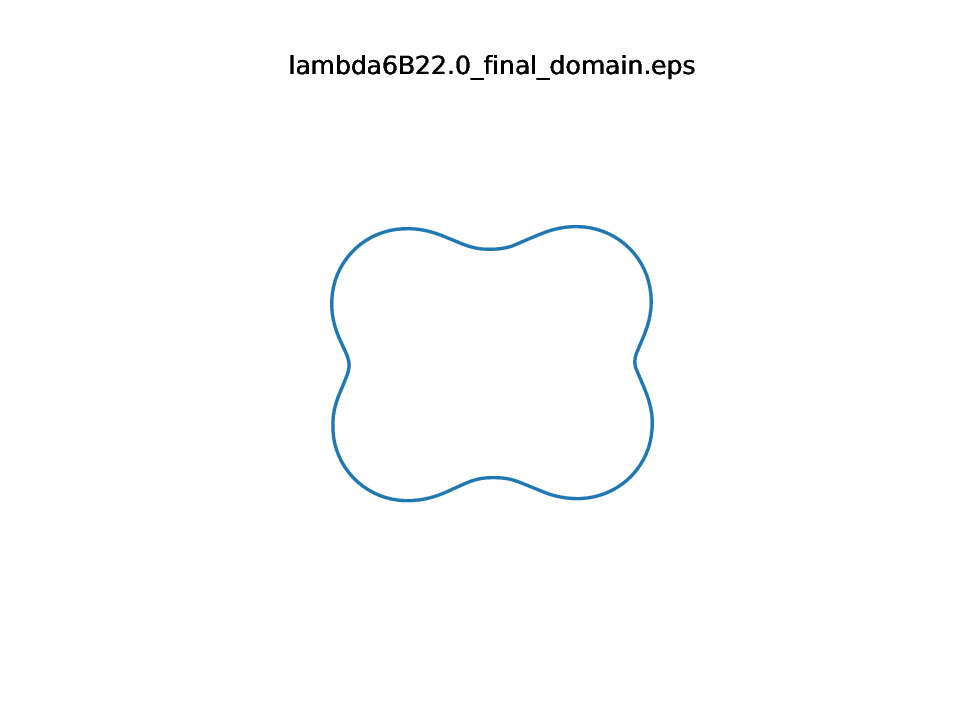}
 \put (30,-15) {$B=22.0$}
\end{overpic}
\begin{overpic}[trim={4cm 3cm 4cm 2.5cm},clip,scale=0.35]{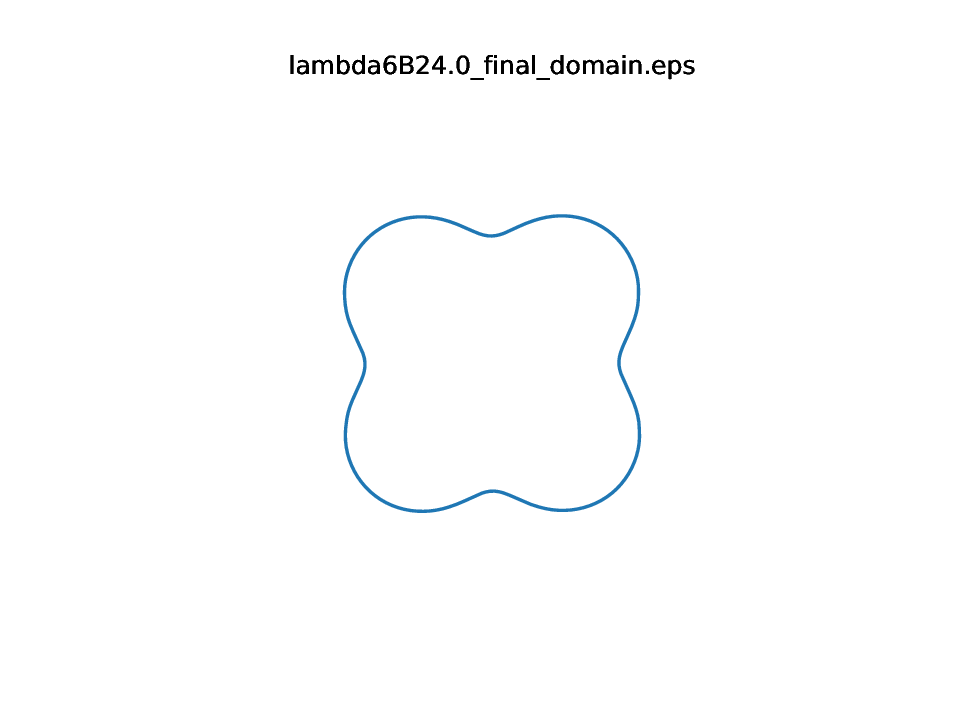}
 \put (30,-15) {$B=24.0$}
\end{overpic}
\begin{overpic}[trim={4cm 3cm 4cm 2.5cm},clip,scale=0.35]{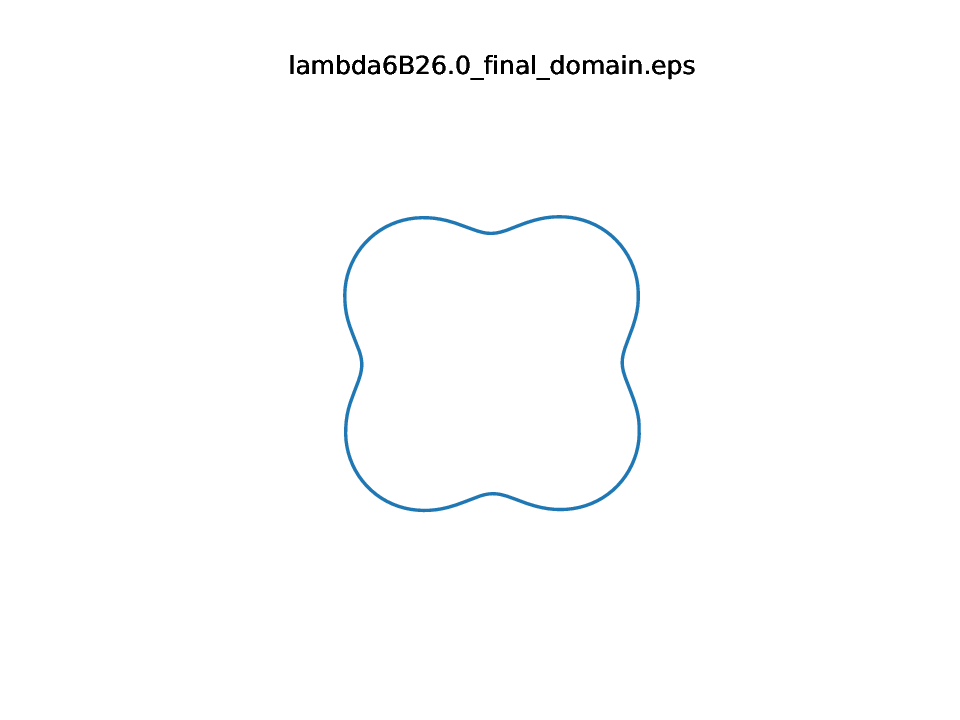}
 \put (30,-15) {$B=26.0$}
\end{overpic}
\vspace{1cm}

\begin{overpic}[trim={4cm 3cm 4cm 2.5cm},clip,scale=0.35]{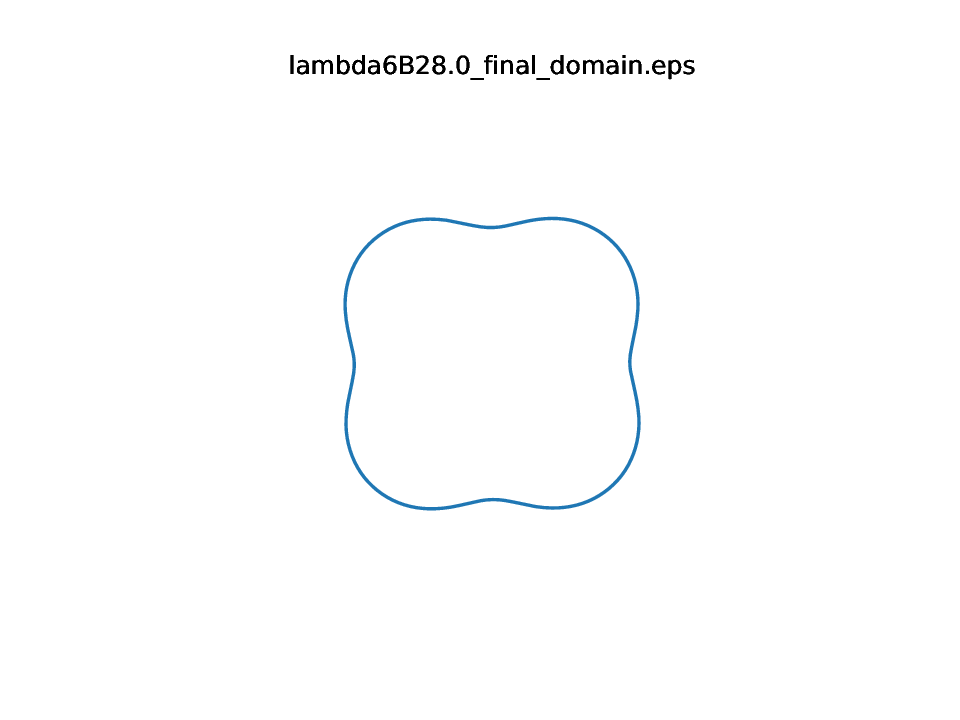}
 \put (30,-15) {$B=28.0$}
\end{overpic}
\begin{overpic}[trim={4cm 3cm 4cm 2.5cm},clip,scale=0.35]{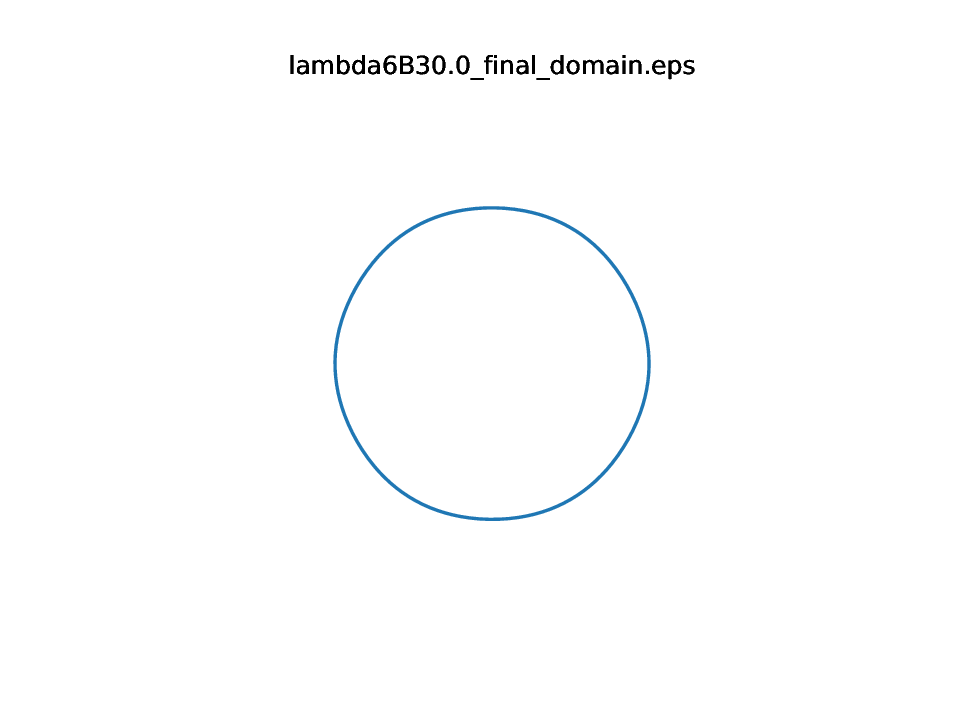}
 \put (30,-15) {$B=30.0$}
\end{overpic}
\begin{overpic}[trim={4cm 3cm 4cm 2.5cm},clip,scale=0.35]{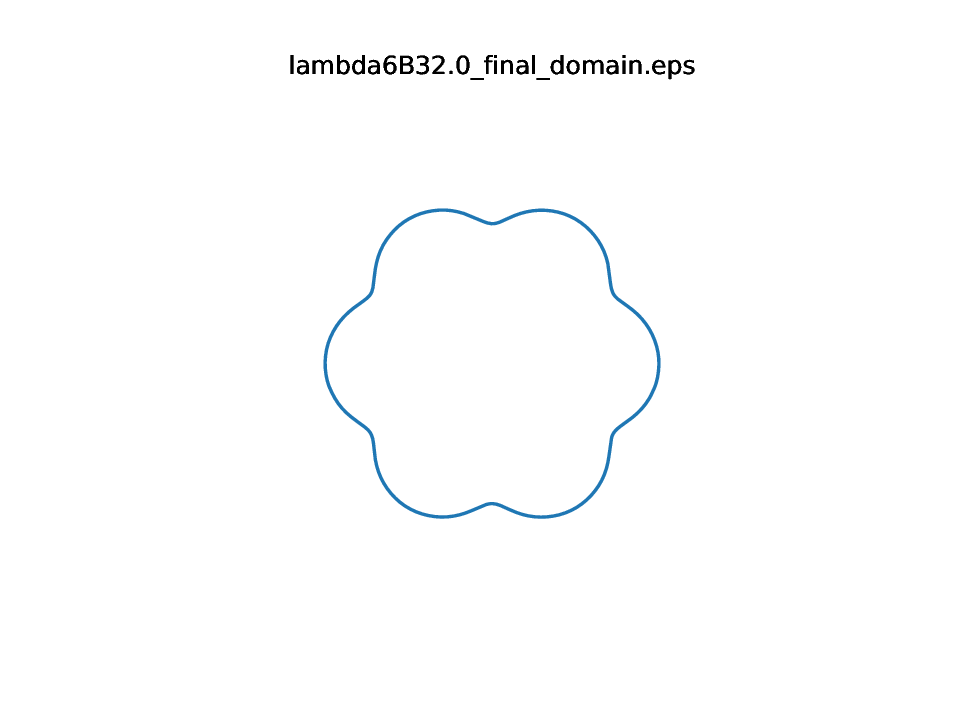}
 \put (30,-15) {$B=32.0$}
\end{overpic}
\begin{overpic}[trim={4cm 3cm 4cm 2.5cm},clip,scale=0.35]{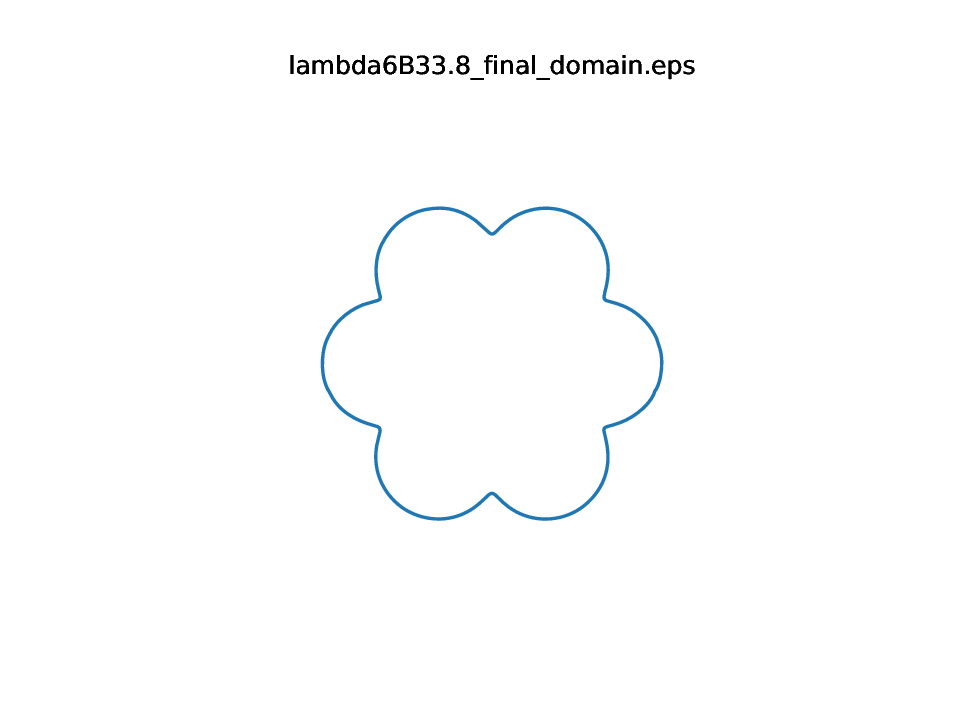}
 \put (30,-15) {$B=33.8$}
\end{overpic}
\vspace{1cm}

\caption{Final domains from minimization of $\lambda_6(\Omega, B)$ for various $B$.} \label{fig:lambda6_final_domain}
\end{figure}

\begin{figure}[h]
\begin{overpic}[trim={4cm 3cm 4cm 2.5cm},clip,scale=0.35]{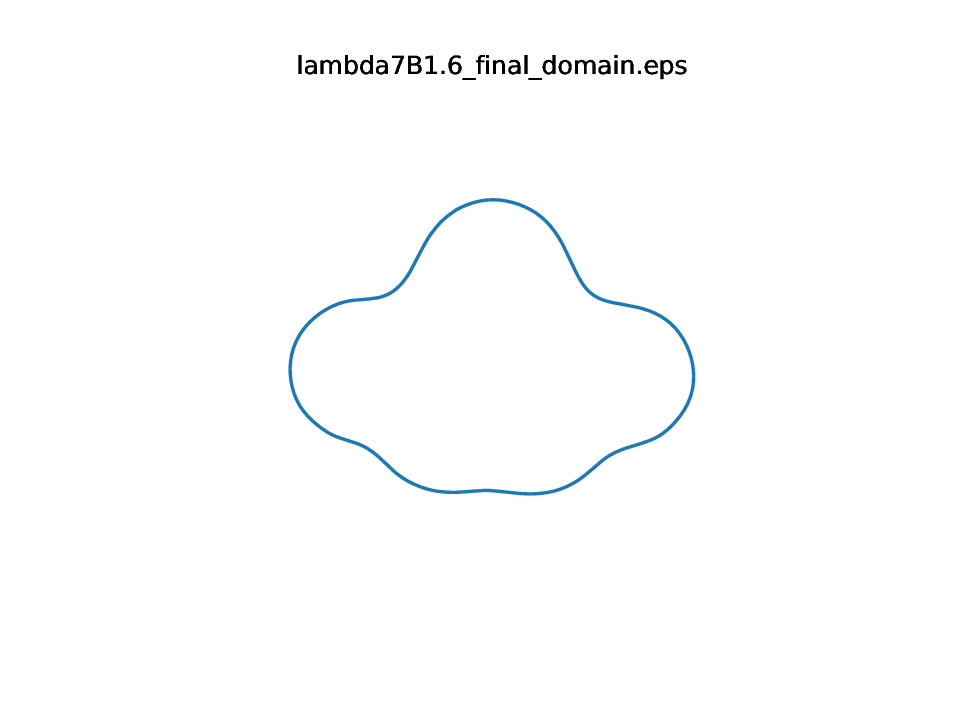}
 \put (32,-15) {$B=1.6$}
\end{overpic}
\begin{overpic}[trim={4cm 3cm 4cm 2.5cm},clip,scale=0.35]{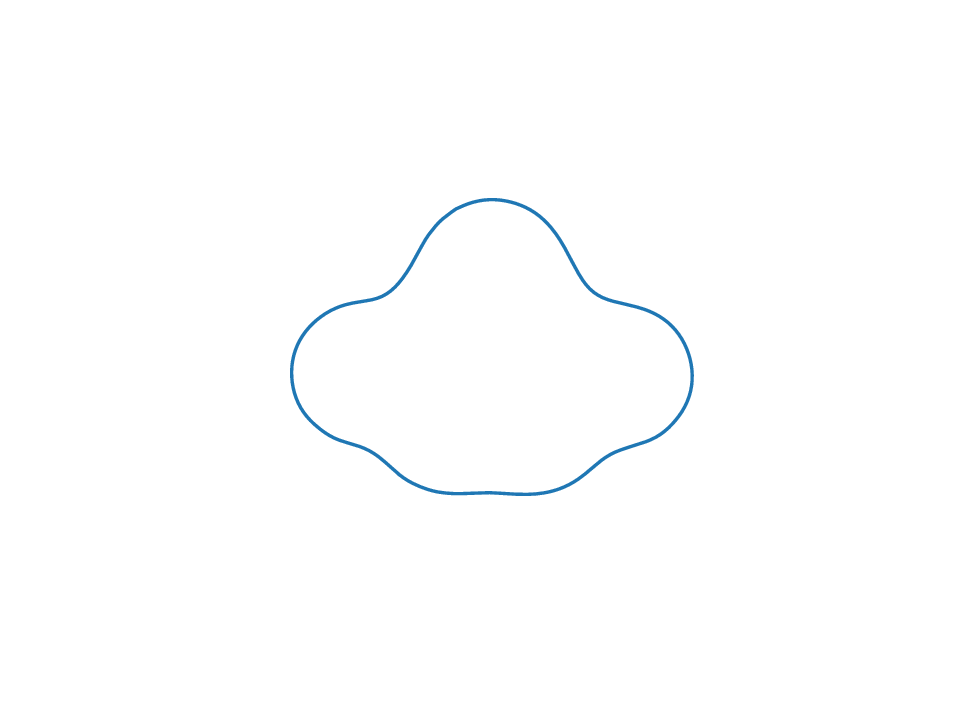}
 \put (32,-15) {$B=3.2$}
\end{overpic}
\begin{overpic}[trim={4cm 3cm 4cm 2.5cm},clip,scale=0.35]{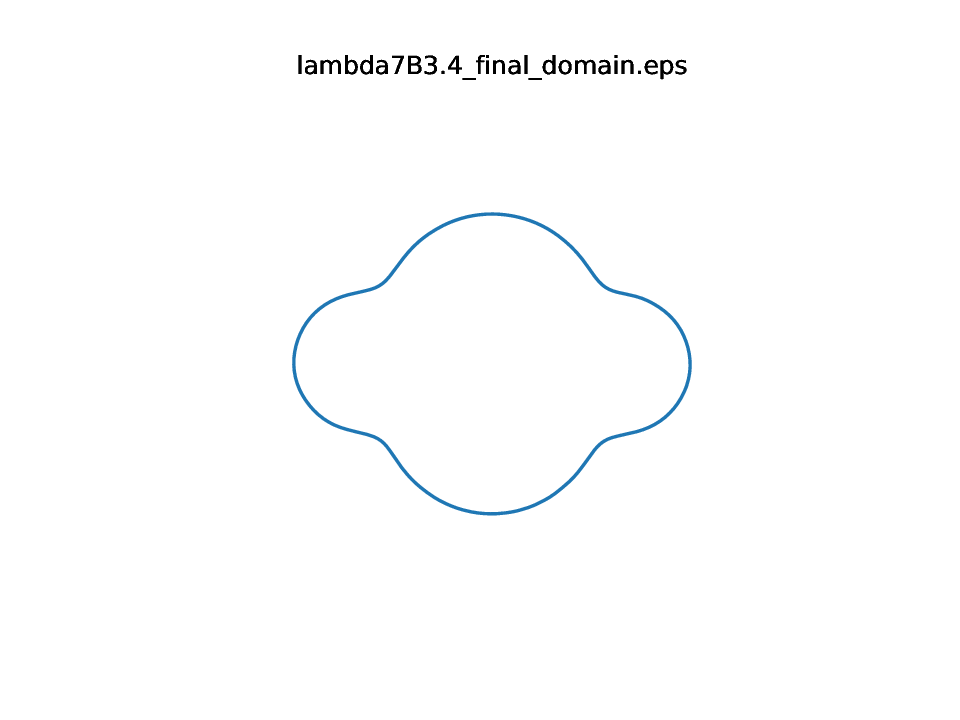}
 \put (32,-15) {$B=3.4$}
\end{overpic}
\begin{overpic}[trim={4cm 3cm 4cm 2.5cm},clip,scale=0.35]{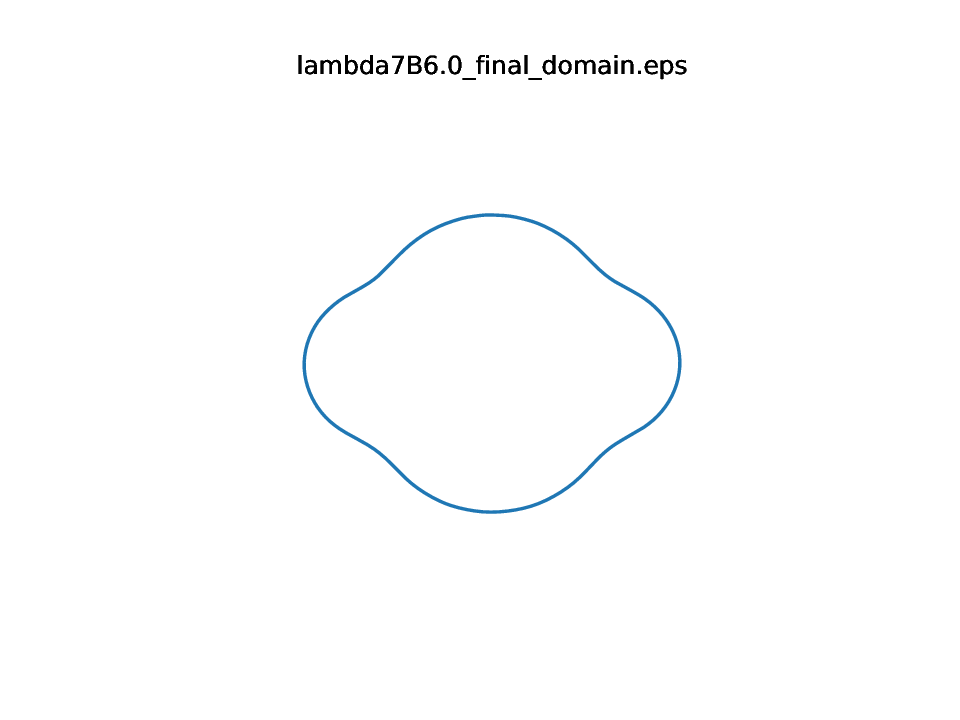}
 \put (32,-15) {$B=6.0$}
\end{overpic}
\begin{overpic}[trim={4cm 3cm 4cm 2.5cm},clip,scale=0.35]{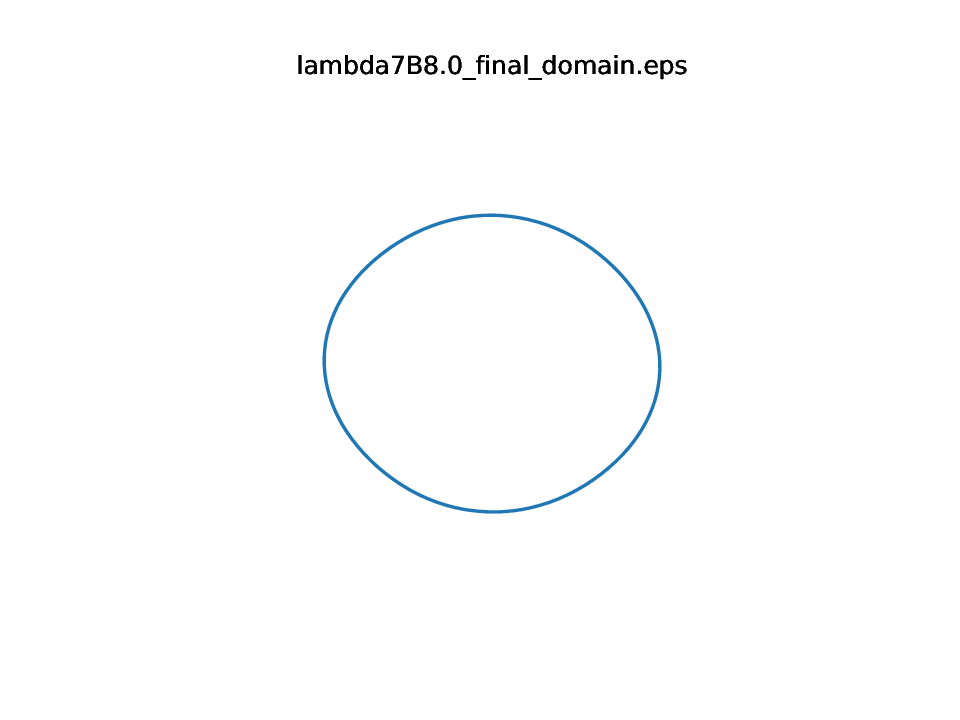}
 \put (32,-15) {$B=8.0$}
\end{overpic}
\vspace{1cm}

\begin{overpic}[trim={4cm 3cm 4cm 2.5cm},clip,scale=0.35]{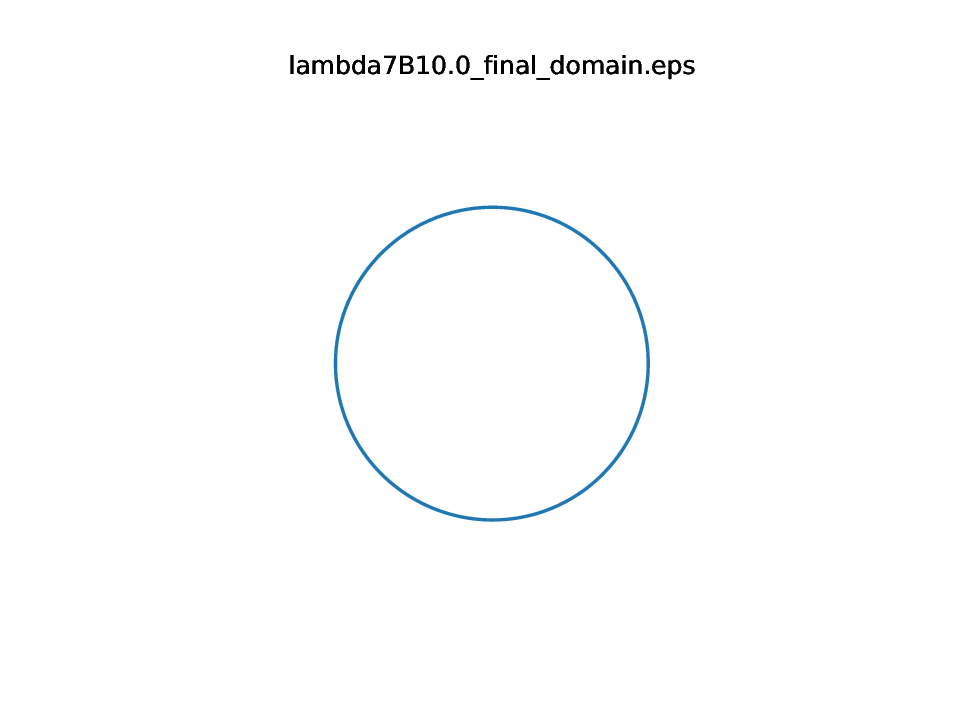}
 \put (30,-15) {$B=10.0$}
\end{overpic}
\begin{overpic}[trim={4cm 3cm 4cm 2.5cm},clip,scale=0.35]{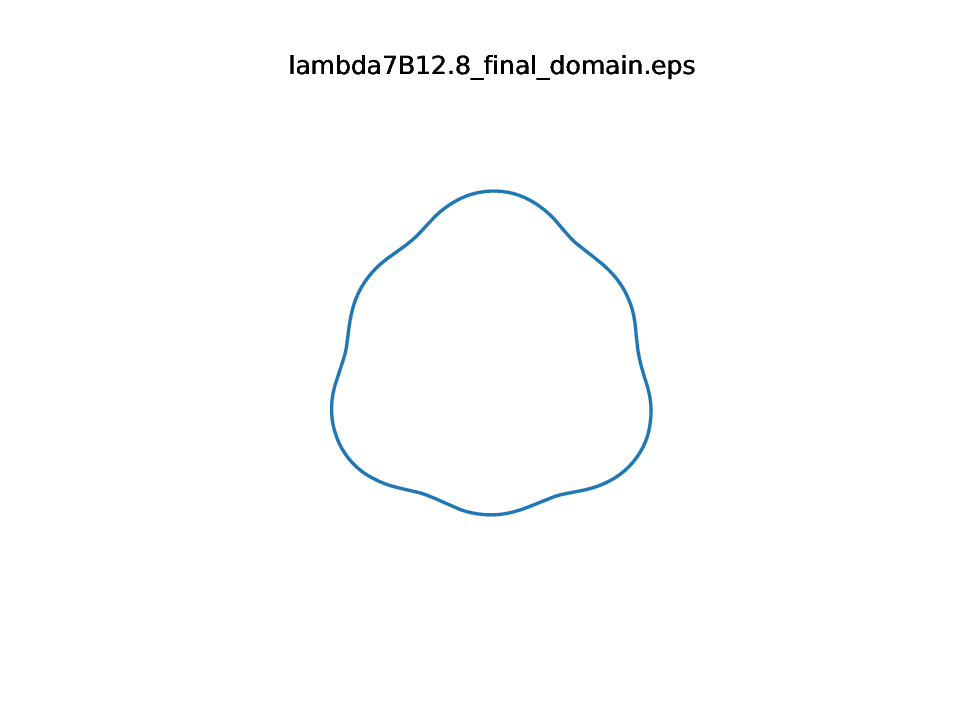}
 \put (30,-15) {$B=12.8$}
\end{overpic}
\begin{overpic}[trim={4cm 3cm 3.5cm 2.5cm},clip,scale=0.35]{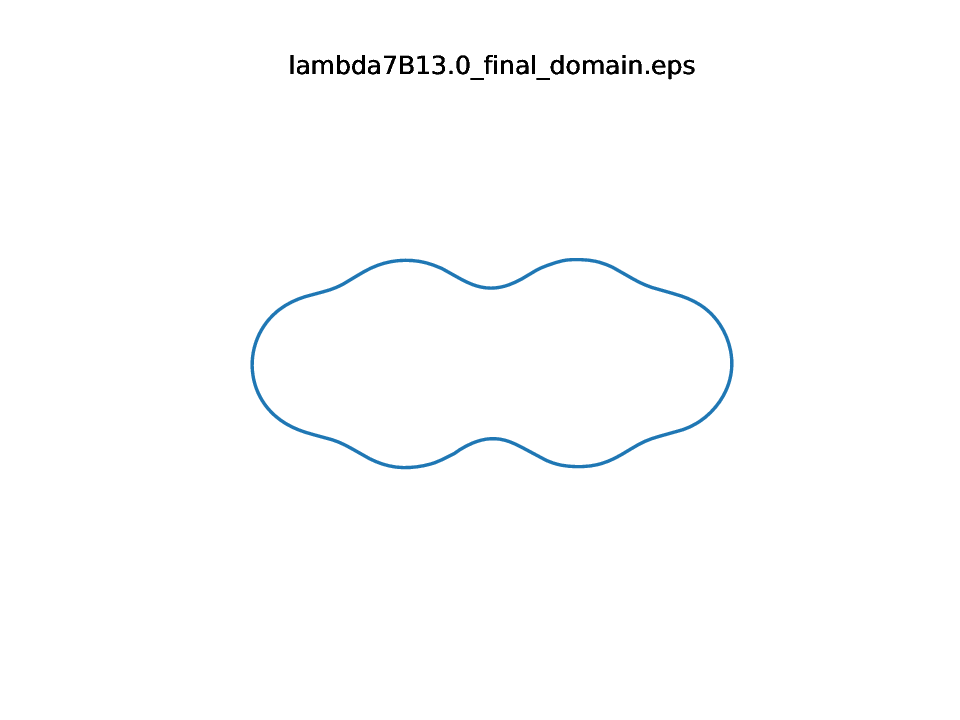}
 \put (30,-15) {$B=13.0$}
\end{overpic}
\begin{overpic}[trim={4cm 3cm 4cm 2.5cm},clip,scale=0.35]{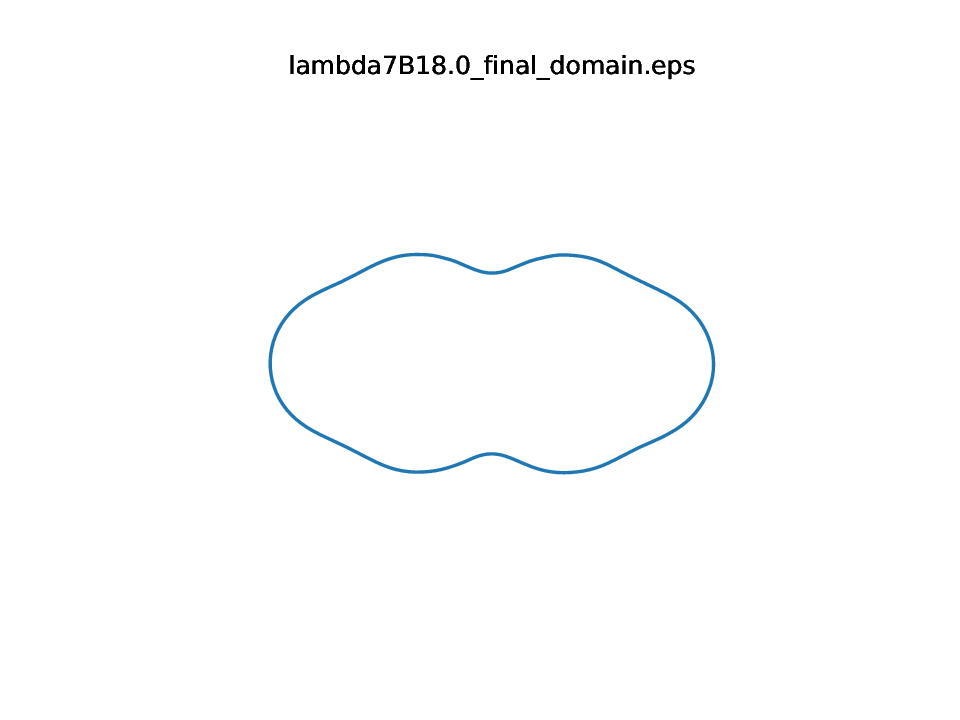}
 \put (30,-15) {$B=18.0$}
\end{overpic}
\begin{overpic}[trim={4cm 3cm 4cm 2.5cm},clip,scale=0.35]{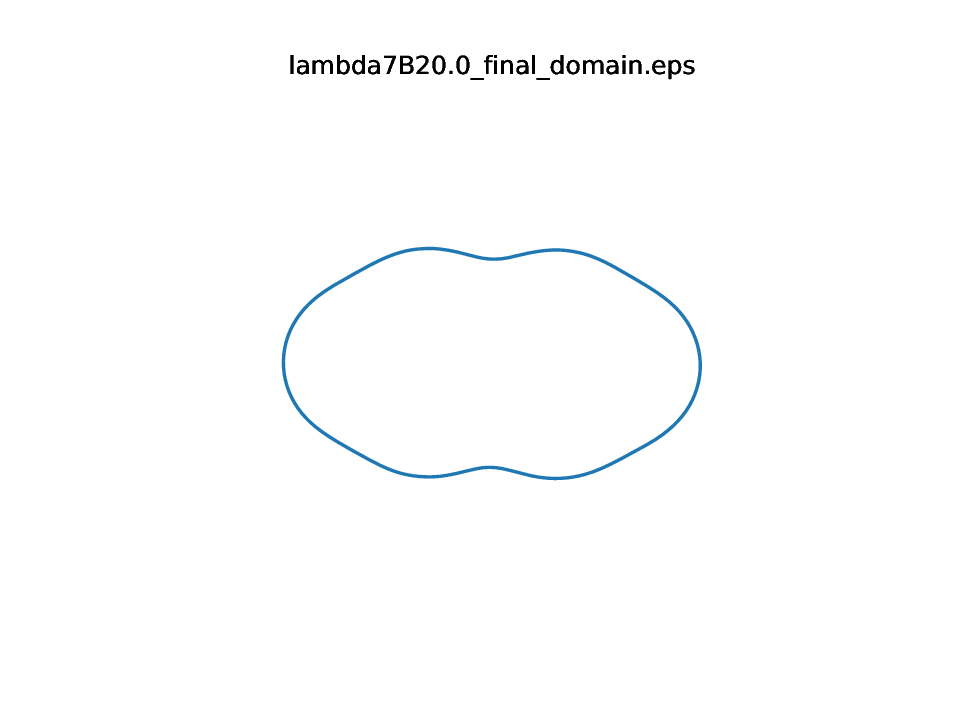}
 \put (30,-15) {$B=20.0$}
\end{overpic}
\vspace{1cm}

\begin{overpic}[trim={4cm 3cm 4cm 2.5cm},clip,scale=0.35]{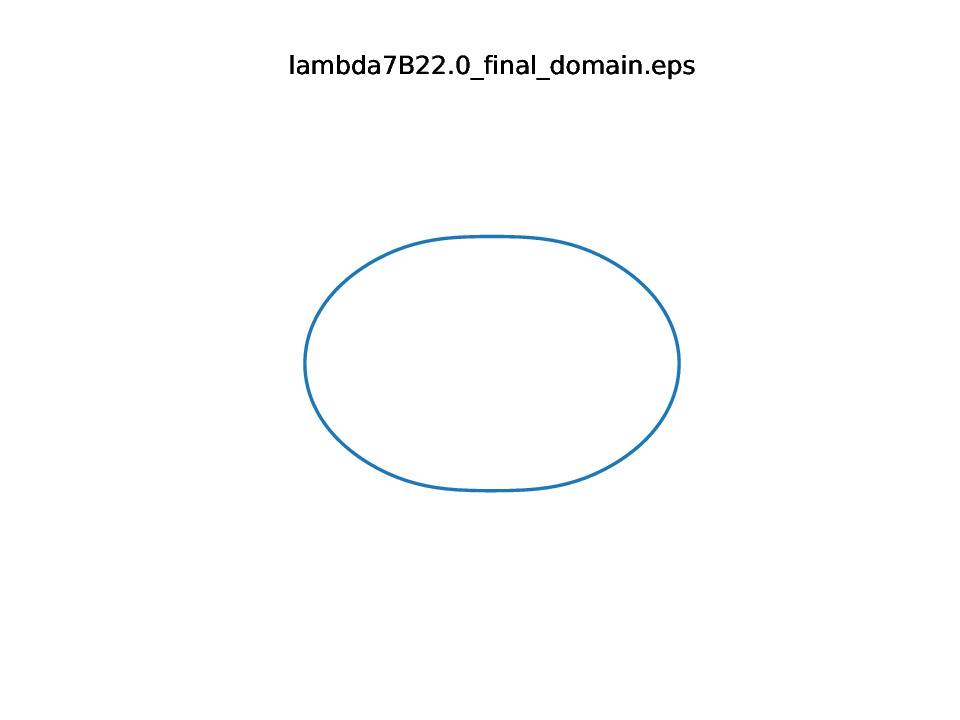}
 \put (30,-15) {$B=22.0$}
\end{overpic}
\begin{overpic}[trim={4cm 3cm 4cm 2.5cm},clip,scale=0.35]{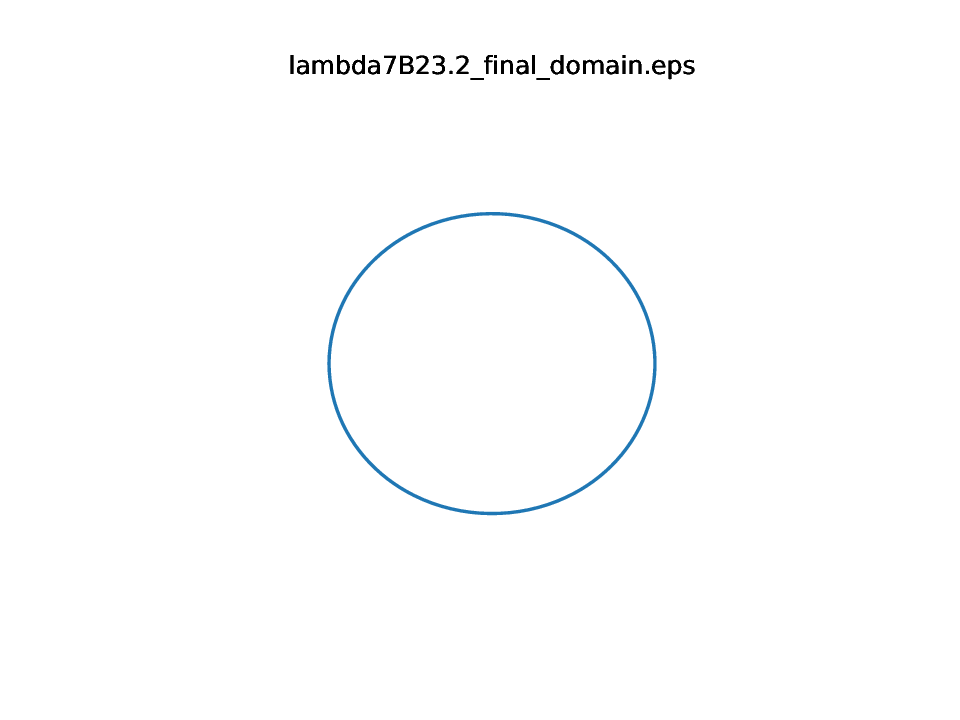}
 \put (30,-15) {$B=23.2$}
\end{overpic}
\begin{overpic}[trim={4cm 3cm 4cm 2.5cm},clip,scale=0.35]{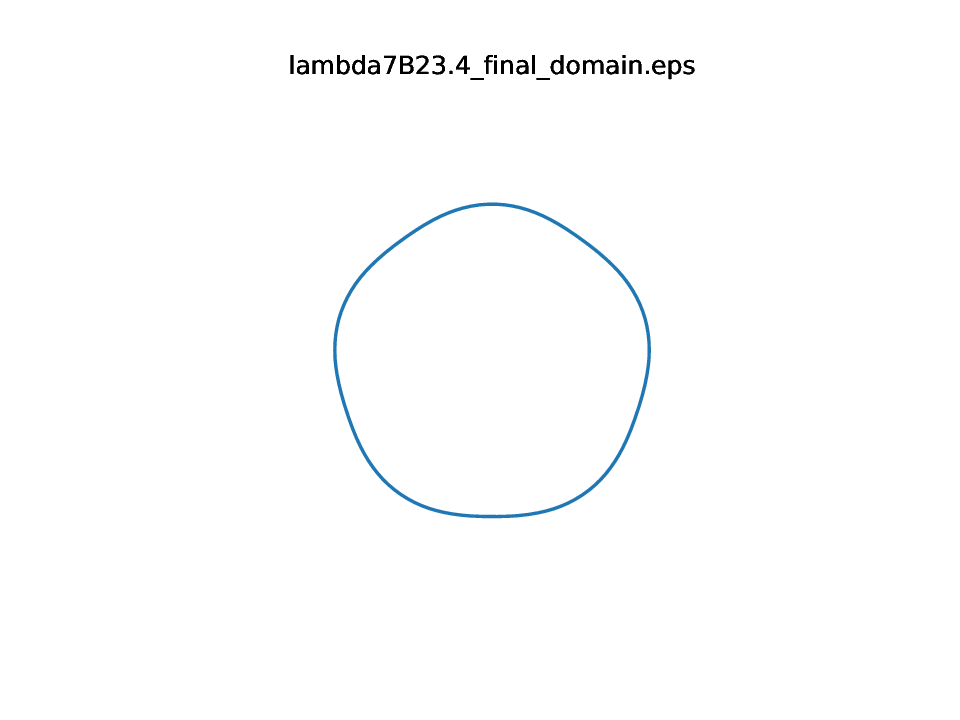}
 \put (30,-15) {$B=23.4$}
\end{overpic}
\begin{overpic}[trim={4cm 3cm 4cm 2.5cm},clip,scale=0.35]{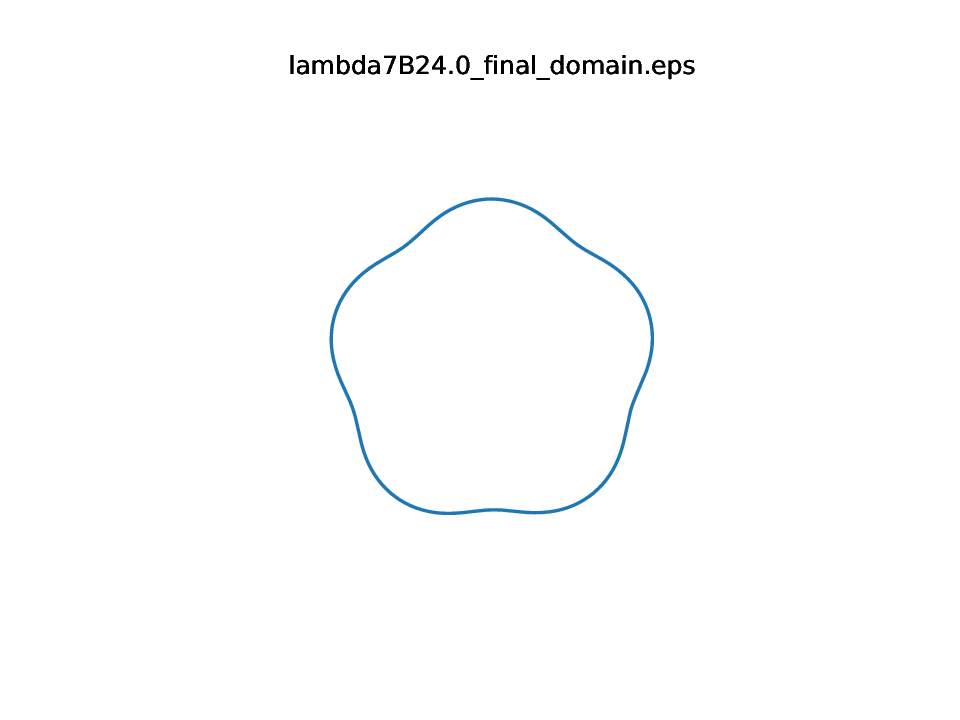}
 \put (30,-15) {$B=24.0$}
\end{overpic}
\begin{overpic}[trim={4cm 3cm 4cm 2.5cm},clip,scale=0.35]{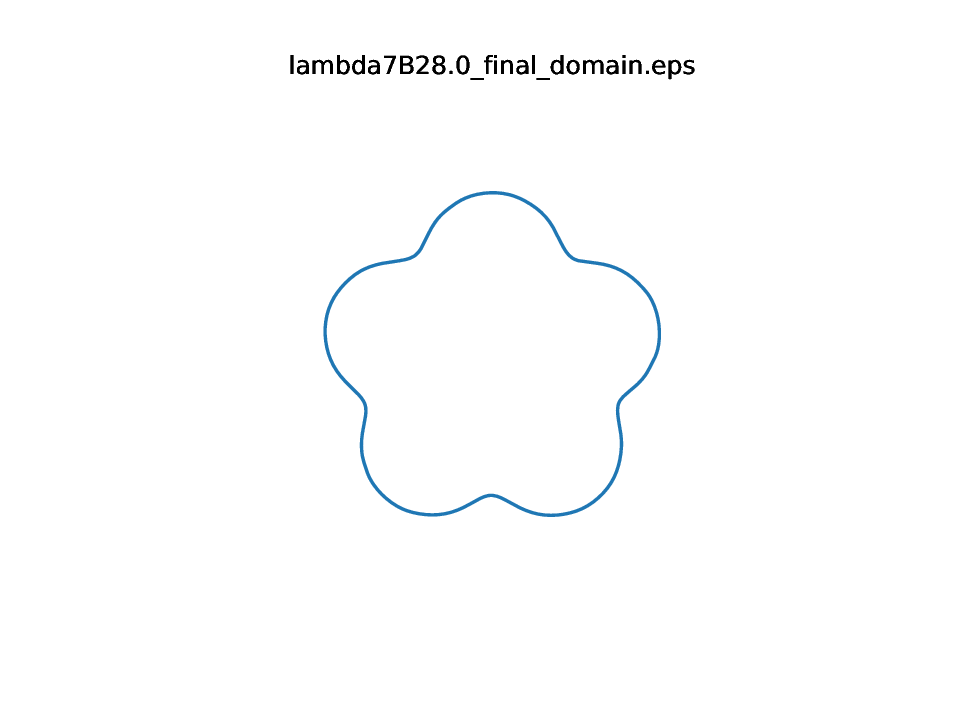}
 \put (30,-15) {$B=28.0$}
\end{overpic}
\vspace{1cm}

\begin{overpic}[trim={4cm 3cm 4cm 2.5cm},clip,scale=0.35]{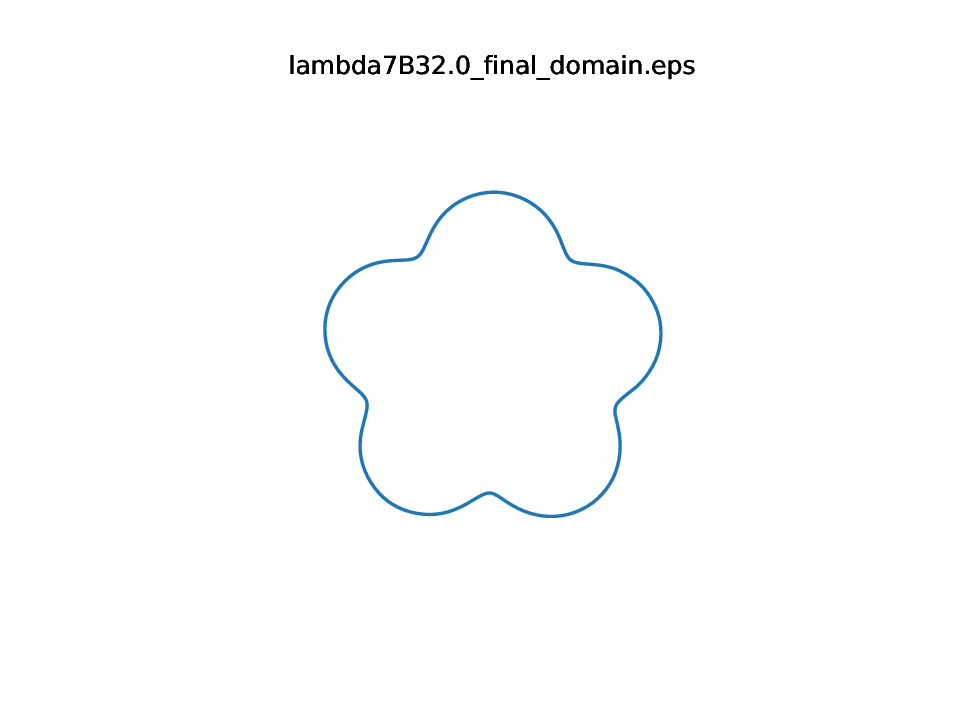}
 \put (30,-15) {$B=32.0$}
\end{overpic}
\begin{overpic}[trim={4cm 3cm 4cm 2.5cm},clip,scale=0.35]{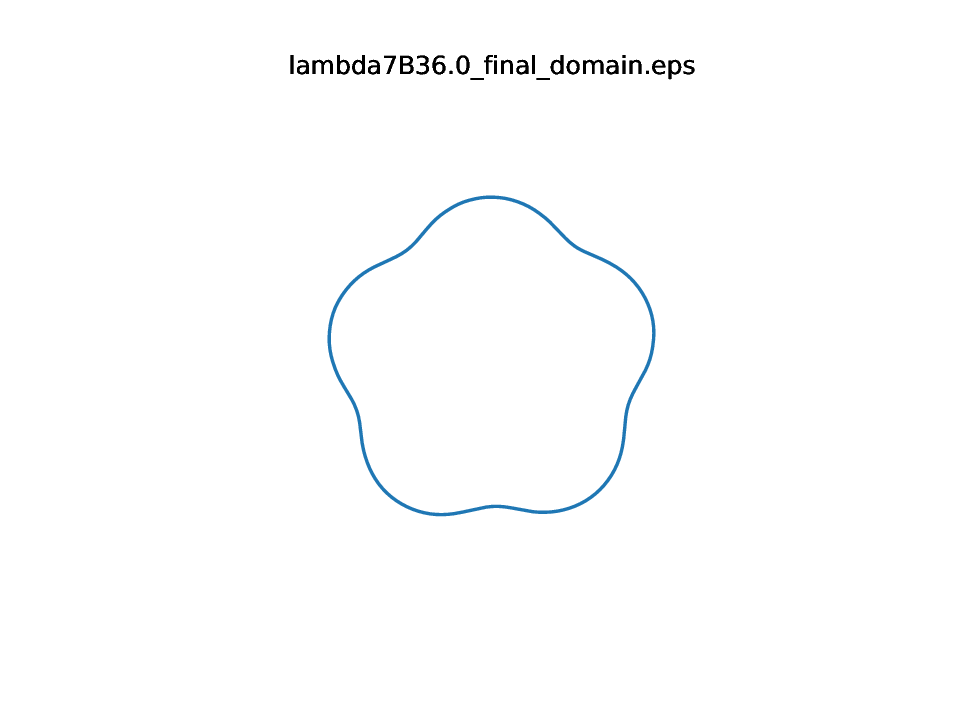}
 \put (30,-15) {$B=36.0$}
\end{overpic}
\begin{overpic}[trim={4cm 3cm 4cm 2.5cm},clip,scale=0.35]{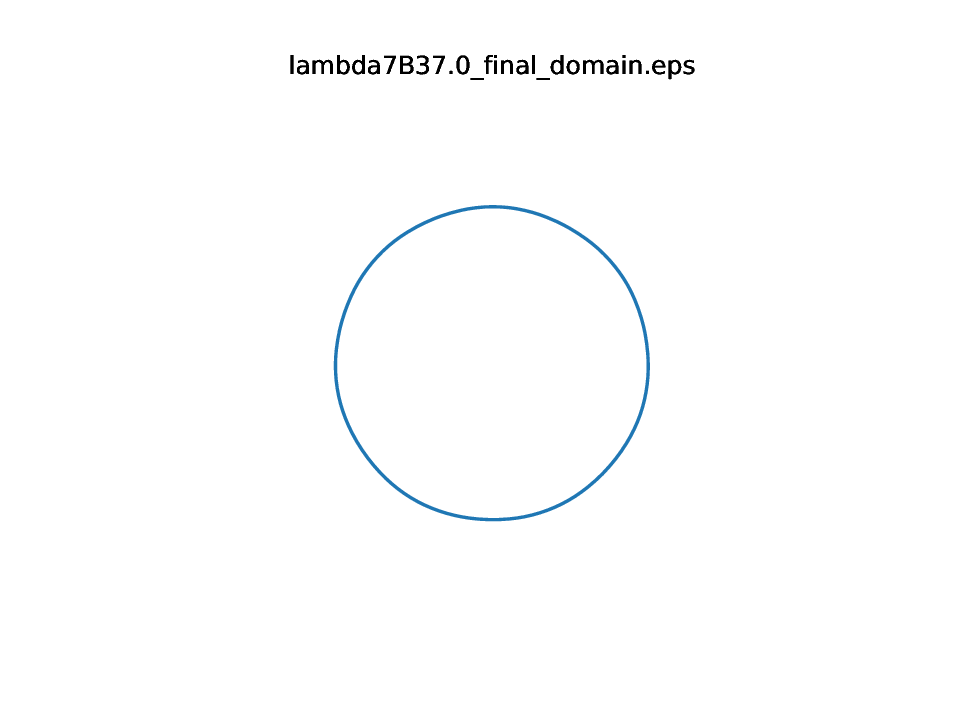}
 \put (30,-15) {$B=37.0$}
\end{overpic}
\begin{overpic}[trim={4cm 3cm 4cm 2.5cm},clip,scale=0.35]{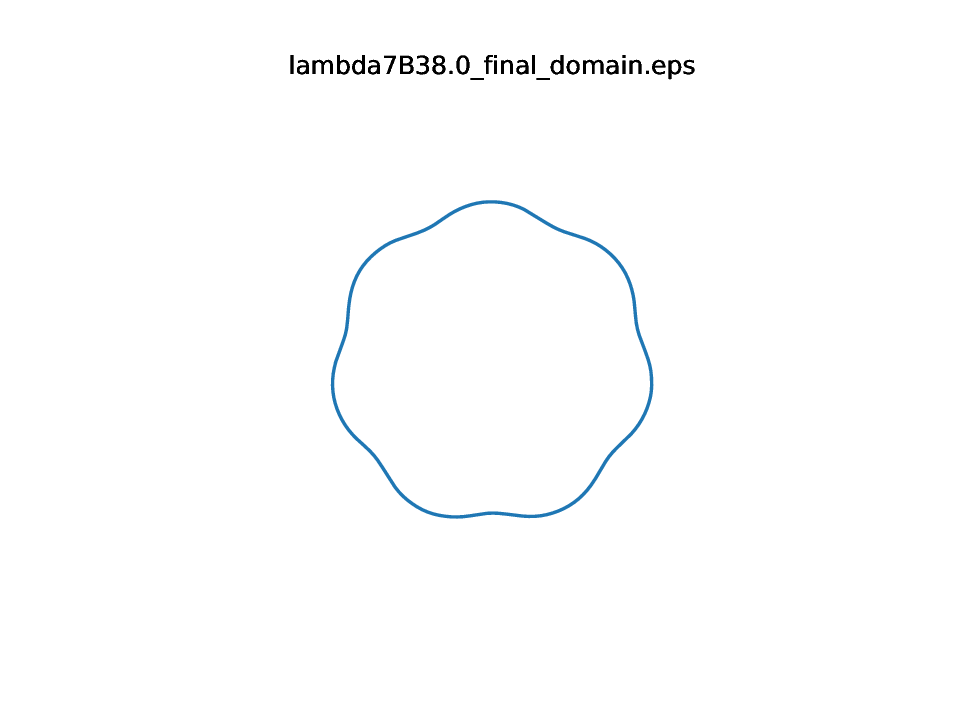}
 \put (30,-15) {$B=38.0$}
\end{overpic}
\begin{overpic}[trim={4cm 3cm 4cm 2.5cm},clip,scale=0.35]{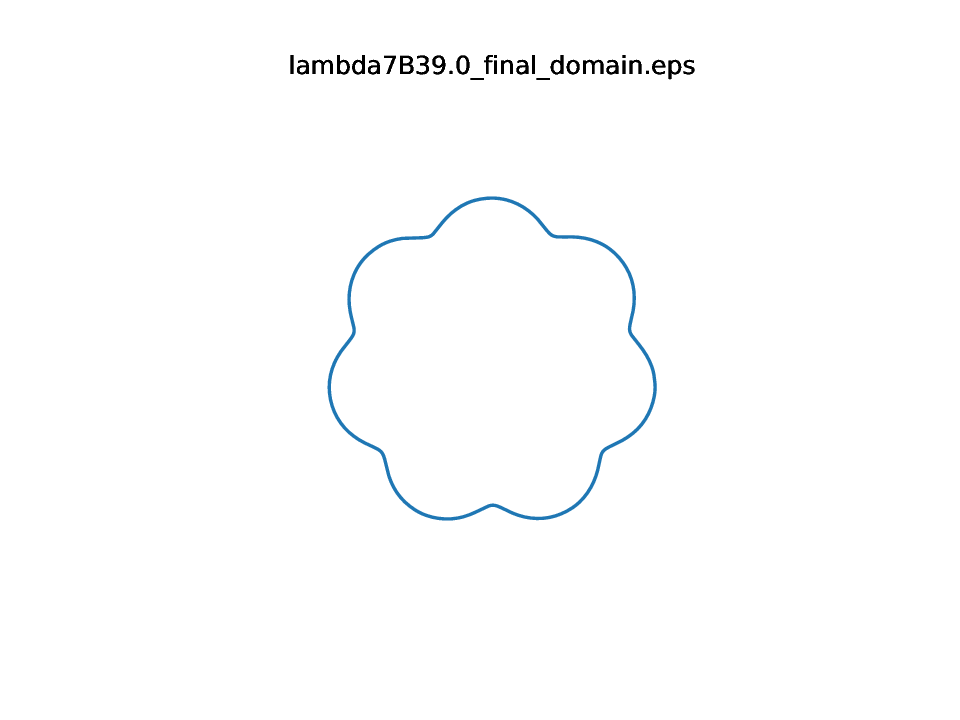}
 \put (30,-15) {$B=39.0$}
\end{overpic}
\vspace{1cm}

\caption{Final domains from minimization of $\lambda_7(\Omega, B)$ for various $B$.} \label{fig:lambda7_final_domain}
\end{figure}

\begin{figure}[h]
\includegraphics[scale=0.5]{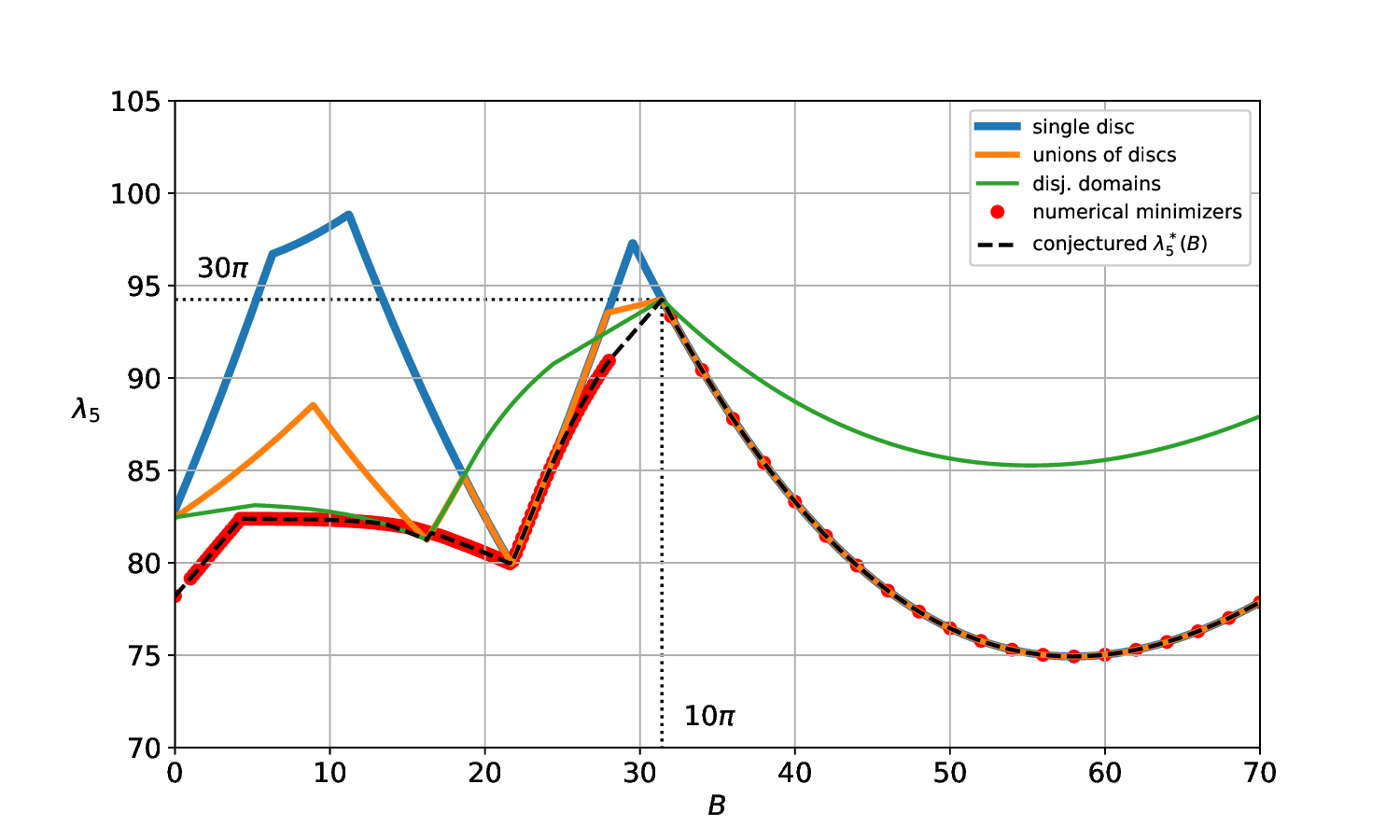}

\includegraphics[scale=0.5]{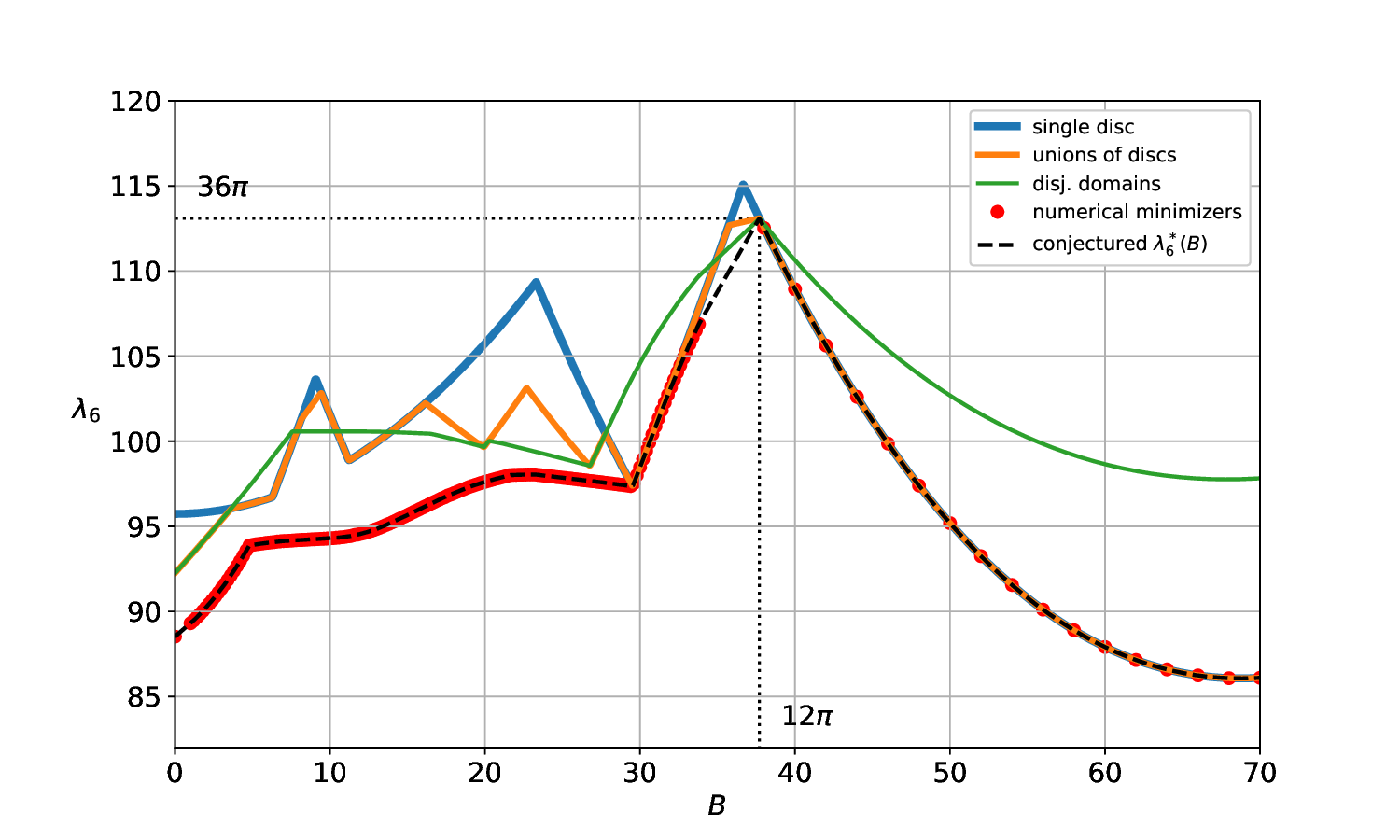}

\includegraphics[scale=0.5]{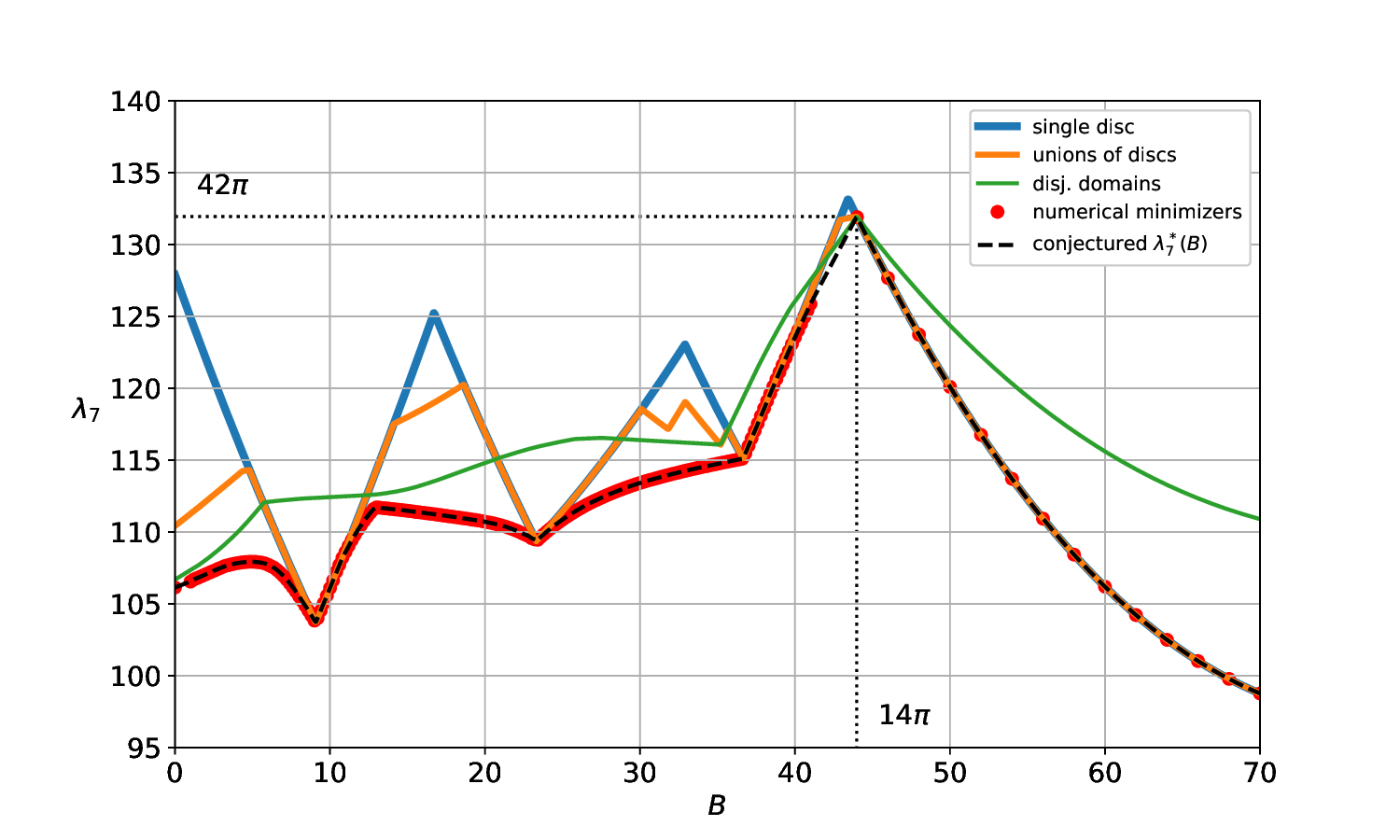}

\caption{Comparisons of $\lambda_5$, $\lambda_6$, $\lambda_7$ of the numerical minimizers (red) with $\lambda_5$, $\lambda_6$, $\lambda_7$ of a single disk (blue), the minimal $\lambda_5$, $\lambda_6$, $\lambda_7$ obtained over arbitrary disjoint unions of disks (orange) and the minimal $\lambda_5$, $\lambda_6$, $\lambda_7$ obtained over disjoint unions of previous minimizers (green) over different field strengths $B$. } \label{fig:compcurve_lambda567}
\end{figure}

\clearpage

\appendix

\section{Disjoint unions of disks}

In this section, we consider the minimization problem \eqref{eq:lambda_n_min_problem} restricted to domains that are disjoint unions of disks, i.e.\ for $k\in\mathbb{N}\cup \{ +\infty\}$ we consider the problem
\begin{align} \label{eq:min_problem_disjoint_union_of_disks}
\min_{\substack{\Omega \text{ disjoint union of} \\ k \text{ or less disks}, \\ |\Omega|=1}} \lambda_n(\Omega, B).
\end{align}
Setting $k=+\infty$ shall be interpreted as posing no restriction on the number of disks, so that we allow any countable number of disks. Clearly, solving problem \eqref{eq:min_problem_disjoint_union_of_disks} requires optimization over the allowed number of disks and their radii. We want to elaborate how this (partly combinatorical) problem can be solved. In the following, let $D$ always denote a disk with $|D|=1$ and let $D_R$ denote a disk of radius $R$.

We remark that we consider problem \eqref{eq:min_problem_disjoint_union_of_disks} because disks are among the very few domains for which the eigenvalue problem can be explicitly solved and thus the spectrum of any disjoint union of disks is explicitly known. We briefly recall the solution of the eigenvalue problem on a disk of radius $R$. 

Eigenvalues of the magnetic Dirichlet Laplacian on $D_R$ are obtained by changing coordinates into polar coordinates and separation of variables. One finds that eigenvalues are related to roots of confluent hypergeometric functions. Specifically, for any fixed $l\in\mathbb{Z}$, the implicit equation
\begin{align}
M \left(\frac{1}{2} \left( l+ |l|+1- \frac{ \lambda}{B}    \right)  ,|l|+1, \frac{BR^2}{2}  \right) = 0, \label{eq:magnetic_disk_eigval_transcendental}
\end{align}
where $M(a,b,z)$ denotes Kummer's confluent hypergeometric function, admits a discrete set of solutions in the parameter $\lambda$. The spectrum $\{ \lambda_n(D,B) \}_{n=1}^\infty$ of the magnetic Dirichlet Laplacian on $D_R$ is then gained by taking the union of all solutions $\lambda$ over all $l\in\mathbb{Z}$ and sorting them in increasing order (potentially counting multiplicities). Note that since equation \eqref{eq:magnetic_disk_eigval_transcendental} can be numerically solved by a simple root search algorithm, the entire spectrum of the disk at any field strength can be computed to very high accuracy.


Let us now come back to \eqref{eq:min_problem_disjoint_union_of_disks}. Suppose $\Omega$ consists of finitely many, say $k$, disks. If $k=1$, we have nothing to optimize in view of \eqref{eq:min_problem_disjoint_union_of_disks}. The minimal $n$-th eigenvalue for one disk of measure one \textit{is} its $n$-th eigenvalue. We set
\begin{align}
\Lambda_{n,1} (B) = \lambda_n(D,B).
\end{align}
Then, $\Lambda_{n,1} (B)$ is the minimum of \eqref{eq:min_problem_disjoint_union_of_disks} for $k=1$.

Now let $k >1$. Assume $\Omega = D_1 \cup ... \cup D_k$ where the $D_j$ are disjoint disks. The set of eigenvalues of $\Omega = D_1 \cup ... \cup D_k$ is the union of all eigenvalues of the $D_j$. This means that there exist sequences $(i_k)_{k\in\mathbb{N}}$ and $(j_k)_{k\in\mathbb{N}}$ such that
\begin{align}
\lambda_k(\Omega,B) = \lambda_{i_k}(D_{j_k},B).
\end{align}
The scaling identity allows us to write for each disk $D_j$ and any $i$
\begin{align} \label{eq:lambda_i_D-j_rescal}
\lambda_i(D_j,B) = \frac{1}{|D_j|} \lambda_i\left(|D_j|^{-\frac{1}{2}} D_j, |D_j| B\right) = \frac{1}{|D_j|} \lambda_i\left(D, |D_j| B\right).
\end{align}
The spectrum of $\Omega$ is hence entirely determined by the spectrum of the normalized disk $D$ at different field strengths.

For $j=1,..., k$ we define 
\begin{align} \label{eq:n-j_cases}
n_j := \begin{cases}
 \max\{ i_k \,: \, j_k=j \text{ and } k\leq n \}, & \text{ if there exists }k\leq n \text{ with } j_k=j, \\
 0, & \text{ else}.
 \end{cases}
\end{align}
The number $n_j$ tells us how many of the $n$ lowest eigenvalues of $\Omega$ are associated with the disk $D_j$. Note that $n_1+...+n_k = n$, which means that $(n_1,..., n_k)$ is a partition of $n$. For the eigenvalue $\lambda_n(\Omega,B)$ now holds
\begin{align}\label{eq:lambda_n-equal_maximum}
\lambda_n(\Omega,B) = \lambda_{i_n}(D_{j_n},B) = \max\left\{ \lambda_{n_1}(D_{1},B),...,  \lambda_{n_k}(D_{k},B) \right\},
\end{align}
where we set $\lambda_0(D_j,B):=0$.
In fact, it holds
\begin{align} \label{eq:lambda_n-equal_minimum_maximum}
\lambda_n(\Omega,B) =\min_{\substack{(n_j')_{j=1}^k\in N_{n,k} }} \max\{ \lambda_{n_1'}(D_{1},B),...,  \lambda_{n_k'}(D_{k},B) \}
\end{align}
where
\begin{align}
N_{n,k} &:= \left\{ (n_j)_{j=1}^k \in (\mathbb{N}\cup \{ 0\})^k \, : \,  \sum_{j=1}^k n_j = n \right\}
\end{align}
is the set of partitions of $n$. To see this, it is only necessary to show ''$\leq$'' in \eqref{eq:lambda_n-equal_minimum_maximum}, as the inequality with ''$\geq$'' follows immediately from \eqref{eq:lambda_n-equal_maximum}. Indeed, for any $(n_j')_{j=1}^k \in N_{n,k}$ with $(n_j')_{j=1}^k \neq (n_j)_{j=1}^k$ there exists some $n_j' \geq n_j + 1$. But given the definition of the $n_j$, we must have $\lambda_{n_j+1}(D_j,B) \geq  \lambda_n(\Omega, B)$ for any $j$ and therefore 
\begin{align}
\max\{ \lambda_{n_1'}(D_{1},B),...,  \lambda_{n_k'}(D_{k},B) \} \geq \lambda_{n_j'}(D_{j},B) \geq \lambda_n(\Omega,B).
\end{align}
This shows \eqref{eq:lambda_n-equal_minimum_maximum}.

Expression \eqref{eq:lambda_n-equal_minimum_maximum} for $\lambda_n(\Omega,B)$ has the advantage that it does not depend on the explicit numbers $n_j$. It only depends on the size of the disks $D_j$. Setting $t_j:=|D_j|$ for $j=1,...,k$ and using \eqref{eq:lambda_i_D-j_rescal} yields the expression
\begin{align} \label{eq:lambda_n_minmax_tj}
\lambda_n(\Omega,B) = \min_{\substack{(n_j)_{j=1}^k\in N_{n,k} }} \max \left\lbrace \frac{\lambda_{n_1}(D,t_1 B)}{t_1}, ... , \frac{\lambda_{n_k}(D,t_k B)}{t_k} \right\rbrace.
\end{align} 
If $\Omega$ is minimizes $\lambda_n(\Omega,B)$, then the measures $t_j$ of the disks $D_j$ must minimize \eqref{eq:lambda_n_minmax_tj}, so we arrive at the minimization problem
\begin{align} \label{eq:lambda_n_opt_disks_k_disks}
\Lambda_{n,k}(B)=\min_{\substack{(n_j)_{j=1}^k\in N_{n,k}  \\ (t_j)_{j=1}^k \in T_k  }} \max \left\lbrace \frac{\lambda_{n_1}(D,t_1 B)}{t_1}, ... , \frac{\lambda_{n_k}(D,t_k B)}{t_k} \right\rbrace.
\end{align}
where we set
\begin{align}
T_k &:= \left\{ (t_j)_{j=1}^k \in [0,1]^k \, : \,  \sum_{j=1}^k t_j = 1 \right\}.
\end{align}
Formally, the expressions $\lambda_n(D,tB)/t$ in \eqref{eq:lambda_n_opt_disks_k_disks} is not defined when $t=0$, but we think of it continued in $t=0$ by $+\infty$ if $n\geq 1$ and by $0$ if $n=0$. From continuity of $t\mapsto \lambda_n(D,tB)/t$ and the compactness of $T_k \subset [0,1]^k$, it can be shown that the min-max-problem \eqref{eq:lambda_n_opt_disks_k_disks} admits a solution.

If $(n_j^*)_{j=1}^k,(t_j^*)_{j=1}^k$ denotes a minimizer of \eqref{eq:lambda_n_opt_disks_k_disks}, we define
\begin{align} \label{eq:omega_k_star_construction}
\Omega_k^* = D_1 \cup .. \cup D_k, \qquad D_j = \begin{cases}
(t_j^*)^{\frac{1}{2}} D+a_j, & t_j^* >0, \\
\varnothing, & t_j^* = 0,
\end{cases}
\end{align}
where $a_j$ are arbitrary translations such that all $D_j$ are pairwise disjoint. The domain $\Omega_k^*$ consists of $k$ or less disjoint disks. By construction,
\begin{align} \label{eq:lambda_n-k_disks_optimale}
\lambda_n(\Omega,B) \geq \lambda_n(\Omega_k^*, B ) = \Lambda_{n,k}(B)
\end{align}
for any $\Omega$ with $|\Omega|=1$ consisting of $k$ disjoint disks. However, by setting $n_k=0$ and $t_k=0$ in the min-max-problem \eqref{eq:lambda_n_opt_disks_k_disks} for $k$ disks it is seen that \eqref{eq:lambda_n_opt_disks_k_disks} includes the corresponding min-max-problem \eqref{eq:lambda_n_opt_disks_k_disks} for $k-1$ disks. By induction, \eqref{eq:lambda_n-k_disks_optimale} is true for any $\Omega$ with $|\Omega|=1$ consisting of $k$ or less disjoint disks. In other words, the domain $\Omega_k^*$ solves \eqref{eq:min_problem_disjoint_union_of_disks}. It is a minimizer of $\lambda_n(\Omega,B)$ among domains consisting of $k$ or less disjoint disks with total measure equal to one. The minimum of problem \eqref{eq:min_problem_disjoint_union_of_disks} is precisely $\Lambda_{n,k}(B)$ and \eqref{eq:omega_k_star_construction} gives a minimizer.

We want to stress the fact that \eqref{eq:min_problem_disjoint_union_of_disks} is not necessarily minimized by a domain consisting of exactly $k$ disks. This manifests itself in \eqref{eq:lambda_n_opt_disks_k_disks} when the minimizer $(n_j^*)_{j=1}^k,(t_j^*)_{j=1}^k$ satisfies $t_{j'}^*=0$ for some $j'$. This requires $n_{j'}^* = 0$ and can be intepreted in the following way. 

Instead of $\Omega_k^*=D_1 \cup ...\cup D_k$ with measures $t_j=|D_j| \geq 0$ consider $\Omega' =D_1' \cup ...\cup D_k'$ with measures $t_j'=|D_j'|>0$ for all $j$. If $t_{j'}'$ is very close to zero for some $j'$, we find that $\lambda_1(D_{j'}',B) > \lambda_n(\Omega',B)$. This implies that any function in the spectral subspace spanned by the eigenfunctions to the eigenvalues $\lambda_1(\Omega',B),...,\lambda_n(\Omega',B)$ is supported on $\Omega' \setminus D_{j'}$. By the variational principle, we conclude that if we discard $D_{j'}$ and blow up all other disks such that the resulting domain has measure one, the $n$-th eigenvalue will decrease, i.e.
\begin{align} \label{eq:discard_disk}
\lambda_n\left(\frac{\Omega' \setminus D_j}{|\Omega' \setminus D_j|^\frac{1}{2}},B \right)\leq \lambda_n(\Omega',B).
\end{align}
Therefore, the disk $D_{j'}$ can be fully collapsed and a minimizer of \eqref{eq:min_problem_disjoint_union_of_disks} can be found among domains consisting of less than $k$ disks. 

If $\Omega =D_1 \cup ... \cup D_k$ consists of $k>n$ disks, then $n_{j'} = 0$ for some $j'$. Iteration of the above argument shows that a minimizer $\Omega_k^*$ of \eqref{eq:min_problem_disjoint_union_of_disks} for $k>n$ is found among unions of $n$ or less disks and $\Lambda_{n,k}(B) = \Lambda_{n,n}(B) $ for any $k>n$. The same is true for domains consisting of countably many disks where one can collapse all but $n$ disks.

Problem \eqref{eq:lambda_n_opt_disks_k_disks} can be solved numerically. It simplifies considerably if one defines the auxiliary problems
\begin{align} \label{eq:Lambda_nk-tilde}
\tilde{\Lambda}_{n,k} (B) = \min_{\substack{(n_j)_{j=1}^k\in \tilde{N}_{n,k}  \\ (t_j)_{j=1}^k \in  \tilde{T}_k  }} \max \left\lbrace \frac{\lambda_{n_1}(D,t_1 B)}{t_1}, ... , \frac{\lambda_{n_k}(D,t_k B)}{t_k} \right\rbrace.
\end{align}
where 
\begin{align}
\tilde{N}_{n,k} = \left\{ (n_j)_{j=1}^k \in \mathbb{N}^k \, : \,  \sum_{j=1}^k n_j = n \right\}, \qquad
\tilde{T}_k = \left\{ (t_j)_{j=1}^k \in (0,1)^k \, : \,  \sum_{j=1}^k t_j = 1 \right\}.
\end{align}
The minimum of \eqref{eq:Lambda_nk-tilde} is attained in $\tilde{T}_k$, because the expression $\max\{... \}$ in \eqref{eq:Lambda_nk-tilde} is a continuous function of the $t_j$ and diverges as one of the $t_j$ approaches zero. 

Note that 
\begin{align}
\lambda_n(\Omega, B) = \min_{\substack{(n_j)_{j=1}^k\in \tilde{N}_{n,k}  } } \max \left\lbrace \frac{\lambda_{n_1}(D,t_1 B)}{t_1}, ... , \frac{\lambda_{n_k}(D,t_k B)}{t_k} \right\rbrace
\end{align}
holds only if $n_j \geq 1$ for any $j$. The case of collapsing disks is hence artificially excluded in problem \eqref{eq:Lambda_nk-tilde}, but to obtain the solution of \eqref{eq:lambda_n_opt_disks_k_disks}, one simply needs to check against all auxiliary problems \eqref{eq:Lambda_nk-tilde} with a smaller number of disks. In other words,
\begin{align} \label{eq:Lambda_n-k_min_tilde_Lambda_n-k}
\Lambda_{n,k}(B) = \min \{ \tilde{\Lambda}_{n,1} (B), \tilde{\Lambda}_{n,2} (B), ... , \tilde{\Lambda}_{n,k} (B)\}.
\end{align}
It is thus sufficient to compute each $\tilde{\Lambda}_{n,k} (B)$. Some further simplifications can be made. If $k=1$, we have 
\begin{align}
\tilde{\Lambda}_{n,1} (B) = \Lambda_{n,1} (B) = \lambda_n(D,B).
\end{align}
If $k=n$, the auxiliary function $\tilde{\Lambda}_{n,k} (B)$ also has a simple form. The map $t \mapsto \lambda_1 (D, t B)/t$ is strictly decreasing with $t$, so the minimum of \eqref{eq:Lambda_nk-tilde} is attained for $t_1=...=t_n = 1/n$, i.e.\ $\Omega$ consisting of $n$ disks of equal size. Hence, 
\begin{align}
\tilde{\Lambda}_{n,n}(B) =  n \lambda_{1}\left(D,\frac{B}{n}\right).
\end{align}
If $ 1<k<n$, symmetry arguments let us further reduce the auxiliary problem \eqref{eq:Lambda_nk-tilde}. By relabeling, it is only necessary to consider $(n_j)_{j=1}^k \in \mathbb{N}^k$ that are partitions of $n$, ordered in descending order. For small $n$, this means we can replace $\tilde{N}_{n,k}$ in \eqref{eq:Lambda_nk-tilde} by $N_{n,k}'$ from the following table
\begin{center}
\begin{tabular}{c|c|l}
$n$ & $k$ & $N_{n,k}'$ \\ \hline
$3$ & $2$ & $\{(2,1) \}$ \\ \hline
$4$ & $2$ & $\{(3,1), (2,2) \}$ \\
$4$ & $3$ & $\{(2,1,1) \}$\\ \hline
$5$ & $2$ & $\{(4,1), (3,2) \}$ \\
$5$ & $3$ & $\{(3,1,1),(2,2,1) \}$\\
$5$ & $4$ & $\{(2,1,1,1) \}$
\end{tabular}
\qquad 
\begin{tabular}{c|c|l}
$n$ & $k$ & $N_{n,k}'$ \\ \hline
$6$ & $2$ & $\{(5,1), (4,2), (3,3) \}$ \\
$6$ & $3$ & $\{(4,1,1),(3,2,1),(2,2,2) \}$\\
$6$ & $4$ & $\{(3,1,1,1),(2,2,1,1) \}$\\
$6$ & $5$ & $\{(2,1,1,1,1) \}$\\ \hline
$...$ & $...$ & \quad  $...$\\ 
 &  & \\ 
\end{tabular}
\end{center}

\begin{figure}[t]
\begin{center}
\includegraphics[scale=0.5]{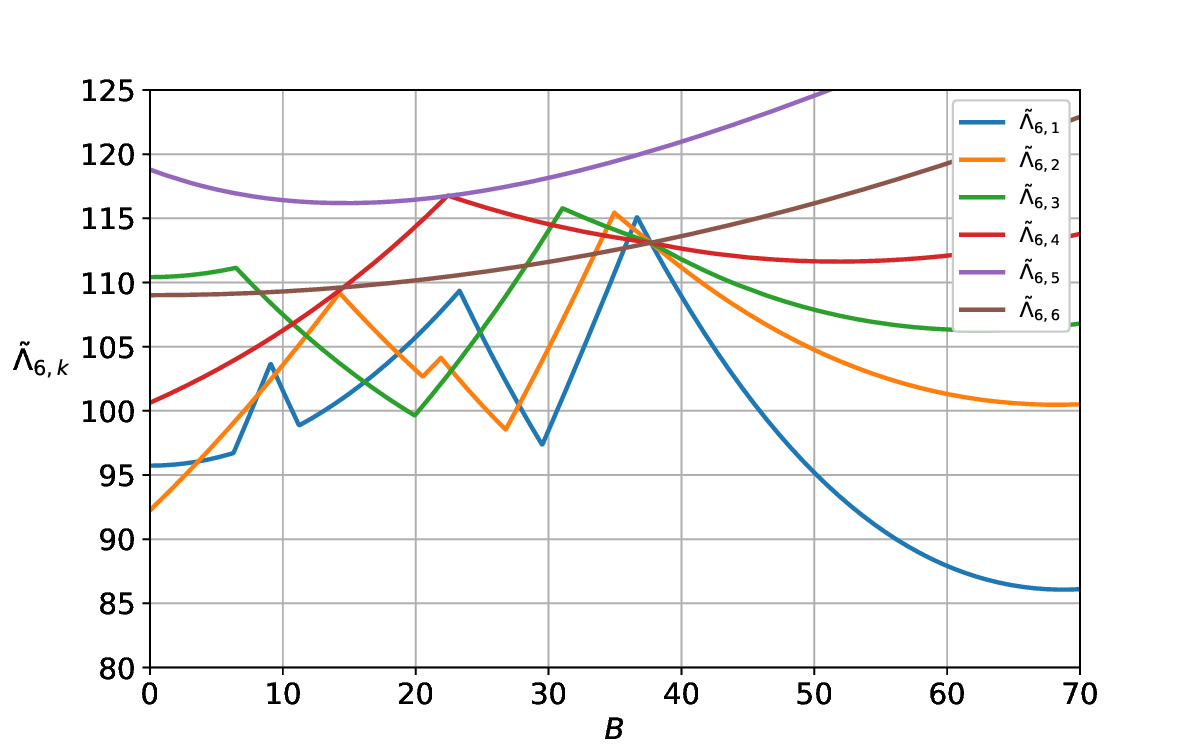}
\end{center}
\caption{The auxiliary functions $\tilde{\Lambda}_{n,k}(B)$ for $n=6$.}\label{fig:Lambda_n-k_aux}
\end{figure}

Figure \ref{fig:Lambda_n-k_aux} shows the auxiliary functions $\tilde{\Lambda}_{n,k}(B)$ for $1 \leq k \leq n=6$. By \eqref{eq:Lambda_n-k_min_tilde_Lambda_n-k}, the minimum of problem \eqref{eq:min_problem_disjoint_union_of_disks} is equal to the minimum of the first $k$ auxiliary curves. This minimum is easily obtained graphically. The picture also allows to identify the number of disks of minimizing domains. If $k'$ is such that $\tilde{\Lambda}_{n,k'}(B) = \min \{\tilde{\Lambda}_{n,1}(B),..., \tilde{\Lambda}_{n,k}(B) \} = \Lambda_{n,k}(B)$, then \eqref{eq:min_problem_disjoint_union_of_disks} has a minimizer consisting of $k'$ disjoint disks. In the figure, this means that for example at $B=20.0$, the minimum of \eqref{eq:min_problem_disjoint_union_of_disks} is attained by one disk if $k=1$, by two disjoint disks if $k=2$ and by three disks if $k\geq 3$. 

A few interesting phenomenons can be observed in the figure for the unrestricted ($k=+\infty$) minimization problem \eqref{eq:min_problem_disjoint_union_of_disks}: For a short interval just before $B = 12 \pi $, the unique minimizer is given by six disjoint disks (necessarily of equal size). At $B=12\pi $, we find five distinct minimizers, consisting of one, two, three, four and six disks. The minimization problem is highly degenerate at this field strength. For $B>12\pi $, the minimization problem \eqref{eq:min_problem_disjoint_union_of_disks} is always minimized by a single disk. 

It seems a similar structure for minimizers among disjoint unions of disks can be found for any $n\geq 2$. For $B$ just below $2\pi n$ (or $\phi$ just below $n$), the minimizer of \eqref{eq:min_problem_disjoint_union_of_disks} is given by $n$ equally sized disks. At $B=2\pi n$ (or $\phi = n$), the minimization problem \eqref{eq:min_problem_disjoint_union_of_disks} has multiple distinct minimizers. Both a single disk and $n$ equally sized disks are always minimizers at $B=2\pi n$. For $B>2\pi n$ (or $\phi> n$), the minimizer of problem \eqref{eq:min_problem_disjoint_union_of_disks} is unique and given by a single disk. We admit that we do not give rigorous proofs for these observations here. They would be concerned with analysis of the solutions of \eqref{eq:magnetic_disk_eigval_transcendental} and - in our view - primarily a technical issue.

\section*{Acknowledgements}

The author is thankful toward Timo Weidl who posed the problem to him. The author also wants to thank James B. Kennedy, Pedro R. S. Antunes and Pedro Freitas for valuable discussions and their hospitality during his visit at Universidade de Lisboa in Summer 2023.

\section*{Code and Data}

The shape optimization code and minimizer data used to generate the plots and figures in the results section is available at \url{https://github.com/matthias-baur/mssopython}.

\printbibliography

\end{document}